\documentclass[12pt,reqno]{amsproc}
\usepackage[margin=1.2in]{geometry}
\usepackage{amsmath, amsthm, amssymb}
\usepackage{times, esint,stackrel}
\usepackage{enumitem, graphicx}
\usepackage{color}
\usepackage{enumitem}   
\usepackage{hyperref}
\usepackage{etoolbox}
\usepackage{mathrsfs,mathabx}
\usepackage{subcaption,graphicx}

\usepackage{hyperref} 
\usepackage{mathrsfs}
\usepackage{framed} 
\usepackage{enumitem}

\usepackage{color} 
\usepackage{xcolor} 
\usepackage{hyperref}

\newcommand{\be}{\begin{equation}}
\newcommand{\ee}{\end{equation}}
\newcommand{\ba}{\begin{array}{l}}
\newcommand{\ea}{\end{array}}

\newcommand{\rmd}{{\rm d}}

\renewcommand{\AA}{\mathscr{A}}

\newcommand{\ve}{{\varepsilon}}

\renewcommand{\div}{{\mbox{div}\,}}

\def\wsc{\overset{\ast}{\rightharpoonup}}

\newcommand{\red}[1]{\textcolor{red}{#1}}
\newcommand{\green}[1]{\textcolor{green!80!blue}{#1}}

\usepackage{tikz}
\usetikzlibrary{patterns}
\usetikzlibrary{fpu}
\usetikzlibrary{plotmarks}
\usetikzlibrary{decorations.markings}
\usepackage{pgfplots}
\usetikzlibrary{intersections, pgfplots.fillbetween}

\definecolor{shadecolor}{gray}{.94}

\makeatletter
\newenvironment{myshade}{%
  \topsep4\p@\@plus4\p@\relax%
  \MakeFramed{\advance\hsize-\width \FrameRestore}}%
 {\par\unskip\endMakeFramed}%
\makeatother

\definecolor{shadecolor2}{rgb}{0.94, 0.9, 0.55}

\makeatletter
\newenvironment{myshade2}{%
  \topsep4\p@\@plus4\p@\relax%
  \MakeFramed{\advance\hsize-\width \FrameRestore}}%
 {\par\unskip\endMakeFramed}%
\makeatother

\definecolor{shadecolor3}{rgb}{1.0, 0.63, 0.48}

\makeatletter
\newenvironment{myshade3}{%
  \topsep4\p@\@plus4\p@\relax%
  \MakeFramed{\advance\hsize-\width \FrameRestore}}%
 {\par\unskip\endMakeFramed}%
\makeatother

\newtheorem{theorem}{Theorem}[section]
\newtheorem{corollary}[theorem]{Corollary}
\newtheorem{lemma}[theorem]{Lemma}

\newtheorem{conjecture}{Conjecture}
\newtheorem{question}{Open Question}
\newtheorem{problem}{Problem}

\theoremstyle{definition}
\newtheorem{definition}[theorem]{Definition}
\newtheorem{example}[theorem]{Example}
\newtheorem{remark}[theorem]{Remark}

\def\and{\ and }

\title[Singularity formation in the Euler equation]{Singularity formation in the incompressible Euler equation in finite and infinite time}
\author[T. D. Drivas]{Theodore D. Drivas}
\address{ Department of Mathematics, Stony Brook University,
Stony Brook, NY, 11794}
\email{tdrivas@math.stonybrook.edu}

\address{ School of Mathematics, Institute for Advanced Study, 1 Einstein Dr., Princeton, NJ 08540}
\email{tdrivas@ias.edu}

\author[T. M. Elgindi]{Tarek M. Elgindi}
\address{ Mathematics Department, Duke University,
 Durham, NC 27708, USA}
\email{tarek.elgindi@duke.edu}

\usepackage[T1]{fontenc}
\usepackage{lmodern}

\begin{document}

\maketitle

\emph{Dedicated to Prof. Dennis Sullivan on the occasion of his 80th birthday.}

\vspace{10mm}

\tableofcontents

\newpage

\section{Introduction}

In this paper,  some classical and recent results on the  Euler equations governing perfect (incompressible and inviscid) fluid motion are collected and reviewed, with some small novelties scattered throughout.  The perspective and emphasis will be given through the lens of infinite dimensional dynamical systems, and various open problems are listed and discussed. We begin in \S \ref{eds} by describing the geometric viewpoint of fluid flow as geodesic motion on the group of volume preserving diffeomorphisms.  Wellposedness results are then discussed for the Eulerian velocity field belonging to various function spaces. We continue in \S \ref{2dsing} by examining 2D fluid motion with an eye towards long-time persistent behaviors.
Subsequently, \S \ref{3dsec} examines mechanisms for finite-time singularity formation and reviews recent advances on blowup of classical solutions for 3D Euler.  Section \S \ref{perspect} gives an outlook on singularity formation in 3D and some open directions. Finally, in  \S \ref{infdsec} we discuss a class of global solutions in any dimension (those with constant pressure) and use them to give an example of finite time blowup from smooth initial data for the Euler equations in ``infinite" spatial dimensions.

\section{The perfect fluid  dynamical system}\label{eds}

Let $M\subset \mathbb{R}^d$ be a bounded simply connected open domain, possibly with boundary $\partial M$ having exterior unit normal $\hat{n}$. The Euler equations governing the velocity  $u(t,x) :[0,T]\times M \to \mathbb{R}^d$ of a fluid which is perfect and confined to $M$ read
  \begin{alignat}{2}\label{eeb}
\partial_t u + u \cdot \nabla u &= -\nabla p,   &\qquad \text{in} \quad M,\\\label{incompress}
\nabla \cdot u &=0,   &\qquad \text{in} \quad M,\\
u|_{t=0} &=u_0 ,   &\qquad\text{in} \quad M,\\
u\cdot \hat{n} &=0,   &\qquad\text{on}\   \partial M \label{eef}
\end{alignat}
where  $p(t,x) :[0,T]\times M \to \mathbb{R}$ is the hydrodynamic pressure which enforces incompressibility \eqref{incompress}. See Figure \ref{fig:domain}.
We shall sometimes refer to $M$  as the \emph{fluid vessel} and we will denote its closure by $\overline{M}$.

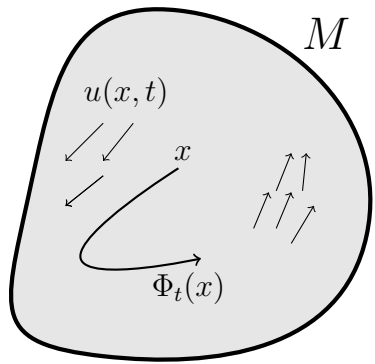
\begin{figure}[h!]
	\centering
	\begin{tikzpicture}[scale=1.0, every node/.style={transform shape}]
       \draw [name=A, ultra thick, black, fill= black!10!white] plot [smooth cycle, tension=1.2] coordinates {(0,-1.6) (3.5,0) (1,3) (-1,0.5)}; 
       \draw[thick, black, ->] plot [smooth, tension=1] coordinates {(1,0.9)  (-.3,-.3) (1.3,-.3)}; 
       	 \draw  (2.5,2.7) node[anchor=west] {\Large $M$};
	        	 \draw  (0.5,-0.7) node[anchor=west] { $\Phi_t(x)$};
		  \draw  (0.8,1.1) node[anchor=west] { $x$};
		   \draw  (-0.4,1.9) node[anchor=west] { $u(x,t)$}; 
		   \draw[thin, black, ->] (0,1.5)--(-.5,1);
		      \draw[thin, black, ->] (0.4,1.5)--(-.0,1);
		         \draw[thin, black, ->] (0,0.8)--(-.5,0.4);
		         
		               \draw[thin, black, ->] (0,0.8)--(-.5,0.4);
		         \draw[thin, black, ->] (2.3,.1)--(2.5,.58);    
		          \draw[thin, black, ->] (2.0,.1)--(2.2,.58);    %
		         \draw[thin, black, ->] (2.3,.6)--(2.5,1.1);%
		           \draw[thin, black, ->] (2.65,.6)--(2.7,1.1);   
		                \draw[thin, black, ->] (2.5,-0.1)--(2.8,.4);		       
\end{tikzpicture}
\caption{Fluid vessel, with the instantaneous velocity and a particle path until time $t$.}
	\label{fig:domain}
\end{figure}

\subsection{Principle of least action on configuration space}\label{lap}
The Euler equations have a beautiful geometric interpretation, due to V.I. Arnold.
Let $\mathscr{D}_\mu(M)$ denote the group of smooth volume-preserving  diffeomorphisms of $M$ which leave the boundary invariant.  This acts as the configuration space of the fluid, labelling particle positions.
Perfect fluid motion is governed by the ODE for $t \mapsto \Phi_t$ in the
space $\mathscr{D}_\mu(M)$:
\begin{equation}\label{ODiff} 
\begin{array}{ll}
\ddot \Phi_t(x) =-\nabla p\left(t,\Phi_t(x)\right) & \text{$(t,x)\in [0,T] \times M$,}\\
\Phi_0(x)=x &\text{$x\in M$,}\\
\Phi_t(\cdot)\in \mathscr{D}_\mu(M) &\text{$t\in [0,T]$,}
\end{array}
\end{equation}
 In these equations, the acceleration (the pressure gradient) acts in keeping with its role as a constraint to enforce incompressibility.  
The system \eqref{ODiff} can be thought of as arising from d'Alembert's principle of constrained motion, namely  $\Phi_t\in \mathscr{D}_\mu$, $\dot{\Phi}_t\in T_{\Phi_t} \mathscr{D}_\mu$ (the tangent space to $\mathscr{D}_\mu$ at $\Phi_t$ consisting of divergence-free vector fields tangent to $\partial M$) and $\ddot{\Phi}_t\in (T_{\Phi_t} \mathscr{D}_\mu)^\perp$ (the orthogonal complement to the tangent space consisting of gradient vector fields).  Said another way, the acceleration (the pressure gradient)  is orthogonal to the constraint (that $\Phi_t$ remain in $\mathscr{D}_\mu$).  The system \eqref{ODiff}  can be considered as the definition of perfect fluid motion.

 Arnold interpreted the ODE  \eqref{ODiff} for $t \mapsto \Phi_t$ as a \emph{geodesic} equation on $\mathscr{D}_\mu(M)$. 
To understand this view, fix $\gamma_1,\gamma_2\in\mathscr{D}_\mu(M)$.  Then, for any path $\gamma_\cdot :[t_1, t_2]\mapsto \mathscr{D}_\mu(M)$  satisfying $\gamma_{t_1} = \gamma_1$ and  $\gamma_{t_2} = \gamma_2$,  define the action functional
\be\label{action}
\AA[\gamma]_{t_1}^{t_2}:= \int_{t_1}^{t_2} \int_M \frac{1}{2} |\dot{\gamma}_t(x)|^2 \rmd x \rmd t.
\ee
We take variations of $\AA$ in path space as follows.  Consider a smooth one-parameter family of paths $\gamma^\ve_\cdot:[t_1,t_2]\mapsto  \mathscr{D}_\mu(M)$ for $\ve\in(-1,1)$ with fixed endpoints  $\gamma^\ve_{t_1} = \gamma_1$ and $\gamma^\ve_{t_2} = \gamma_2$ and where $\gamma^0= \gamma$.  Then we define the variation by 
\be
\delta \AA[\gamma]_{t_1}^{t_2}  := \frac{\rmd}{\rmd \ve} \AA[\gamma^\ve]_{t_1}^{t_2}\Big|_{\ve=0}.
\ee
In order to compute this object we need the variation of the path, defined by
\be\label{variation}
\delta \gamma_t(x) :=  \frac{\rmd}{\rmd \ve} \gamma^\ve_t(x)\Big|_{\ve=0}.
\ee
 Fixing $x\in M$, the variation  $\delta \gamma_\cdot(x):[t_1,t_2] \mapsto T_{\gamma_\cdot(x)} M$ defines an element of the tangent space of the manifold at $\gamma_\cdot(x)$ (formally $\delta \gamma_\cdot:[t_1,t_2] \mapsto T_{\gamma_\cdot} \mathscr{D}_\mu(M)$ defines an element of the tangent space of $ \mathscr{D}_\mu(M)$ along the path $\gamma$).   Composing with $\gamma^{-1}$, $\delta \gamma\circ \gamma^{-1} :[t_1,t_2] \mapsto T_{{\rm id}} \mathscr{D}_\mu(M)$ gives an element of the tangent space to the volume preserving diffeomorphism group at the identity.  Let $\mathfrak{X}_\mu(M)$ be the space of smooth divergence-free vector fields over $M$ which are tangent to the boundary. The tangent space $T_{{\rm id}} \mathscr{D}_\mu(M)$ can be identified with $\mathfrak{X}_\mu(M)$.  For our discussion, we require only that for any variation defined by pinned paths as above, it holds
 \be\label{eqntvec}
\delta \gamma_t( \gamma_t^{-1}(x))= v (t,x),
\ee
 for some  $v:[t_1,t_2]\mapsto \mathfrak{X}_\mu(M)$ with $v(t_1)= v(t_2)=0$, and vice versa.  The proof is elementary as everything is taken to be smooth:

\begin{myshade}
  \vspace{-1mm}
\begin{lemma}\label{tangentspace}
Fix $\gamma_\cdot:[t_1,t_2]\mapsto  \mathscr{D}_\mu(M)$. The following two statements hold
\begin{enumerate}
\item
Fix $v:[t_1,t_2]\mapsto \mathfrak{X}_\mu(M)$ with $v(t_1)= v(t_2)=0$.  There is a family $\gamma^\ve_\cdot:[t_1,t_2]\mapsto  \mathscr{D}_\mu(M)$ for $\ve\in(-1,1)$ with  $\gamma^\ve_{t_1} = \gamma_1$, $\gamma^\ve_{t_2} = \gamma_2$ and $\gamma^0= \gamma$ such that  \eqref{eqntvec} holds.
\item Let $\gamma^\ve_\cdot:[t_1,t_2]\mapsto  \mathscr{D}_\mu(M)$ for $\ve\in(-1,1)$ be  paths with  $\gamma^\ve_{t_1} = \gamma_1$, $\gamma^\ve_{t_2} = \gamma_2$ and $\gamma^0= \gamma$.   There exists $v:[t_1,t_2]\mapsto \mathfrak{X}_\mu(M)$ with $v(t_1)= v(t_2)=0$  such that  \eqref{eqntvec} holds.
\end{enumerate}
\end{lemma}
\end{myshade}

\begin{proof}
To establish the first direction, define the family $\gamma^\ve_\cdot:[t_1,t_2]\mapsto  \mathscr{D}_\mu(M)$  by
\be\label{eqnxieps}
\frac{\rmd}{\rmd \ve} \gamma^\ve_t(x) = v(t,  \gamma^\ve_t (x)), \qquad \gamma^0_t(x)=\gamma_t(x).
\ee
Since $ v\in  \mathfrak{X}_\mu(M)$, by Liouville's theorem it follows that $\det(\nabla  \gamma^\ve_t(x) )=1$ and thus  $ \gamma^\ve\in\mathscr{D}_\mu(M)$ for all $\ve$. Note that $\frac{\rmd}{\rmd \ve} \gamma^\ve_t\big|_{t_1}=\frac{\rmd}{\rmd \ve} \gamma^\ve_t\big|_{t_2}=0$ since $v$ vanishes at those times so that $\gamma^\ve_{t_1} = \gamma_1$, $\gamma^\ve_{t_2} = \gamma_2$.  According to definition \eqref{variation}, it follows \eqref{eqntvec} holds.

In the other direction, define
\be
v^\ve(t,x) := \left(\tfrac{\rmd}{\rmd \ve} \gamma^\ve_t\right)({(\gamma^\ve_t)}^{-1}(x)),
\ee
so that $\tfrac{\rmd}{\rmd \ve} \gamma^\ve_t(x)= v^\ve(t,\gamma^\ve_t(x))$.  Since $\gamma^\ve_t$ preserves volume, again by Liouville's theorem we have that  $v^\ve:[t_1,t_2] \mapsto   \mathfrak{X}_\mu(M)$ for all $\ve \in[0,1)$. Moreover $v^\ve(t_1)=v^\ve(t_2) =0$ since  $\tfrac{\rmd}{\rmd \ve} \gamma^\ve_{t_1}=\tfrac{\rmd}{\rmd \ve} \gamma^\ve_{t_2}=0$.  With $v:=v^0$, the claim follows.
\end{proof}
With this in hand, we arrive at the formal variational principle:

\begin{myshade}
  \vspace{-1mm}
\begin{theorem}[Action principle for perfect fluid motion]\label{theoremaction}
Let $\Phi_{t_1}, \Phi_{t_2}\in \mathscr{D}_\mu(M)$ be configurations of the fluid at times $t_1$ and $t_2$ with $t_1<t_2$. A trajectory $\Phi_t :[t_1, t_2]\mapsto \mathscr{D}_\mu(M)$ is an perfect fluid flow, i.e. $u:=\dot{\Phi}_t\circ \Phi_t^{-1}$ satisfies equations \eqref{eeb}--\eqref{eef}, if and only if $\Phi$ is a critical point of the action $\AA$, i.e.
\be
\delta \AA[\Phi]_{t_1}^{t_2} = 0, \qquad\text{with} \qquad  \delta \Phi_{t_1} = 0 , \ \  \delta \Phi_{t_2}=0.
\ee
\end{theorem}
\end{myshade}

\begin{proof}
Compute the first variation 
\begin{align*}
\delta \AA[\Phi]_{t_1}^{t_2} &=  \int_{t_1}^{t_2} \int_M \dot{\Phi}_t(x)\cdot \delta \dot{\Phi}_t(x) \rmd x \rmd t\\
&=  \int_{t_1}^{t_2} \int_M  \ddot{\Phi}_t(x)\cdot \delta  {\Phi}_t(x) \rmd x \rmd t  \qquad \text{(using that $ \delta \Phi_{t_1} =\delta \Phi_{t_2}=0$)}.
\end{align*}
According to Lemma \ref{tangentspace},  $ \delta{\Phi}_t(\Phi_t^{-1}(x))=v(x,t)$ for some  divergence-free velocity field $v$ which satisfies $v\cdot\hat{n}|_{\partial M}=0$ and   $v(t_1)= v(t_2)=0$. 
This in hand, using that $\Phi_t^{-1}$ preserves volume and $\Phi_t(M)=M$ we write the variation of the action as
  \begin{align}\label{variationact}
\delta \AA[\Phi]_{t_1}^{t_2} &=  \int_{t_1}^{t_2} \int_M  \ddot{\Phi}_t(\Phi_t^{-1}(x))\cdot v(x,t)\rmd x \rmd t.
\end{align}
Assume first that  $\delta \AA[\Phi]_{t_1}^{t_2}=0$, namely the action is stationary on $\Phi$ under any variation. By Lemma \ref{tangentspace}(1), the object \eqref{variationact} must vanish in particular for vector fields of the form $v(x,t) = f(t) \psi(x)$ for $f\in C^\infty_0(t_1,t_2)$ and $\psi\in  \mathfrak{X}_\mu(M)$.  
Thus
  \begin{align}\label{variationact2}
0 &=  \int_{t_1}^{t_2} f(t) g(t) \rmd t, \qquad g(t):=  \int_M  \ddot{\Phi}_t(\Phi_t^{-1}(x))\cdot \psi(x)\rmd x 
\end{align}
for all $ \psi\in\mathfrak{X}_\mu(M), f\in C^\infty_0(t_1,t_2)$. Since $g$ is continuous in time, we may take
$f$ to approximate $g$ on $[t_1,t_2]$ to conclude that $g(t) =0$ for each $t\in [t_1,t_2]$ (the fundamental lemma of calculus of variations).  We deduce for each $t\in [t_1,t_2]$ that
  \begin{align}
0= \int_M  \ddot{\Phi}_t(\Phi_t^{-1}(x))\cdot \psi(x)\rmd x \rmd t, \qquad \forall \psi\in\mathfrak{X}_\mu(M).\end{align}
The arbitrariness of $\psi\in\mathfrak{X}_\mu(M)$ together with the Hodge decomposition allows us to conclude the existence of $p(t,x):  [t_1,t_2]\times M\to \mathbb{R}$ such that 
\be\label{EE}
\ddot{\Phi}_t(x) = -\nabla p(t,\Phi_t(x)), \qquad  \forall t\in [t_1,t_2].
\ee
Since $\ddot{\Phi}_t(x) = (\partial_t u + u \cdot \nabla u) ({\Phi}_t(x) ,t)$ where $u=\dot{\Phi}_t\circ \Phi_t^{-1}$, we see that \eqref{EE} implies that $u$ solves Euler.  Contrariwise, if $u$ solves the Euler equations then \eqref{EE} holds with $p$ as the pressure field so that $\delta \AA[\Phi]_{t_1}^{t_2}=0$ by \eqref{variationact} and $v\in  \mathfrak{X}_\mu(M)$.  
\end{proof}

Next we show that for short times the Euler flow minimizes the action. This fact was pointed out by Ebin and Marsden \cite[Section 9]{ebinmarsden}.  The following version is due to  Y. Brenier in \cite[Section 5]{brenier1} or \cite[Proposition 3.2]{brenier4}.

\begin{myshade}
  \vspace{-1mm}
\begin{theorem}[Perfect fluid flow minimizes the action for short times]\label{theoremgeo}
Let $u\in C^1([0,T]\times \overline{M})$, $p\in C( [0,T]; C_x^2(\overline{M}))$.   Suppose that $T>0$ is such that
\be\label{condition}
T^2\leq \pi^2/K
\ee
with $K:= \sup_{t\in[0,T]}\sup_{x\in \overline{M}} \sup_{|y|=1} {y} \cdot \nabla^2 p(x,t) \cdot y$.
If $(u,p)$ is a solution of the Euler equations \eqref{eeb}--\eqref{eef} and $\dot{\Phi}_t = u(\Phi_t,t)$ with $\Phi_0= {\rm id}$, then
\be\label{ineq}
\AA[\Phi]_{0}^{T} \leq \AA[\gamma]_{0}^{T}
\ee
among all $\gamma_\cdot:[0,T]\mapsto  \mathscr{D}_\mu(M)$ with $\gamma_{0} = {\rm id}$ and $\gamma_{T} = \Phi_{T}$. If $T^2< \pi^2/K$, equality holds if and only if $\gamma=\Phi$.
\end{theorem}
\end{myshade}

From a geometric standpoint, the pressure Hessian is the second fundamental form -- encoding how the submanifold of volume preserving diffeomorphisms $\mathscr{D}_\mu(M)$ sits inside the ambient group of all diffeomorphisms $\mathscr{D}(M)$ \cite{mis1}.

\begin{proof}
Since $(u,p)$ is an Euler solution, we have $\ddot{\Phi}_t(x) = -\nabla p(\Phi_t(x))$. Note first
\be
\AA[\Phi]_{0}^{T} =\AA[\gamma]_{0}^{T}+   \int_0^T \int_M\dot{\Phi}_t(x)\cdot(\dot{\Phi}_t(x)- \dot{\gamma}_t(x))\rmd x \rmd t-\AA[\Phi-\gamma]_{0}^{T}.
\ee
The last two terms are negative for short time. Indeed, 
by Poincar\'{e}'s inequality,\footnote{Indeed, consider any absolutely continuous curve $f:[0,T]\mapsto \mathbb{R}^d$ with $f(0)=f(T) = 0$ and with $\rmd f(t)/\rmd t \in L^2([0,T])$.  Using the Fourier series representation of $f$ together with Plancherel's theorem we find immediately that $\|f(\cdot)\|_{L^2([0,T])}^2\leq  \frac{T^2}{\pi^2} \|\frac{\rmd}{\rmd t}f(\cdot)\|_{L^2([0,T])}^2$.   }
\be\nonumber
\AA[\Phi-\gamma]_{0}^{T}=\frac{1}{2} \int_0^T \int_M |\dot{\Phi}_t(x)- \dot{\gamma}_t(x)|^2 \rmd x \rmd t \geq \frac{\pi^2}{2T^2}  \int_0^T \int_M | {\Phi}_t(x)- {\gamma}_t(x)|^2 \rmd x \rmd t. 
\ee
On the other hand, since $\Phi_{0}(x)= \gamma_{0}(x)$ and $\Phi_{T}(x)= \gamma_{T}(x)$, we have 
\begin{align} \nonumber
 \int_0^T \int_M\dot{\Phi}_t(x)\cdot(\dot{\Phi}_t(x)- \dot{\gamma}_t(x))\rmd x \rmd t&=
\int_0^T \int_M \ddot{\Phi}_t(x)\cdot( \Phi_t(x)- \gamma_t(x))\rmd x \rmd t \\
&= -\int_0^T \int_M\nabla p(\Phi_t(x))\cdot( \Phi_t(x)- \gamma_t(x))\rmd x \rmd t. \nonumber
\end{align}
Since the pressure is twice differentiable, by Taylor's theorem we have 
\begin{align*}
p(\gamma_t(x))&= p(\Phi_t(x)) +\nabla p(\Phi_t(x))\cdot( \Phi_t(x)- \gamma_t(x))\\
&\qquad + \frac{1}{2}( \Phi_t(x)- \gamma_t(x))\cdot \nabla^2 p(Y_t(x))\cdot( \Phi_t(x)- \gamma_t(x)),
\end{align*}
where $Y_t(x)$ is on the chord connecting $\gamma_t(x)$ to $\Phi_t(x)$. Upon integrating, using the fact that $\gamma_t$ and $\Phi_t$ preserve volume, we obtain
\begin{align*}\nonumber
 -\int_0^T \int_M\nabla &p(\Phi_t(x))\cdot( \Phi_t(x)- \gamma_t(x))\rmd x \rmd t \\
 &=  \frac{1}{2}\int_0^T \int_M\ ( \Phi_t(x)- \gamma_t(x))\cdot \nabla^2 p(Y_t(x))\cdot( \Phi_t(x)- \gamma_t(x))\rmd x \rmd t.
\end{align*}
It follows that 
\be
\left| \int_0^T \int_M\dot{\Phi}_t(x)\cdot(\dot{\Phi}_t(x)- \dot{\gamma}_t(x))\rmd x \rmd t\right|\leq \frac{K}{2}  \int_0^T \int_M | {\Phi}_t(x)- {\gamma}_t(x)|^2 \rmd x \rmd t.
\ee
Thus we obtain the upper bound 
\begin{align}\label{finalineq}
\AA[\Phi]_{0}^{T} \leq \AA[\gamma]_{0}^{T} + \frac{1}{2}\left(K- \frac{\pi^2}{T^2}\right)  \int_0^T \int_M | {\Phi}_t(x)- {\gamma}_t(x)|^2 \rmd x \rmd t,
\end{align}
whence for $T^2\leq \pi^2/K$ we have $\AA[\Phi]_{0}^{T} \leq \AA[\gamma]_{0}^{T}$ as claimed.  If equality holds, then from \eqref{finalineq} we deduce  $ \int_0^T \int_M | {\Phi}_t(x)- {\gamma}_t(x)|^2 \rmd x \rmd t=0$ so that $\gamma=\Phi$ as claimed.
\end{proof}

\begin{remark}[Failure to be a minimizer]\label{remfail}
The condition \eqref{condition} on the time  is sharp in the following senses.
Consider the two-dimensional example of solid body rotation, i.e. $u(x)= x^\perp$. This is an exact stationary solution of the Euler equations  having pressure $p(x)= \frac{1}{2} |x|^2$ on a disc domain. The corresponding flowmap is $\Phi_t(x) = \mathsf{R}_{ t \! \! \mod 2\pi} x$ where $\mathsf{R}_{\theta}$ denotes the (counterclockwise) rotation matrix by angle $\theta\in[0,2\pi)$ about 0.  Brenier \cite{brenier1} points out that at time $T=\pi$ (a half rotation of the disk), there fails to be a unique minimizer of the action.  Indeed, the action does not depend on whether the rotation is clockwise or counter clockwise, both of which are geodesics connecting these two states. Note that since $\nabla^2 p = I$ we have $K=1$.  Thus, at exactly this moment, $T= \pi$, the condition \eqref{condition} is violated illustrating its sharpness.   For $T>\pi$, there exists a shorter path (just rotate clockwise) and thus after this moment, the original fluid flow is not the minimum of the action any longer. We remark that solid body rotation has a \emph{cut point} at $T=\pi$ -- see \cite{DM21} for further discussion of this geometric notion.  It is also an example of \emph{isochronal} flow: one for which the Lagrangian flowmap is time periodic (see Definition \ref{isodef} below).  Geometrically, this corresponds to $\Phi$ being a \emph{closed geodesic} in $\mathscr{D}_\mu(M)$ \cite{serre93}. These flows will be discussed further in \S \ref{strifesec}.
\end{remark}

\begin{remark}[Euler as geodesic motion on $\mathscr{D}_\mu(M)$]  We now describe V.I. Arnold's geometric picture in greater detail.
Formally, one can view the space $\mathscr{D}_\mu(M)$  as an infinite-dimensional manifold with the
metric inherited from the embedding in $L^2(M;\mathbb{R}^d)$, and with tangent
space made by the divergence-free vector fields tangent to the boundary of $M$.  We can define the length of a path $\gamma_\cdot:[t_1,t_2]\mapsto  \mathscr{D}_\mu(M)$  by the expression
\be
\mathscr{L}[\gamma]_{t_1}^{t_2}:= \int_{t_1}^{t_2} \|\dot{\gamma}_t(\cdot)\|_{L^2(M)} \rmd t.
\ee
We formally define the \emph{geodesic distance} connecting two states $ \gamma_1,\gamma_2\in \mathscr{D}_\mu(M)$ by
\be
{\rm dist}_{\mathscr{D}_\mu(M)}(\gamma_0,\gamma_1) = \inf_{\substack{ \gamma_\cdot:[0,1]\mapsto  \mathscr{D}_\mu(M)\\
    \gamma(0)=\gamma_0, \ \gamma(1)=\gamma_1 }}\mathscr{L}[\gamma]_{0}^{1}.
    \ee
A geodesic curve $\Phi_\cdot:[t_1,t_2]\mapsto  \mathscr{D}_\mu(M)$  is defined to be one so that for all $t_1\in\mathbb{R}$ there exists a $\tau>0$ such that if $t_1< t_2< t_1+ \tau$ then
\be
{\rm dist}_{\mathscr{D}_\mu(M)}(\Phi(t_1) ,\Phi(t_2))  =  \mathscr{L}[\Phi]_{t_1}^{t_2}.
\ee
If additionally the parametrization by $t$ is chosen so that $\|\dot{\Phi}_t(\cdot)\|_{L^2(M)} ={\rm(const.)}$, then $\Phi$ minimizes the action  \eqref{action} among all smooth paths connecting $\Phi(t_1)$ and $\Phi(t_2)$. Indeed by Schwarz's inequality, we have $\mathscr{A}[\gamma]_{t_1}^{t_2}\geq {(\mathscr{L}[\gamma]_{t_1}^{t_2})^2}/{2(t_2-t_2)}$ with equality if and only if $\|\dot{\gamma}_t(\cdot)\|_{L^2(M)} ={\rm(const.)}$.  In general, according to Theorem \ref{theoremaction} perfect fluid motion is a critical point of both functionals  $\mathscr{A}[\gamma]_{t_1}^{t_2}$ and $\mathscr{L}[\gamma]_{t_1}^{t_2}$. It is, in fact, a geodesic according to Theorem \ref{theoremgeo}, although as we say in Remark \ref{remfail} it need not be the curve of minimal length for long times.

\begin{figure}[h!]
\centering
\includegraphics[scale=.5]{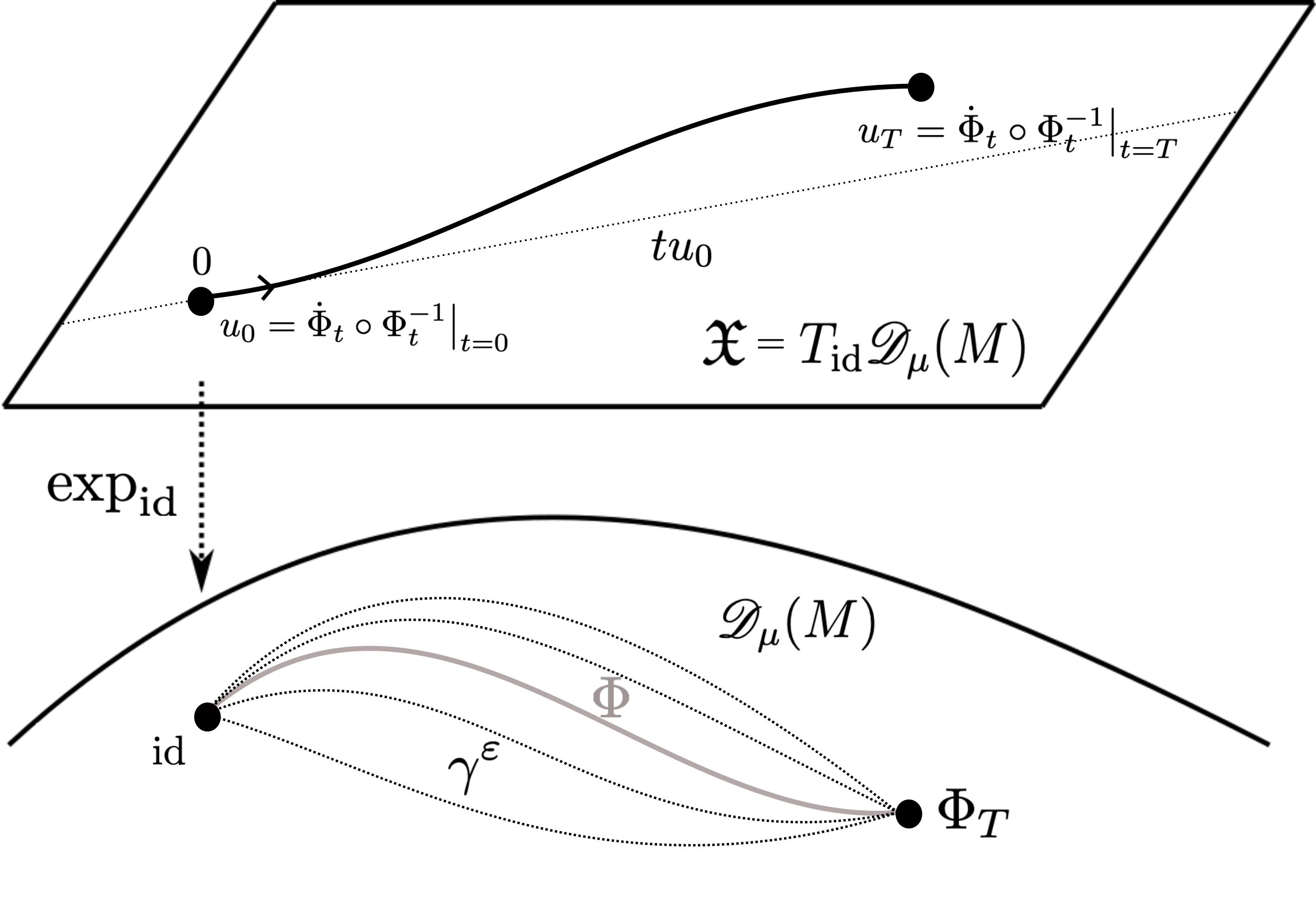}
\caption{Depiction of the geometry of fluid motion.}
\label{cubes}
\end{figure}  

The above discussion is somewhat formal in that it ignores issues of regularity required for precise definitions of the variations. To make things more precise, one may consider the group $\mathscr{D}_\mu^s(M)$ for $s>d/2+1$, which is a submanifold of all $H^s$ diffeomorphisms of $M$, see \cite{ebinmarsden}.
The $L^2$ exponential map is defined as the solution map of 
the geodesic equations: 
it maps lines through the origin in the tangent space at a given diffeomorphism
onto geodesics in the diffeomorphism group. 
More precisely, at the identity we set 
\begin{equation} \label{eq:exp} 
\mathrm{exp}_{\rm id} : T_{\rm id} \mathscr{D}_\mu^s \to \mathscr{D}_\mu^s, 
\qquad 
\mathrm{exp}_{\rm id}{tu_0} = \Phi_t, 
\quad 
t \in \mathbb{R} 
\end{equation} 
where $\Phi_t$ is the unique $L^2$ geodesic (at least for short times) starting from ${\rm id}$ with velocity $u_0$.  We note that in any spatial dimension, $\mathrm{exp}_{\rm id}$ is a local diffeomorphism near ${\rm id}$ and in two dimensions, it is defined on the whole tangent space.  The study of analytical properties of this map is the subject of the classical work of Ebin and Marsden \cite{ebinmarsden}.  This framework opened up the possibility to ask purely geometric questions about fluid motion, such as those concerning the existence of conjugate and cut points. See \cite{mis2,shnir2,DM21}. 
\end{remark}

\begin{remark}[Euler--Arnold equations]
Arnold's geometric description of the motion of a physical system as geodesic flow on a Lie group with respect to some metric extends widely. These occur either by changing the group (e.g. instead of $\mathscr{D}_\mu(M)$, consider $\mathscr{D}(M)$, $SO(n)$, the Virasoro group, etc) or changing the metric (e.g. instead of $L^2$, consider $H^1$, $\dot{H}^1$, $H^{-1}$, etc). Some of the systems described in this framework are the $n$-dimensional top \cite{MR13}, the Camassa--Holm equation \cite{SK99}, the Euler-$\alpha$ mean motion model \cite{HMR,Shk}, the Burgers equation and compressible fluids \cite{SP,KMM21}, the Hunter--Saxton equation \cite{JK,SP2}, among others.  See \cite{ak}.
One can also accommodate systems such as compressible and quantum fluids, by the inclusion of an appropriate potential
related to material properties of the system.  See the recent survey
\cite{KMM21}.
\end{remark}

\begin{remark}[Time analyticity of particle trajectories]
Shnirelman proved \cite{shn12} that $\exp_e$ defined by \eqref{eq:exp} is a real analytic map 
provided that the fluid domain is an analytic manifold.  
As a corollary Lagrangian trajectories are analytic functions of time 
at fixed labels \cite{Lich25}.  This fact was proved first by Serfati  \cite{ser} and, for streamlines of stationary states, by Nadirashvili \cite{N14}. The difference between Eulerian and Lagrangian analyticity was highlighted by Constantin-Kukuvika-Vicol \cite{ckv}, and these results have subsequently extended  by Constantin-Vicol-Wu \cite{cvw} to other systems. See also the discussion by Frisch and Zheligovsky \cite{frisch}.  It would be interesting to understand which Eulerian phase spaces have this property.  For example, are  particle trajectories  real analytic  for bounded vorticity solutions (the Yudovich phase space, see Theorem \ref{yudothm})?
\end{remark}

\begin{remark}[Two-point fluid problem]
The principle of least action suggests the following `infinite dimensional Dirichlet principle': given two isotopic configurations $\gamma_1$ and $\gamma_2\in \mathscr{D}_\mu$, construct a perfect fluid flow connecting them by identifying the shortest path between them in the diffeomorphism group (in the $L^2$ metric).  From the above discussion, if such a path exists it is automatically a perfect fluid flow.  This problem was first investigated by Shnirelman \cite{shnir1}, where he proved that if $d\geq 3$, this variational problem does not have minimizers for some pairs of configurations $\gamma_1$ and $\gamma_2$.
  In $d=2$, this is open even in the following relaxed form:

  \begin{myshade3}
  \vspace{-1mm}
\begin{question}[Shnirelman (1985), \cite{shnir1}]\label{quest1}
Let $M\subset \mathbb{R}^2$ be a domain with smooth boundary.   Does there exist a perfect fluid flow connecting any two isotopic states $\gamma_1,\gamma_2\in \mathscr{D}_\mu(M)$?
\end{question}
\end{myshade3}

See also \cite{K05}. In particular, it is not known whether the image of the space of incompressible vector fields in $ \mathscr{D}_\mu$ defined by the $L^2$ exponential map is the whole group $ \mathscr{D}_\mu$, i.e. that any such diffeomorphism could be realized as a time-1 value of some solution of the Euler-Lagrange equation. It is a question of accessibility of the entire configuration space by perfect fluid flows. If Question \ref{quest1} is answered in the affirmative, it would represent a hydrodynamical analogue of the Hopf-Rinow theorem for the group of diffeomorphisms. The difficulty is that this group is not locally compact.
 However, Misiolek,  and Preston proved that the $L^2$ exponential map is a covering space map on an open connected component 
$\mathcal{U} \subset \mathscr{D}_\mu^s$ of the identity whose $L^2$ diameter is infinite, 
cf. \cite{mispre}.  This is a consequence of  the fact that $\mathrm{exp}_{\rm id}$ is a nonlinear Fredholm map of index zero, 
see \cite{emp,shnir3}. 
An affirmative resolution of Open Question \ref{quest1} (conjectured by Shnirelman in \cite{shnir1}) would result from showing this connected component is the whole group.  We remark finally that a different but related question is that of finding minimizers (shortest paths) connecting the states $\gamma_1$ and $\gamma_2$.  The existence of conjugate and cut points along geodesics generated by simple steady solutions on $M=\mathbb{T}^2$ having streamfunction $\psi(x_1,x_2) = \sin(n x_1) \sin(m x_2)$ (Kolmogorov flows) \cite{mis2,DM21} indicate that the minimum may fail to be achieved at long times (see the numerical study \cite{L12}), as was conjectured in \cite{shnir0}.

In view of the fact that classical minimizers of the two-point problem need not always exist  \cite{shnir1}, Brenier introduced \emph{generalized flows}, which are a wider class of (stochastic) objects over which the variation problem is always solvable \cite{brenier1}. Shnirelman used this idea to show that any sufficiently long geodesic in $\mathscr{D}_\mu^s(M)$ 
will contain a local cut point if $\dim{M}=3$ (a point such that shorter paths can be chosen arbitrarily close to  the given geodesic in the manifold topology).
\end{remark}

\subsection{Cauchy wellposedness in phase space}

To understand the appropriate phase space for the Eulerian (for the velocity field rather than the flowmap) dynamics \eqref{eeb}--\eqref{eef}, a useful quantity to introduce is the vorticity  two-form  $\Omega:=\frac{1}{2}(\nabla u- \nabla u^t)$, which is the antisymmetric part of the velocity gradient tensor.   Recall that the Lie derivative $ \pounds_v$ along a vector field $v$ acting on a two-form $a$  reads 
$\pounds_v a = v\cdot \nabla a +a (\nabla v)^t +(\nabla v)a.
$
In any dimension, perfect fluid motion  \eqref{eeb}--\eqref{eef} is equivalently
 defined by the property that vorticity two-form is transported:
  \begin{alignat}{2}\label{eevb}
\partial_t\Omega + \pounds_u \Omega&=0 &\qquad \text{in} \quad  M,\\
\Omega|_{t=0} &=\Omega_0 ,   &\qquad\text{in} \quad  M,\label{eevb1}
\end{alignat}
where $u:=u_{\textsf{rot}} + u_{\textsf{H}}$ where $ u_{\textsf{H}}$ is the harmonic part\footnote{Recall the space of harmonic vector fields on a compact Riemannian manifold $M$ is isomorphic to the first cohomology group of $M$.  See e.g. \cite[Book 1, Chapter 5]{mt96}.  If  $M$ has trivial  first cohomology (e.g. $M$ is simply connected), then  there are no non-trivial harmonic vector fields, the curl is injective and the vorticity formulation \eqref{eevb}--\eqref{eevb1} with $u$ recovered by Biot-Savart is equivalent to the velocity formulation. If the first cohomology is non-trivial (e.g. $M=\mathbb{T}^d$) then the harmonic part of the solution must be kept track of separately according to equation \eqref{harmcomp}.
 } of the velocity field and  the non-harmonic part $u_{\textsf{rot}}$  is  recovered from $\Omega$ by the Biot-Savart law $K_ M$. More specifically, $u_{\textsf{rot}}=K_ M[\Omega] $ is defined as the unique  solution of the following elliptic system having trivial harmonic component\footnote{One could instead define the Biot-Savart operator to have exactly the harmonic component consistent with the Euler equations. As discussed here, this can be accomplished by solving the coupled system of \eqref{harmcomp} and \eqref{eevb}--\eqref{bsf} in tandem. However, there can be simpler prescriptions. For example, suppose $M \subset \mathbb{R}^2$ is a bounded planar domain with connected components of the boundary denoted by $\Gamma_0, \Gamma_1, \dots \Gamma_N$ with $\Gamma_0$ bordering the unbounded connected component of $\mathbb{R}^2\setminus M$.  Then the harmonic part of the velocity is fixed by demanding in addition to \eqref{bs1}--\eqref{bsf} that the circulations $\oint_{\Gamma_i}u \cdot \rmd \ell=\gamma_i$ be given.   Euler preserves these circulations by Kelvin's theorem since connected components of the boundary are invariant under the Lagrangian flow. Thus  $\gamma_i := \oint_{\Gamma_i}u_0 \cdot \rmd \ell$, closing the system without the need for an additional evolution equation.\label{harmfoot}} (see e.g. \cite{K02})
  \begin{alignat}{2}\label{bs1}
\Delta u &= 2 \nabla \cdot \Omega \qquad \  \text{in} \quad  M,\\
u \cdot \hat{n} &= 0 \qquad\qquad \ \ \text{on} \quad  \partial M,\\
\hat{\tau} \cdot (\partial_n u -   2\hat{n} \cdot \Omega) &= \hat{\tau} \cdot  \kappa \cdot u    \qquad\text{on} \quad  \partial M,\label{bsf}
\end{alignat}
for all $\{\hat{\tau}_i\}_{i=1}^{d-1}$  tangent vector fields and where $\kappa_{ij}=  \partial_i \hat{n}_j$ is the second fundamental form of the boundary. 
The solution to \eqref{bs1}--\eqref{bsf} defines   the Biot-Savart law $u=K_ M[\Omega]$.  On the other hand, if $\overline{M}$ is compact and $ \mathbb{P}_{\textsf{H}} $ denotes the orthogonal projection of $L^2$ onto the finite-dimensional space of harmonic vector fields on $M$ tangent to the boundary $\partial M$, then
 an evolution equation for the harmonic part of the velocity $u_{\textsf{H}}  = \mathbb{P}_{\textsf{H}} u $ is found by projecting the Euler equation  \eqref{eeb}--\eqref{eef} yielding
\be\label{harmcomp}
\partial_t u_{\textsf{H}} + \mathbb{P}_{\textsf{H}}[ \div  (u_{\textsf{rot}} + u_{\textsf{H}})\otimes  (u_{\textsf{rot}} + u_{\textsf{H}})] =0,
\ee
where, in deriving \eqref{harmcomp}, we used that gradient vector fields (the pressure forces) are orthogonal to harmonic vector fields which are tangent to the boundary.  Equation \eqref{harmcomp} is a closed equation for the harmonic component $u_{\textsf{H}}$ given the vorticity $\Omega$.  Thus, it can be evolved along side the vorticity evolution equation \eqref{eevb}--\eqref{eevb1} in order to recover the full Euler solution $u:=u_{\textsf{rot}} + u_{\textsf{H}}$. As such,  a solution to \eqref{eevb}--\eqref{harmcomp} generates a solution to \eqref{eeb}--\eqref{eef} and vice versa.  

\begin{remark}[Solutions with stationary vorticity]
Note that an immediate consequence of this discussion is that the space of harmonic vector fields is invariant under the Euler evolution and the dynamics restricted to this subspace becomes  finite dimensional. In fact every harmonic vector field is a stationary Euler solution as they are divergence free and satisfy  $u_{\textsf{H}} \cdot \nabla u_{\textsf{H}} = \nabla \frac{1}{2} |u_{\textsf{H}}|^2$ which follows from their vorticity being trivial.  This fact can also be used to eliminate the nonlinear term in $u_{\textsf{H}}$ from the equation \eqref{harmcomp}.  More generally, if the vorticity $\Omega$ is time independent (and thus $u_{\textsf{rot}}$),  then so is  $u_{\textsf{H}}$ provided that either $M$ has  trivial first cohomology (by Hodge theory), the dimension of the space spanned by the harmonic vector fields tangent to the boundary is unity (by energy conservation), or $M$ is a compact planar domain (by a structural cancellation).  
See \cite[Book 3, Chapter 11, Section 1]{mt96}. Outside of these settings, it would be interesting to study what kind of motion the projected dynamics \eqref{harmcomp} with steady vorticity can give in the space of harmonic vector fields.
\end{remark}

We remark that, in the case $M = \mathbb{T}^{d}$, there are $d$ independent harmonic vector fields $\{h_i\}_{i=1}^d$ which an be identified as the direction fields $h_i = e_i$.  Thus, the orthogonal projection $ \mathbb{P}_{\textsf{H}} $ is simply integration over the domain, i.e. $u_{\textsf{H}}= \mathbb{P}_{\textsf{H}} u := \sum_{i=1}^d (u, h_i)_{L^2} h_i =\int_{\mathbb{T}^d}u \rmd x$.  Since $u_{\textsf{rot}}, u_{\textsf{H}}$ are periodic, equation \eqref{harmcomp} shows that $u_{\textsf{H}}(t)=\int_{\mathbb{T}^d} u_0\rmd x$ is constant in time.  A similar remark holds on the periodic channel $M = \mathbb{T}^{d-1}\times [0,1]$ where the horizontal means are preserved.  As such, the harmonic-free part (the dynamical component) can be solved using the Euler equation written in a moving reference frame, i.e. by Galilean transformation.

Hereon, we shall denote by $S_t$ the vorticity solution map $\Omega(t) = S_t(\Omega_0)$. We now discuss the issue of local well-posedness of the Euler equation on a given function  space $X$.  We call $X$ a \emph{good phase space} for the Euler dynamics if given data in $X$, the equations admit unique solutions in $\Omega_0\in X$, at least for short time, i.e. the solution operator $S_t:X\ni \Omega_0 \mapsto \Omega(t)\in X$ is a well defined injective map for all $t$ sufficiently small.  Generally speaking,  this will be the case provided
\begin{enumerate}[label=(\alph*)]
\item $X$ is an algebra,
\item trajectories $\dot{\Phi}_t(x) = u( \Phi_t(x),t)$ are unique for all $x\in M$ provided $\nabla u \in X$,
\item $X$ is compatible with the Biot-Savart law $u=K_M[\Omega]$, i.e. $\|\nabla u\|_X\leq C \|\Omega\|_{X}$.
\end{enumerate}
Examples of such a phase space include $X = C^\alpha$ for $\alpha>0$ with $\alpha \notin \mathbb{N}$ and $X= W^{s,p}(M)$ for $s>d/p$ for $p>1$.  Generally, any space $X$ that embeds in the space of continuous functions will satisfy (a).  Particle trajectories are usually well defined at this level of regularity, e.g. (b) is satisfied. The H\"{o}lder spaces $C^k$ or Sobolev spaces  $W^{k,p}$ with $k\in \mathbb{N}$  and $p=1,\infty$ are  not good phase spaces for Euler since it is well known that the elliptic regularity estimate (c) can fail in those endpoint spaces and, as a result, the dynamics can be illposed \cite{BourgainLi,EM1,MY}.

We now recall the classical result on local wellposedness in H\"{o}lder spaces:

\begin{myshade}
  \vspace{-1mm}
\begin{theorem}[Local well-posedness for $d\geq 2$]\label{lwpthm}  Let $ M\subset\mathbb{R}^d$ be a bounded domain with smooth boundary $\partial M$.  Let $\alpha>0$ and $u_0\in C^{1,\alpha}(M)$ be divergence-free and tangent to $\partial M$.  For $T>0$ sufficiently small, there exists a unique  $u\in C^{1,\alpha} ((-T,T)\times  M)$  solving \eqref{eeb}--\eqref{eef} with vorticity $\Omega\in C^{\alpha} ((-T,T)\times  M)$.  If $[0, T)$ is the maximal interval of existence, then  
\be\label{bkm}
\int_0^T \| \Omega(t)\|_{L^\infty( M)} \rmd t = \infty.
\ee
\end{theorem}
\end{myshade}
Local existence is due  to Wolibner \cite{W33}, H\"{o}lder \cite{H33} and Ebin \& Marsden \cite{ebinmarsden}.
The continuation criterion in terms of vorticity \eqref{bkm} was proved by Beale-Kato-Majda \cite{bkm} (generalized to bounded domains and manifolds in \cite{f93,mt96}).
In two dimensions one can identify the vorticity $\Omega$ with a scalar field $\omega= \nabla^\perp \cdot u$ where $\nabla^\perp = (-\partial_2,\partial_1)$ and in three dimensions with a vector field $\omega= \nabla \times u$.  These are transported by $u$ as a scalar field and vector field respectively
  \begin{alignat}{2}\label{2dvort}
d=2: \qquad \partial_t\omega + u\cdot \nabla \omega &=0 \\ \label{3dvort}
d=3: \qquad  \partial_t\omega + u\cdot \nabla \omega &=\omega \cdot \nabla u.
\end{alignat}
Note that, in lower regularity, there has been an explosion of work in recent years related to problems in turbulence.  See the recent reviews \cite{bv,ds}.

This difference in character of the vorticity leads to extreme differences in behaviors of two and three dimensional fluids.  In two dimensions, conservation of vorticity magnitudes results in an inverse cascade where large-scale velocity structures emerge over time through a process of vortex mergers \cite{Batchelor,Kraichnan}.  In three-dimensions, the vorticity magnitude is not invariant due to stretching by the (symmetric part of the) velocity gradient tensor. This results in direct energy cascade \cite{Kolmogorov} and a fine-scale filamentary structure of the vorticity field where large values are spatiotemporally sparse but highly relevant dynamically.  See Figure \ref{vorticityfig} for a visualization of time snapshots from direct numerical simulations.

\begin{figure}[htb]\centering \label{vorticityfig}
    \includegraphics[width=.45\columnwidth]{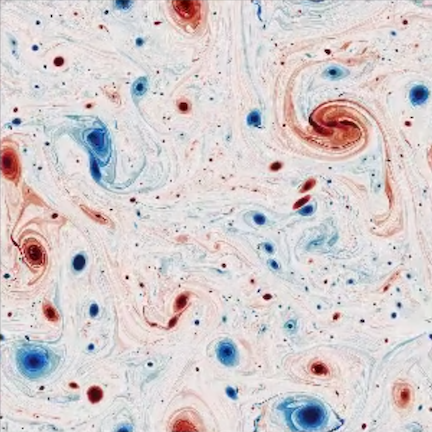} \qquad
    \includegraphics[width=.45\columnwidth]{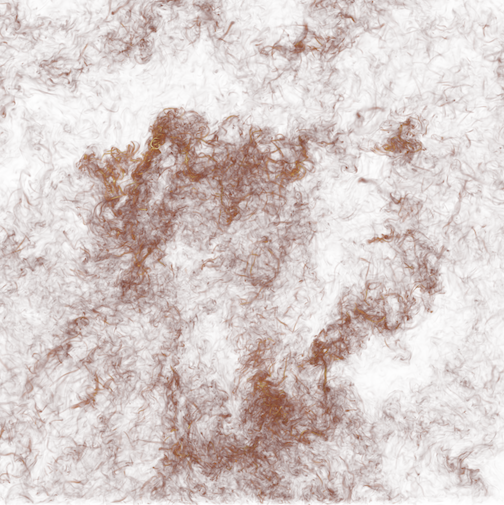} 
  \caption{ Typical instances of vorticity in two-dimensions (left panel, visualized through its signed magnitude) and three-dimensions (right panel, thresholded at large values).}
\label{figvort}
\end{figure}

In two dimensions, a consequence of the transport structure \eqref{2dvort} is that the solution has the Lagrangian representation formula 
\be\label{trans2d}
\omega(t) = \omega_0\circ \Phi_t^{-1}
\ee
where  $\dot{\Phi}_t = u(\Phi_t,t)$ is the Lagrangian flowmap labelled at $\Phi_0={\rm id}$.   A consequence of this fact together with Theorem \ref{lwpthm} is the global wellposedness in 2D \cite{W33,H33}:

\begin{myshade}
  \vspace{-1mm}
\begin{theorem}[Global well-posedness in 2D]\label{gwpthm}  Let $ M\subset\mathbb{R}^2$ be a bounded domain with smooth boundary $\partial M$.  Let $\alpha>0$  and $u_0\in C^{1,\alpha}( M)$ be divergence-free and tangent to $\partial M$.  Then there exists a unique  $u\in C^{1,\alpha} ((-\infty,\infty)\times  M)$  solving \eqref{eeb}--\eqref{eef} with vorticity $\omega\in C^{\alpha} ((-\infty,\infty)\times  M)$. Moreover, there exists a constant $c:=c(M)$ so that the following bound holds for all $t\in \mathbb{R}$
\be\label{dexp}
\frac{\| \omega(t)\|_{C^{\alpha}( M)}}{\|\omega_0\|_{L^\infty(M)} }  \leq \left(\frac{\| \omega_0\|_{C^{\alpha}(M)}}{\|\omega_0\|_{L^\infty( M)} }\right)^{\exp\left({c\|\omega_0\|_{L^\infty( M)}  |t|/\alpha}\right)}.
\ee
\end{theorem}
\end{myshade}
We remark that the time dependence in \eqref{dexp} is known to be sharp due to work of Kiselev-Šverák, see 
Theorem \ref{KSthm} herein.

Thus, the 2D Euler equations  \eqref{eeb}--\eqref{eef} form an infinite dimensional dynamical system  $S_t: C^{\alpha}\to C^{\alpha}$ for all time $t\in \mathbb{R}$ on the space of $C^{\alpha}$ vorticity fields. 
We are interested in the long time behavior of this dynamical system.  The formula \eqref{trans2d} shows that the vorticity at any later time is just some area preserving rearrangement of its initial conditions.  As discussed in \S \ref{lap}, any perfect fluid flow can be thought of as a path in the group of such transformations $\mathscr{D}_\mu(M)$ (the configuration space of the flow).  As such, the diffeomorphism could conceivably become very complex: filamenting, stretching, kneading and folding the initial vorticity isolines as time progresses. This action causes norms which measure regularity of the vorticity to grow over time.  There are some limitations;    the bound \eqref{dexp} shows that the orbit $\omega(t)= S_t(\omega_0)$ may exit the phase space $C^{\alpha}$ at infinite time at a rate at most double exponential.  There is great interest in finding examples of such exits with various rates -- each rate corresponds to some physical process of generation of small scales.  We will discuss some examples in detail along with some speculations about generic behavior in  \S \ref{2dsing}.  

The strongest known control uniform in time on the vorticity is in $L^\infty$.  Using this fact, in the celebrated paper \cite{Y63}, Yudovich proved that the 2D Euler equations are globally wellposed on this space $S_t: L^\infty\to L^\infty$ for all $t\in \mathbb{R}$:

\begin{myshade}
  \vspace{-1mm}
\begin{theorem}[Global bounded vorticity solutions]\label{yudothm}  Let $ M\subset\mathbb{R}^2$ be a bounded domain with smooth boundary $\partial M$.  Let $\omega_0\in L^\infty( M)$.  There exists a unique weak Euler solution of class $\omega\in L^\infty((-\infty,\infty)\times  M)$ with the following properties
\begin{enumerate}[label=(\alph*)]
\item $\omega(t)$ takes the form \eqref{trans2d} for a unique H\"{o}lder continuous flow of homeomorphisms $\Phi_t\in C^{\alpha(t)}(M)$ for $C:=C(M)$ and $\alpha(t):=\exp(- C |t| \|\omega_0\|_{L^\infty})$. Consequently, $\|\omega(t)\|_{L^\infty(M)} =\|\omega_0\|_{L^\infty(M)}  $ for all $t\in \mathbb{R}$.
\item  $\omega \in C(\mathbb{R}; L^p(M))\cap C_w(\mathbb{R};L^\infty(M))$ for all $p\in [1,\infty)$, where $C_w(\mathbb{R};L^\infty(M))$ denotes the space of functions which are continuous in time with values in the weak-$*$ topology of $L^\infty(M)$.  
\item The solution map $S_t$ is  continuous in the weak-$*$ topology of $L^\infty(M)$, i.e. if $\omega_0^n \wsc\omega_0$ in $L^\infty(M)$, then $S_t(\omega_0^n) \wsc S_t(\omega_0)$ in $L^\infty(M)$ for all $t\in \mathbb{R}$.
\end{enumerate}
\end{theorem}
\end{myshade}

Property (a) that the flowmap is H\"older continuous follows from the fact that velocity fields with bounded vorticity enjoy Log-Lipschitz regularity.\footnote{For Yudovich solutions, the velocity is thus Osgood continuous and particle trajectories are unique.  This is at the heart of why there is a good uniqueness theory for such solutions in two dimensions - the mathematics of which has strong parallels with the measurable Riemann mapping theorem. See the discourse of D. Sullivan \cite{S15}.}
Property (b) shows that  the solution map $S_t: L^\infty\to L^\infty$ is well defined for every $t\in \mathbb{R}$. 
Regarding Property (c),  it is also known that the solution map is continuous in $L^p(M)$ on bounded sets in $L^\infty(M)$ \cite{cde20}.   As stated, Property (c) (see \cite{sv11, Ng21,LML}) gives continuity of the solution map in the weak-$*$ topology and it is particularly important in connection with long time behavior.  Specifically,
let $X_*=X_*(M)$ be the unit ball in the $L^\infty$ phase space 
\be\label{Xstar}
X_*:= \{ \omega\in L^\infty(M) \ : \ \|\omega\|_{L^\infty} \leq 1\}
\ee
endowed with the weak-$*$ topology.\footnote{We recall that the standard way to metrize $X_*$  which gives the weak-$*$ topology.  Consider a countable dense set of integrable functions $\{\varphi_j\}_{j\in \mathbb{N}}$ satifying $\|\varphi_j\|_{L^1(M)}\leq 1$. Let $p_j(\omega):= |\int_M \varphi_j \omega \rmd x|$ and set ${\rm dist}(\omega_1,\omega_2) = \sum_j 2^{-j} \frac{p_j(\omega_1-\omega_2)}{1+p_j(\omega_1-\omega_2)}$.} In light of properties (a)--(c), Euler forms a  dynamical system in the
compact metric space $X_*$.  By time reversibility of the dynamics and uniqueness, $S_t(X_*)=X_*$ for all $t\in \mathbb{R}$.
By property (a), weak-$*$ limits of any $\omega_0\in X_*$ always exist along subsequences as $t\to\infty$. We define the Omega limit set for any $\omega_0\in X_*$ by
\be\label{omlimset} 
\Omega_+(\omega_0) := \bigcap_{s\geq 0} \overline{\{ S_t(\omega_0), t\geq s \}}^*.
\ee
This set represents the collection of all persistent `coarse-scale' (in that information can be ``averaged out" in the weak-$*$ limits)  behaviors of the Euler orbit emanating from $\omega_0$.  In effect, the set \eqref{omlimset} can be thought of as what a myopic observer without glasses would see (e.g. averaging fine-scale filaments together) when looking at the flow evolved for a long period of time.  
This set is compact and connected in $X_*$, which follows from the fact that $X_*$ is equipped with the weak$^*$ topology.  We shall discuss further properties of this set, as well as conjectures concerning its composition in \S \ref{conjectent}.  See also the lecture notes of Šverák for a detailed discussion of these points \cite{sv11}. 

We remark finally that Theorem \ref{yudothm} can be generalized to accommodate weakly divergent vorticities.  This was proved slightly later by Yudovich himself \cite{Y95} (independently established by Serfati \cite{S94a}).  For example, let $Y=Y( M)$ be the space consisting of all functions $f\in L^p( M)$ for $2\leq p< \infty$ which are finite in the norm (see \cite{Y95,CS21} for generality):
 \be\label{yudospace}
 \| f\|_Y := \sup_{p\geq 2} \frac{ \|f\|_{L^p}}{\log( p)}.
 \ee
The Euler equation defines a dynamical system $S_t: Y\to Y$ for all times $t\in \mathbb{R}$. Note that for any $\omega\in Y$, the corresponding velocity field is Osgood continuous, giving uniqueness of Lagrangian trajectories.  As a consequence, all $L^p$ with $1\leq p<\infty$ norms are conserved for solutions in $Y$.    This makes the space $Y$ endowed the topology defined by the norm \eqref{yudospace} another good (albeit slightly larger) phase space on which to study the long time properties of 2D fluid motion.

\section{2D fluids: a tale of isolation, wandering, and long-time strife }\label{2dsing}

In this section, we discuss  properties of two-dimensional perfect fluids exclusively. Two dimensions comes with the simplifying feature that any divergence-free velocity field $u\in H^1(M)\cap C(\overline{M})$ which is tangent to $\partial M$ can be represented as a perpendicular gradient of a scalar streamfunction which is constant on connected components of the boundary.  That is, there is a streamfunction $\psi:M\to \mathbb{R}$ such that $\psi|_{\partial M}={\rm (const.)}$ and  $u=\nabla^\perp \psi$.  Since vorticity is a scalar, \eqref{eevb}--\eqref{eevb1} say
\begin{align}
\partial_t \omega + u \cdot \nabla \omega &= 0  \quad \qquad\qquad \ \text{in} \quad  M,\\ 
u &= K_M[\omega] \quad \qquad \text{in} \quad  M,\\
\omega|_{t=0}&=\omega_0 \quad \qquad\qquad \text{in} \quad  M.
\end{align}
Here, if $M$ is simply connected then  $K_ M= \nabla^\perp \Delta^{-1}$ is the perpendicular gradient $\nabla^\perp =(-\partial_2,\partial_1)$ of the Green function for the Dirichlet Laplacian on $ M$. If the domain is multiply connected, circulations along the inner boundaries (which are preserved in time by the Euler motion by the circulation theorem) must be also given to recover the harmonic component (see Footnote \ref{harmfoot}). This velocity is automatically divergence-free  and tangent to the boundary $u\cdot \hat{n}|_{\partial M}=0$.

 \subsection{Isolation: stationary states, symmetry and stability}\label{isolation}

The simplest class of permanent fluid motions are those which have stationary (time independent) velocity fields.  The structure of such solutions (their stability, symmetry properties, stagnation sets, etc) sheds a light on large scale features which can emerge and persist over time in dynamical solutions.   In terms of  the streamfunction $\psi:M\to\mathbb{R}$, the condition for being a stationary Euler solution is the statement that the gradients of the streamfunction and the vorticity are collinear, i.e. $\nabla \psi \ \| \ \nabla \omega$. There are a plethora of stationary solutions of the two-dimensional Euler equation.
For example, every  vector field on $\mathbb{T}\times [0,1]$ or $ \mathbb{T}\times \mathbb{R}$  taking the form of a shear flow
\be
u(x_1,x_2)= \begin{pmatrix} U(x_2) \\ 0
\end{pmatrix} ,
\ee
for some $U\in C^{1}$ is a stationary solution  of the Euler equation with vorticity $\omega = -U'(x_2)$ and constant pressure.  Similarly, every circular flow  on the disk
\be
u (x_1,x_2) = V(|x|)  \frac{x^\perp}{|x|}
\ee
 with  $r=|x|$ and $V\in C^{1}$  is stationary and has vorticity $\omega(r)=\frac{1}{r} \partial_r(r V(r))$ and pressure $p(r) = -\int^r \frac{V(\rho)^2}{\rho}\rmd \rho$.
  One can also generate steady solutions on general Riemannian manifolds with or without smooth boundary (simply connected for simplicity) by specifying a Lipschitz function $F :\mathbb{R}\to \mathbb{R}$ and solving the elliptic problem 
\begin{align}\label{sse1}
\Delta \psi &= F(\psi)  \quad  \ \text{in} \quad  M,\\ 
\psi &= 0\quad \qquad \text{on} \quad  \partial M,\label{sse2}
\end{align}
where  $\Delta$ is the Laplace-Beltrami operator of the Riemannian metric $g$ on $M$.  
 Then $u= \nabla^\perp \psi$ defines a stationary solution of the Euler equations with vorticity $\omega= F(\psi)$ and pressure $p= -\frac{1}{2}|\nabla\psi|^2 + G(\psi)$ where $G$ is an antiderivative of $F$.   Examples of stationary solutions satisfying \eqref{sse1}--\eqref{sse2}  include cellular flows\footnote{We remark in passing that it is an intriguing open issue to establish nonlinear instability of the steady state \eqref{cellflow} in $L^2$ of vorticity (or, in fact, velocity).  This would represent a large-scale instability, unlike the small-scale instability/singularity formation which is the subject of \S \ref{strifesec}  and in contrast to the stability Theorem \ref{thmArn} below.  Being on the second shell and possessing a hyperbolic stagnation point, this steady state is susceptible to generation of larger scale motions which may destabilize the structure \cite{fv}. Established methods for proving instability \cite{MS,fsv,FS} are difficult to apply to this example due to its fully two-dimensional structure.  We remark that for higher eigenfunctions $\psi = \sin(x_1) \sin(mx_2)$ in the highly oscillatory regime with $m\gg1$, unstable eigenfunctions can be constructed using averaging methods together with the  Meshalkin-Sinai continued fraction technique \cite{fvy}. A compelling piece of evidence for nonlinear instability of  \eqref{cellflow} is that, from the work of Shvydkoy and Latushkin \cite{sl}, we know that the spectrum of the linearized Euler in $H^1$ and $H^{-1}$ is the full band $ |\mathfrak{Re} z| \leq 1$ and so it is very unstable linearly.  It is unknown whether or not there is an embedded eigenvalue; if so, the instability would  follow from the work of Lin \cite{L04}.
 } (also known as Kolmogorov flows) with streamfunction
\be\label{cellflow}
\psi_{\mathsf{c}} (x)= \sin(x_1)\sin(x_2)
\ee
which are eigenfunctions of the Laplacian, i.e. $\Delta \psi_{\mathsf{c}}= - 2 \psi_{\mathsf{c}}$. There are also the Kelvin-Stuart cat's--eye vortices defined by the streamfunction
\be\label{cats}
\psi_{\mathsf{KS}}(x)= -\log \left(\gamma \cosh x_1 + \sqrt{\gamma^2-1} \cos x_2\right), \qquad \gamma\geq 1.
\ee
which solve equation $\Delta \psi_{\mathsf{KS}} = -e^{2 \psi_{\mathsf{KS}}}$ (these flows make an interesting connection with conformal geometry \cite{taylor}).

 Note that \eqref{sse1}--\eqref{sse2} is not generally a ``good" equation for $\psi$ in that it may have no, one, or many solutions depending on $F$.  For instance, taking affine $F(\psi)=-\lambda \psi$, we obtain the eigenvalue problem for the Laplacian which exhibits all those phenomena depending on the proportionality constant.
 Note also that not every Euler solution need have a globally defined relationship between vorticity and the streamfunction. Indeed there are shear flows (perturbations of Kolmogorov flow) for which this is not so: e.g.  defined on $M=[0,2\pi]^2$ for each $\ve>0$ by
 \be
\psi^\ve(x_1,x_2)= \sin(x_2) + \ve \sin(2x_2).
 \ee
 In this case, $ \psi^\ve$ has a line of critical points.  The recent work \cite{GSPS21} gives examples of this phenomenon for non-radial vortex patches with compactly supported velocity field on the plane.
If $ \psi$ has no critical points (on a multiply connected domain) or isolated critical points, there always will exist a single global map $F$.  



There is a privileged class of such stationary solutions \eqref{sse1}--\eqref{sse2} that are ``stable." Specifically, if $F :\mathbb{R}\to \mathbb{R}$  is Lipschitz, single valued and either
\be\label{astabcond}
-\lambda_1(M) < F'(\psi)< 0 \qquad \text{or}\qquad 0 < F'(\psi)< \infty
\ee 
where $\lambda_1>0$ is the smallest eigenvalue of $-\Delta$ in $M$.  Such steady states $\omega$  are called \emph{Arnold stable}.  They are nonlinearly (Lyapunov) stable in the $L^2(M)$ norm (see \cite{ak}; Thm. 4.3 and Thm. \ref{thmArn} below). The two ranges are referred to as type I and type II Arnold stability conditions. They ensure that the steady state is either a minimum or a maximum
of the energy (the action) on the orbit of a vorticity distribution in the group of volume preserving diffeomorphims \cite{ak}. The degenerate case of $F(\psi)={\rm (const.)}$ corresponds to a steady state $\omega$ which is nonlinearly stable in $L^\infty(M)$.  All these steady states have the property that 
  \begin{itemize}
 \item [\textbf{(H0)}] the operator  
\be\label{Lop}
\mathcal{L}_{\psi} := \Delta - F'(\psi) 
\ee  
 has a trivial kernel in $H^1_0(M)$.  
 \end{itemize}
The following result of \cite{cdg} shows that all such steady states are severely constrained by the symmetries of the vessel $M$:

  \begin{myshade}
  \vspace{-1mm}
\begin{theorem}[Symmetry of stable equilibria \cite{cdg}]\label{thmcdg}
Let $(M,g)$ be a compact two-dimensional Riemannian manifold with smooth boundary $\partial M$.
Suppose that there exists a Killing field $\xi$ for $g$ tangent to $\partial M$. 
Let  $ \psi \in C^3(M)$ be a steady state streamfunction with vorticity $\omega=F(\psi)$ such that $\mathcal{L}_{\psi}$ satisfies {\rm \textbf{(H0)}}.
Then $\pounds_\xi\psi=0$ where $\pounds_\xi$ is the Lie derivative along $\xi$. 
\end{theorem}
\end{myshade}

Theorem \ref{thmcdg} is essentially the fact that in a Hamiltonian system, a critical point of the Hamiltonian is invariant under the symmetries of the Hamiltonian unless there is some non-trivial central manifold. Its proof is elementary and follows from applying $\pounds_\xi$ to the equation \eqref{sse1}, using the fact that the Lie derivative commutes with the Laplace-Beltrami operator, and noting that $\pounds_\xi$ is a tangential derivative along the boundary. 
A sufficient condition to ensure that $\mathcal{L}_{\psi}$ has a trivial kernel  is $F'(\psi)> -\lambda_1$ where $\lambda_1$ is the first eigenvalue of $-\Delta_g$ on $M$ --  all of the above stability conditions ensure this.
Theorem \ref{thmcdg} immediately implies that on the channel with $\xi=e_{x_1}$ 
all Arnold stable states are shears $u= v(x_2) e_{x_1}$, on the disk (or annulus) with $\xi=e_\theta$ 
all Arnold stable states are circular $u= v(r) e_\theta$ and on the hemisphere, all Arnold stable stationary solutions are zonal (functions of latitude).  It also implies that whenever there are two transversal Killing fields (as on the sphere or the torus), there are no non-trivial flows satisfying the Arnold stability conditions.  We note that rotation on the sphere ($\beta$--plane equations on the torus) introduces an anisotropy in the system and non-trivial Arnold stable solutions  can again exist \cite{cg21,WS99}.

Since  $F'(\psi)> -\lambda_1$ is an open condition, Theorem \ref{thmcdg} shows that such stable steady states on domains with symmetry are \emph{isolated} in $C^1$ of vorticity from asymmetric equilibria. Here isolated takes the following precise meaning
  \begin{myshade}
  \vspace{-1mm}
 \begin{definition}[Isolation from $G$ in $X$]\label{isodef}
 Given a Banach space $X$, we say that an element $u\in X$ is \emph{isolated} from $G\subset X$ if ${\rm dist}_X(u,G)>\ve$ for some $\ve>0$.
 \end{definition}
\end{myshade}
 Thus, Theorem \ref{thmcdg} reveals a strong form of rigidity of stable steady states and, as such, indirectly sheds light on the everyday observable phenomena such as axisymmetrization of tea or coffee when stirred, and formation of large stable vortices (e.g. hurricanes, the great red spot of Jupiter) in planetary atmospheres.

 It is not only the dynamically stable steady states which assume the symmetries of their vessels.  In some cases, all stationary solutions having certain structural properties to do with their stagnation sets must also ``fit" the geometry.  In particular,

  \begin{myshade}
  \vspace{-1mm}
\begin{theorem}[Hamel \& Nadirashvili \cite{HN17,HN19}]\label{thmNH}  The following hold:
\begin{itemize}
\item Let $M=  \mathbb{T} \times [0,1]$ be the periodic channel.  Let $u\in C^{2}(\overline{M})$ be  any stationary solution of the Euler equations satisfying $\inf_\Omega |u|> 0$.  Then $u$ is a shear flow $u= v(x_2) e_{x_1}$ with $v\in C^2$ having a strict sign.
\item Let $M=  \{ x \in \mathbb{R}^2 \ : \ |x|\in (1/2,1)\}$ be the annulus.  Let $u\in C^{2}(\overline{M})$ be any stationary solution of the Euler equations satisfying $\inf_\Omega |u|> 0$. Then $u$ is a circular flow $u= v(r) e_\theta$ with $v\in C^2$ having a strict sign.
\item Let $M=  \{ x \in \mathbb{R}^2 \ : \ |x|<1\}$ be the disk.  Let $u\in C^{2}(\overline{M}\setminus \{0\})$ be any stationary solution of the Euler equations satisfying $u> 0$ on $\overline{M}\setminus \{0\}$. Then $u$ is a circular flow $u= v(r) e_\theta$ with $v\in C^2$ having a strict sign.
\end{itemize} 
\end{theorem}
\end{myshade}
In Theorem \ref{thmNH}, we have highlighted just a couple results of \cite{HN17,HN19} -- termed Liouville theorems -- which show symmetry of solutions having a certain structure.   We remark that it is an interesting open issue (conjectured in \cite{HN19}) whether on the disk domain it is enough to say  $u\in C^{2}(\overline{M}\setminus \{z\})$ for some interior point $z\in M$ with $u> 0$ on $\overline{M}\setminus \{z\}$ to conclude that $z$ must be the origin and $u$ must be circular.

In a similar spirit to Theorem \ref{thmNH}, there is the recent work \cite{GSPSY21} which establishes radial symmetry for all compactly supported single signed vorticity distributions on $\mathbb{R}^2$, as well as for more singular vortex sheet configurations \cite{GSPSY21b}.  Free boundary fluid bodies also exhibit strong forms of rigidity \cite{HN19,cdgb}.

Note that the flows of Theorem \ref{thmNH} are \textit{not} isolated from other stationary solutions, since one can find nearby shear/radial flows in any given $C^{1}$ neighborhood.  However, this theorem shows that stationary solutions without stagnation points \textit{are} isolated from any non-shear stationary solution of the equations in $C^{1}$ (e.g. from cellular flow \eqref{cellflow} and cat's--eyes \eqref{cats}).  
It is important to emphasize that without the assumption of no stagnation points, shear flows are no longer isolated and that these statements depend strongly on topology. For instance, Lin and Zeng \cite{LZ11} showed that  Couette flow $u(x)=(x_2,0)$ is isolated from non-shears in the $H^s$, $s>3/2$ topology while there exist non-shear cat's--eye vortices arbitrarily close in the $s<3/2$ topologies.  Coti Zelati--Elgindi--Widmayer \cite{CEW20} showed that Kolmogorov flow $u_\mathsf{K}(x) = (\sin (x_2), 0)$ is not isolated from non-shear steady states even in the analytic topology, showing an extreme example of the effect of stagnation. However the work of \cite{CEW20} also showed that the assumption of non-stagnation is not always necessary for isolation by showing that all (Sobolev $H^s$, $s>5$) neighboring solutions to Poiseuille flow $u_0(x)=(|x_2|^2,0)$ are shear flows.

One can ask about the ``flexibility" of steady states more systematically.  For example, say one wants to find nearby steady states corresponding to  
a slight change in the vorticity, a  wrinkling of the domains (slight changes of the background metric), or a wiggling in the boundary of the vessel (see Figure \ref{fig:chan}).

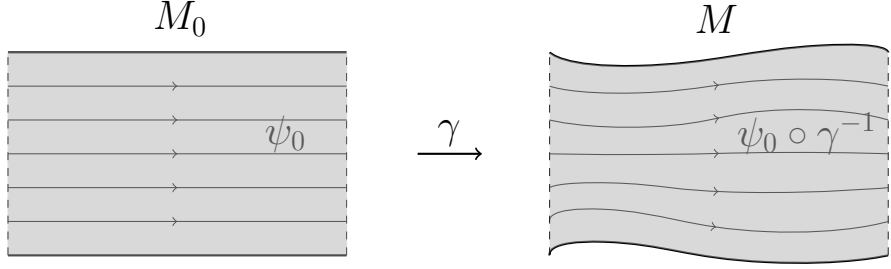
\begin{figure}[h!]
\centering
\begin{tikzpicture}[scale=0.9, every node/.style={transform shape}]

	 \draw[thick, black, ->] (3.55,1.5)--(4.5,1.5);
	 \draw  (3.7,1.8) node[anchor=west] {\Large ${\gamma}$};
	 \draw  (-0.5,3.5) node[anchor=west] {\Large $M_0$};
	 \draw  (7.5,3.5) node[anchor=west] {\Large $M$};
         \draw  (1.6,1.8) node { \Large $\psi_0$};
           \draw  (9.3,1.8) node { \Large $\psi_0\circ \gamma^{-1}$};

	\draw [thick] (-2.5,0)--(2.5,0);
	\draw [thick] (-2.5,3)--(2.5,3);
	
	\draw[thin,   decoration={markings, mark = at position -0.5 with {\arrow{>}}},
        postaction={decorate}] (-2.5,2.5)--(2.5,2.5);
	\draw [thin,   decoration={markings, mark = at position -0.5 with {\arrow{>}}},
        postaction={decorate}]  (-2.5,2)--(2.5,2);
	\draw [thin,   decoration={markings, mark = at position -0.5 with {\arrow{>}}},
        postaction={decorate}]  (-2.5,1.5)--(2.5,1.5);
	\draw [thin,   decoration={markings, mark = at position -0.5 with {\arrow{>}}},
        postaction={decorate}] (-2.5,1)--(2.5,1);
	\draw[thin,   decoration={markings, mark = at position -0.5 with {\arrow{>}}},
        postaction={decorate}] (-2.5,0.5)--(2.5,0.5);
	
        \draw [dashed] (-2.5,0)--(-2.5,3);
        \draw [dashed] (2.5,0)--(2.5,3);
	\draw[fill, black!30!white, opacity=0.5] (-2.5,0)--(2.5,0)-- (2.5,3)--(-2.5,3)--cycle;
	
	\draw [name path=A,thick] plot [smooth, tension=1]  coordinates {(5.5,0) (6.3,0.2) (8.7,-0.1)  (10.5,0)};
	\draw [name path=B, thick] plot [smooth, tension=1]  coordinates {(5.5,3) (6.8,2.8)  (9.5,3.1)  (10.5,3)} ;

	\draw [thin,   decoration={markings, mark = at position -0.5 with {\arrow{>}}},
        postaction={decorate}]  plot [smooth, tension=1]  coordinates {(5.5,2.5) (6.8,2.4) (9,2.6)   (10.5,2.5)};
	\draw [thin,   decoration={markings, mark = at position -0.5 with {\arrow{>}}},
        postaction={decorate}]  plot [smooth, tension=1]  coordinates {(5.5,2) (7,1.9) (9,2.15)   (10.5,2)};
	\draw [thin,   decoration={markings, mark = at position -0.5 with {\arrow{>}}},
        postaction={decorate}]  plot [smooth, tension=1]  coordinates {(5.5,1.5)  (7,1.49) (9,1.52)  (10.5,1.5)};
	\draw [thin,   decoration={markings, mark = at position -0.5 with {\arrow{>}}},
        postaction={decorate}] plot [smooth, tension=1]  coordinates {(5.5,1) (6.4,1.07) (8.4,0.93) (10.5,1)};
	\draw [thin,   decoration={markings, mark = at position -0.5 with {\arrow{>}}},
        postaction={decorate}]  plot [smooth, tension=1]  coordinates {(5.5,0.5) (6.3,0.68) (8.7,0.35) (10.5,0.5)};
	
        \draw [dashed] (5.5,0)--(5.5,3);
        \draw [dashed] (10.5,0)--(10.5,3);
\tikzfillbetween[of=A and B]{black!30!white, opacity=0.5};	   
\end{tikzpicture}
  \caption{Deformation of periodic channel $M_0$ and streamfunction $\psi_0$ by $\gamma$.}
  \label{fig:chan}
\end{figure}

There have been a number of results recently on the flexibility of stationary solutions. Here we discuss the implications for a large class of stable steady states.
Specifically, let $M_0\subset \mathbb{R}^2$ and $\psi_0:M_0\to \mathbb{R}$ be any streamfunction  satisfying:
  \begin{itemize}
 \item [\textbf{(H1)}] The vorticity $\omega_0=\Delta \psi_0=F_0(\psi_0)$ satisfies
 $F_0'(\psi_0)> -\lambda_1(M_0)$. \vspace{2mm}
  \item [\textbf{(H2)}]
There exists $c_{\psi_0}>0$ such that $\forall$ $c\in {\rm image}(\psi_0)$ one has $ \mu(c)\leq 1/c_{\psi_0}$ where
   \be\label{rotation}
 \mu(c):= \oint_{\{\psi_0=c\}} \frac{\rmd \ell}{|\nabla \psi_0|}.
  \ee
  \end{itemize}

The purpose of Hypothesis \textbf{(H1)} is to ensure something slightly stronger than \textbf{(H0)}, namely that the operator $\mathcal{L}_\psi$ be positive definite. Hypothesis \textbf{(H1)} is guaranteed for all known stable states (e.g.  \eqref{astabcond} and constant vorticity).
Hypothesis \textbf{(H2)}   ensures that the period of revolution of fluid parcels along streamlines (the travel time) is bounded.  This condition  is satisfied for any base flow without stagnation points on  multiply connected domains.  A sufficient condition  on simply connected domains  to ensure it holds  is that there is a unique critical point where the vorticity does not vanish $\omega_0\neq 0$.  Such solutions are flexible:

  \begin{myshade}
  \vspace{-1mm}
\begin{theorem}[Structural stability of stable states \cite{cdg}]\label{2dEthmarnold}
  \label{existence}
 Let $\alpha>0$ and $k \geq 3$. Let  $(M_0,g_0)$ be a compact two-dimensional Riemannian submanifold of $\mathbb{R}^2$ with $C^{k,\alpha}$ boundary.
  Suppose $\psi_0\in  C^{k,\alpha}(M_0)$ satisfies  {\rm \textbf{(H1)}--\textbf{(H2)}} on $(M_0,g_0)$ with vorticity profile $F_0\in C^{k-2,\alpha}(\mathbb{R})$.
Then there are constants $\ve_1,\ve_2, \ve_3$ depending only on $ M_0, F_0, g_0$ and
$\|\psi_0\|_{C^{k,\alpha}}$ such that if $(M, g)$ is a compact Riemannian manifold
and $\rho:M_0\to \mathbb{R}$  with $\int_{D_0} \rho\  \rmd {\rm vol}_{g_0} = {\rm Vol}_g(D)$ and $g:M_0\to \mathbb{R}^2$ satisfy
\begin{align*}
\|\partial M - \partial M_0\|_{C^{k,\alpha}}\leq \ve_1,
\quad
  \|\rho-1\|_{C^{k,\alpha}(M_0)}\leq \ve_2,\quad
  \|g-g_0\|_{C^{k,\alpha}(M_0)}&\leq \ve_3,
  \end{align*}
then there is a diffeomorphism $\gamma: M_0 \to M$ with Jacobian ${\rm det} (\nabla \gamma)= \rho$, and a function $F: \mathbb{R} \to \mathbb{R}$ close in $C^{k-2,\alpha}$ to $F_0$
so that $\psi = \psi_0\circ \gamma^{-1} \in C^{k,\alpha}(M)$ and $\psi$ satisfies \eqref{sse1}--\eqref{sse2} on $(M,g)$.
Thus, $u =\nabla^\perp \psi$  is an
Euler solution on $(M,g)$ satisfying   {\rm \textbf{(H1)}--\textbf{(H2)}}  nearby $u_0$ whose vorticity is $\omega =  F(\psi)$.
\end{theorem}
\end{myshade}

Theorem \ref{2dEthmarnold} generates solutions with the same streamline topologies of the base state since the new streamfunction $\psi$ is on the orbit of the base $\psi_0$ in the group of area preserving diffeomorphisms.\footnote{As we will described in \S \ref{strifesec}, this fact is useful for constructing isochronal steady solutions (having time-periodic flowmap) on  domain nearby disks/ellipses. }
This result built on the previous works of \cite{WV05,CS12}.  Wirosoetisno and Vanneste \cite{WV05} proved that shear flows without stagnation points on the channel can be deformed to domains with wiggled boundaries.  This analysis can be extended also to circular flows on annuli or with single critical points on the disk, showing that all the steady states of Theorem \ref{thmNH} are flexible to deformation of the fluid vessel. Choffrut and  Šverák \cite{CS12} gave a full characterization of the steady states nearby certain ones satisfying a version of (\textbf{H1})  on annular domains by showing that they are in one-to-one correspondence with their distribution functions, i.e. for all $\omega_0$ sufficiently close to $\omega$, there exists a unique stationary solution on the orbit 
$
\mathcal{O}_{\omega_0}:= \{ \omega_0 \circ \varphi \ :\ \varphi  \in \mathscr{D}_\mu(M)  \}.$
 in the group of area preserving diffeomorphisms.  That stationary states come in such rich families is a truly infinite dimensional feature of the Euler equations. 
 
We remark in passing that, in the opposite direction,
 Ginzburg and Khesin \cite{GK0,GK} show that if  $M$ is a simply connected planar domain and $\omega_0$ is Morse, positive and has both a local maximum and minimum in the interior, then $\mathcal{O}_{\omega_0}$ contains no smooth Euler steady state.  
Izosimov and Khesin \cite{IK} later gave
necessary conditions on the vorticity $\omega_0$  to have a smooth steady Euler solution on $\mathcal{O}_{\omega_0}$  for
any metric, as well as a sufficient condition for the existence of a steady solution for some metric.

It is important to note that Hypothesis (\textbf{H0}) or (\textbf{H1}) are not strictly needed for the conclusions of Theorem \ref{2dEthmarnold} to hold if the kernel of $\mathcal{L}_\psi$ can be well understood.  This is demonstrated by the work of Coti Zelati--Elgindi--Widmayer \cite{CEW20} for Kolmogorov flow $u_\mathsf{K}(x_1,x_2) = (\sin (x_2), 0)$. These shear flows are eigenfunctions of the Laplacian having vorticity $F_0(\psi_\mathsf{K})= -\psi_\mathsf{K}$. The operator $\mathcal{L}_{\psi_\mathsf{K}}:=\Delta-F'_0(\psi_\mathsf{K})$ has a non-trivial kernel $ {\rm Ker}(\mathcal{L}_{\psi_\mathsf{K}})= {\rm span}\{\sin(x_1), \cos(x_1),\sin(x_2), \cos(x_2)\}$.  Nevertheless, nearby non-shear solutions can be found by a Lyapunov-Schmidt scheme in which, to deal with this degeneracy,  extra degrees of freedom from the kernel are introduced within the construction.

All of the above results point to Euler equilibria being non-isolated from one another.

 \begin{myshade2}
  \vspace{-1mm}
\begin{problem}\label{conj1}
Show that there is no steady 2D Euler solution $\omega\in C^{1}$ that is isolated in $C^{1}$ from other stationary solutions (modulo symmetries).
\end{problem}
\end{myshade2}

If isolated steady states did exist, then the ray generated by scaling would be a candidate attractor for the Eulerian dynamics either forward or backward in time.

We now begin our discussion of dynamics nearby steady states.  We first recall the famous theorem of Arnold showing that certain large-scale structures (which include monotone shear flows on the channel and radial vortices with decreasing vorticity on the disk or plane) do persist for all time under evolution.  It reads:

\begin{myshade}
  \vspace{-1mm}
\begin{theorem}[Arnold \cite{arn11,arn12,arn13,ak}]\label{thmArn}  
Let $M\subset \mathbb{R}^2$ be a simply connected domain with smooth boundary.  Let $u_*:= \nabla^\perp \psi_*$ be a steady solution of Euler with the property that $\omega_*:=\nabla^\perp \cdot u_* = F(\psi_*)$. If $F$ is single valued and
\be\label{ranges}
-\lambda_1(M) < F'(\psi)< 0 \quad \text{or}\quad 0 < F'(\psi)< \infty
\ee 
where $\lambda_1>0$ is the smallest eigenvalue of $-\Delta$ in $M$,  then $\omega$ is nonlinearly stable in the $L^2(M)$ norm.  That is, for any $\delta>0$, there is an $\ve:=\ve(\delta)$ such that for all $\omega_0\in L^\infty(M)$ satisfying $\|\omega_0-\omega_*\|_{L^2}<\ve$ and with equal circulations at the boundary of $\Omega$ to $v$, then $\| S_t(\omega_0)-\omega_*\|_{L^2} \leq  \delta$ for all $t\in \mathbb{R}$.
\end{theorem}
\end{myshade}

The two ranges in \eqref{ranges} are referred to as type I and type II.  The result holds also for multiply connected domains if we restrict to perturbations which conserve the circulation along the boundaries.  See the recent work of Gallay and Šverák for a detailed discussion of this theory for radial vortices \cite{GS21}. We note also that similar results can be established for other nonlinear fluid equations \cite{HMRW}.

We remark that, since Arnold stable solutions come in infinite dimensional families  according to Theorem \ref{2dEthmarnold} and the results of \cite{CS12}, the asymptotic behavior of the long time limits can be extremely difficult to determine. Some results and conjectures in this direction are deferred to \S \ref{conjectent}.

\begin{proof}
For a suitable $H:\mathbb{R}\to \mathbb{R}$ to be chosen, we introduce the functional
\be
A[\omega] = \frac{1}{2} \int_M |u|^2 \rmd x + \int_M H[\omega] \rmd x.
\ee
We want $A[\omega]$ to be locally convex/concave and coercive with respect to the $L^2$ norms of the velocity
and vorticity in a neighborhood of the critical point. Recall, coercive with respect to a norm means
that control on the energy functional controls the norm.
Since $F$ is strictly decreasing  for type I flows and increasing for type II flows, it is invertible and we can $H$ such that $H'[\omega]= F^{-1}(M)$ and
\be
H''[\omega] =1/F'(F^{-1}(M)).
\ee
It follows that  $H$ is strictly concave if the type I condition holds and strictly convex if the type II condition holds.

Note that $A$ is invariant under the Euler flow  (energy and enstrophy are conserved).
For $\omega_0\in L^\infty(M)$, denote $\omega(t)= S_t(\omega_0)$ and $u(t) = K_M[\omega(t)]$.  Then
\begin{align*}
A[\omega_0] - A[\omega_*]&= A[S_t(\omega_0)]  - A[\omega_*].
\end{align*}
We next write
\begin{align*}
A[S_t(\omega_0)]  - A[\omega_*]&= \frac{1}{2} \int_M\Big( |u(t)|^2-  |u_*|^2\Big) \rmd x + \int_M \Big(H[\omega(t)] - H[\omega_*] \Big)\rmd x\\
&= \frac{1}{2} \int_M\Big( |u(t)|^2-  |u_*|^2\Big) \rmd x \\
&\qquad + \int_M \Big(H[\omega(t)] - H[\omega_*]- H'[\omega_*] (\omega-\omega_*)  \Big)\rmd x\\
&\qquad \qquad + \int_M H'[\omega_*] (\omega-\omega_*) \rmd x.
\end{align*}
Using the fact that $H'[\omega_*] =F^{-1}(\omega_*) = \psi_*$, and that $u_*^\perp=-\nabla \psi_*$, we have
\begin{align*}
\int_\Omega H'[\omega_*] (\omega-\omega_*) \rmd x&= \int_M \psi_*\nabla^\perp \cdot (u-u_*) \rmd x = \int_M\Big(|u_*|^2 - u_*\cdot u\Big)\rmd x,
\end{align*}
since $\psi_*|_{\partial M}=0$.
Thus we obtain
\begin{align} \nonumber
 A[\omega_0]  - A[\omega_*]&=  \frac{1}{2} \int_M |u(t)-u_*|^2 \rmd x +  \int_M \Big(H[\omega(t)] - H[\omega_*]- H'[\omega_*] (\omega-\omega_*)  \Big)\rmd x.
\end{align}
By concavity/convexity of $H$, we have the following bounds
\begin{align*}
\text{type I}: \qquad 
 -\frac{c}{2}(\omega-\omega_*)^2 \geq  H[\omega] - H[\omega_*] - &H'[\omega_*] (\omega-\omega_*)  \geq   -\frac{C}{2}(\omega-\omega_*)^2,\\
\text{type II}: \qquad 
 \frac{C}{2}(\omega-\omega_*)^2 \geq H[\omega] - H[\omega_*] - &H'[\omega_*] (\omega-\omega_*)  \geq   \frac{c}{2}(\omega-\omega_*)^2,
\end{align*}
for appropriate constants $0<c<C<\infty$. Recalling Poincar\'{e}'s inequality, 
\be
\lambda_1 \int_M |u(t)-u_*|^2 \rmd x \leq  \int_M |\omega(t)-\omega_*|^2\rmd x, 
\ee
we deduce that
\begin{align*} 
\text{type I}: \qquad  A[\omega_*]-A[\omega_0]  &\geq  \frac{C-1/\lambda_1}{2}  \int_M |\omega(t)-\omega_*|^2\rmd x,\\ 
\text{type II}: \qquad  A[\omega_0]  - A[\omega_*]&\geq  \frac{c}{2}  \int_M |\omega(t)-\omega_*|^2\rmd x.
\end{align*}
For type I steady states, $C> 1/\lambda_1$ giving coercive control.
Thus the kinetic energy and the enstrophy of the perturbation are uniformly bounded by data with
\be
 |A[\omega_0]  - A[\omega_*]| \lesssim  \int_M |\omega_0-\omega_*|^2\rmd x,
\ee
thereby completing the proof.
\end{proof}


A consequence of this result together with interpolation is that $\omega$ is Lyapunov stable in the same sense in  $L^{p}(M)$ for $p\in [2,\infty)$.  However, the dynamics of the Euler equation in the spaces $L^{p}(M)$ is not understood  for $p<\infty$.  Global existence of solutions with initial data $\omega_0\in L^2( M)$ is known \cite{Y63,DM87} but not uniqueness. In fact, Euler is known to possess multiple weak solutions emanating from $\omega_0\in L^p$ with $p<\infty$, at least in the presence of an external force (see Vishik \cite{V18}). As such, Arnold's stability theorem is  an ``a priori estimate" rather than a true dynamical stability \cite{arn12}.
In view of this, it is of interest to understand the stability properties of stationary solutions in spaces on which Euler is a well defined dynamical system.  This question was asked  in the context of the space $L^\infty$ by Yudovich (Question \textbf{6a} of \cite{Y03}) of Arnold stable solutions. Unfortunately $L^2$ stable Euler states are not, in general, stable in the $L^\infty$ topology (see Remark \ref{rem} below).

Here we give an example of a Banach space, $X$, on which Euler is a dynamical system and Arnold stable solutions (in fact, any $L^2$ stable solutions) are Lyapunov stable in the $X$ topology.
Specifically, let $X=X( M)$ be the space consisting of all functions $f\in L^p( M)$ for $2\leq p< \infty$ which are finite in the norm:
\be\label{xtop}
\|f\|_{X}=\sum_{p=2}^\infty \frac{\|f\|_{L^p}}{p (\log(p))^2}.
\ee
Observe that $\|f\|_{X}\lesssim \|f\|_{L^\infty}$ whenever the right hand side is bounded.
 It is also not difficult to check that this is an existence and uniqueness class for the 2d Euler equation. Indeed  (see Lemma \ref{lem} below), $ \| f\|_Y \lesssim  \| f\|_X$ where $Y$ is the Yudovich space defined by \eqref{yudospace} in which Euler  is globally wellposed \cite{Y95}. Since all $L^p$ with $1\leq p<\infty$ norms of vorticity are conserved for solutions in $Y$, one has $\|\omega_0\|_X= \|\omega(t)\|_X$ and $\omega(t)\in X$ for all $t\in \mathbb{R}$.

We show that if $ \omega_* \in X$ is the vorticity of a steady Euler solution which is stable in $L^2$ (e.g. satisfies Arnold's conditions  \eqref{astabcond}), then it is stable in $X$:

\begin{myshade}
  \vspace{-1mm}
\begin{theorem}\label{stabinX}
Let $ \omega_* \in X$ be a stationary state of the 2D Euler equation that is Lyapunov stable in $L^2( M)$. Then, $\omega_*$ is Lyapunov stable in $X$. 
\end{theorem}
\end{myshade}

\begin{proof}
Fix $\ve\in(0,1)$. Since $\omega_*\in X$, there exists $N_\ve>1$ so that
\be
\sum_{p=N_\ve }^\infty \frac{\|\omega_*\|_{L^p}}{p (\log(p))^2}<\frac{\ve}{4}.
\ee 
Thus, if $\omega_0\in B_{\ve/4}( \omega_*)$ in the $X$ topology and $\omega(t):= S_t(\omega_0)$, we have 
\be
\sum_{p=N_\ve }^\infty \frac{\|\omega(t)\|_{L^p}}{p (\log(p))^2}< \frac{\ve}{2}, \qquad \text{for all} \qquad t\geq 0,
\ee 
by conservation of the $L^p$ norms. 
Since $\omega_*$ is stable in $L^2$, for each $\kappa>0$, there exists $\delta_\kappa>0$ so that $\|\omega_*-\omega_0\|_{X}<\delta_\kappa$ implies
\be
 \|\omega_*-\omega(t)\|_{L^2}<\kappa  \quad \text{for all} \quad t\geq 0 .
\ee
This follows since the $X$ norm controls the $L^2$ norm. 
Now, by interpolation, if $2\leq p\leq N_\ve-1$ we have
\begin{align} \nonumber
\|\omega_*-\omega(t)\|_{L^p} &\leq \|\omega_*-\omega(t)\|_{L^2}^\frac{N_\ve-p}{N_\ve-2}\|\omega_*-\omega(t)\|_{L^{N_\ve}}^{\frac{p-2}{N_\ve-2}}\\ \nonumber
&<\kappa^{\frac{1}{N_\ve-2}}\max\{\|\omega_*-\omega(t)\|_{L^{N_\ve}},1\}\\
&\leq \kappa^{\frac{1}{N_\ve-2}}N_\ve(\log(N_\ve))^2.
\end{align}
Thus, we obtain
\begin{align} \nonumber
\|\omega_*-\omega(t)\|_{X}&\leq \sum_{p=1}^{N_\ve-1}\frac{\|\omega_*-\omega(t)\|_{L^p}}{p(\log(p))^2}+\sum_{p=N_\ve}^\infty \frac{\|\omega_*\|_{L^p}+\|\omega(t)\|_{L^p}}{p(\log(p))^2}\\
&<\kappa^{\frac{1}{N_\ve-2}}(N_\ve)^2(\log(N_\ve))^2+\frac{3\ve}{4}.
\end{align}
Choosing
$\kappa= \Big({\ve}/{4N_\ve^2\log(N_\ve)^2}\Big)^{N_\ve-2},$ we find $\|\omega_*-\omega(t)\|_{X}<\ve$ for all $t\geq 0$.
\end{proof}

\begin{lemma}\label{lem}
We have the following inclusions: $L^\infty \subset X\subset Y.$
\end{lemma}
\begin{proof}
The first inclusion $L^\infty \subset X$ is obvious.  To see that $X\subset Y$, note that for any $f\in X$ there exists a subsequence $\{p_i\}_{i\in N}$ with $p_i \to \infty$ such that 
\be
\|f\|_{L^{p_i}} \leq C \log(p_i)
\ee
for some constant $C>0$.  We aim to show that $\|f\|_{L^{p}} \leq C' \log(p)$ on the whole sequence for a possibly different constant $C'>0$.  Indeed, fix any $q \in [p_i, p_{i+1}]$.  By interpolation, we have
\be
\|f\|_{L^{q}} \leq   \|f\|_{L^1}^\frac{p_{i+1} -q}{p_{i+1}-1}\|f\|_{L^{p_{i+1}}}^{\frac{q-1}{p_{i+1}-1}} .
\ee
Now note that for all  $m$ sufficiently large we have
\be
\log(n)^{\frac{m-1}{n-1}} \leq \log(m), \qquad \text{for all } \quad n \geq m.
\ee
Thus, for all $i$ sufficiently large, we have for all $q \in [p_i, p_{i+1}]$ that
\begin{align}
\|f\|_{L^{q}} &\leq   \|f\|_{L^1}^\frac{p_{i+1} -q}{p_{i+1}-1} C^{\frac{q-1}{p_{i+1}-1}} \ln(q)\leq \max\{   \|f\|_{L^1}, 1, C\}   \ln(q).
\end{align}
This proves that $f\in Y$.
\end{proof}

  We remark that stability trivially holds for the phase space $X_*$ given by \eqref{Xstar} endowed with the $L^2$ topology.  However, the space $X$ in Theorem \ref{stabinX} reflects a stronger topology for which all $L^2$ stable steady states are Lyapunov stable and for which particle trajectories are controlled (since $\omega\in X$ implies $u$ is Osgood).

\begin{remark}[Nonlinear stability in $L^\infty$]\label{rem}
There exist particular stationary solutions which possess stronger stability properties.  For example, in simply connected domain, every flow with constant vorticity (Couette flow on the periodic channel) is nonlinearly stable in $L^\infty$ since  there is no ``stretching" of a vortical perturbation. 
In view of this fact, one may ask whether or not stability in $L^2$ can be promoted to stability in $L^\infty$ in general \cite{MSY,Y03}.
This is not the case.  For example, consider the Rankine vortex, which is the stationary solution
$\omega_*= \mathbf{1}_{B_1(0)}$.   The vortex patch is known to be nonlinearly stable in $L^2$ for a certain class of perturbations (\S 7 of \cite{WP85}).  However, it is not stable in $L^\infty$.  To see this, consider an initial condition of two patches, the Rankine vortex and a small patch of intensity $\ve$ outside.  The presence of the second patch immediately shifts the Rankine vortex, resulting in an $O(1)$ change in $L^\infty$.  On the other hand, this example reflects a kind of illposedness (in the sense of Hadamard, a lack of continuity of the solution operator $S_t:L^\infty\to L^\infty$) for the Cauchy problem in $L^\infty$.  We note that if $\omega_*$ is continuous, then $S_t:L^\infty\to L^\infty$ is continuous at $\omega_*$.  A natural question is 
\end{remark}

  \begin{myshade3}
  \vspace{-1mm}
\begin{question}[Yudovich (2003), \cite{Y03}]\label{yudoquest}
Does there exist a continuous $L^2$ stable steady state $\omega_*$  that is unstable in $L^\infty$?
\end{question}
\end{myshade3}

  Theorems \ref{thmArn} and \ref{stabinX} shed some light on the Eulerian dynamics on the phase space $X$.
  Systems that cover all accessible phase volume are sometimes called ergodic, mixing or stochastic. Ergodicity is often invoked to predict the asymptotic behavior
  as maximizers of some entropy functional defined on the phase space -- we shall briefly return to this point in \S \ref{conjectent}.
An immediate corollary of  \ref{thmArn} and \ref{stabinX} is the failure of ergodicity in this sense in the $X$ phase space due to the presence of Lyapunov stable equilibria.
Another example of failure of ergodicity is the existence of wandering points. This is the subject of the subsequent section \S \ref{wanderingsec}.
We remark that it is possible that neither of these phenomena exclude the possibility of ergodicity in a weak-$*$ sense.  Establishing that this is not the case is the essence of Problem \ref{omlimprob}.

We close this section with one final open question:
  \begin{myshade3}
  \vspace{-1mm}
\begin{question}\label{queststab}
Let $(M,g)$ be a compact two-dimensional Riemannian manifold without boundary (e.g. $\mathbb{S}^2$ with the round metric or $\mathbb{T}^2$ with the flat metric).  Do there exist any non-trivial stable steady Euler solutions?
\end{question}
\end{myshade3}
In the above, the perturbations should have trivial harmonic component. We should remark that the first eigenfunctions of Laplacian (harmonics) on any such manifold are stable, but these are necessarily constant and thus are trivial.   Desingularized periodic point vortex lattices \cite{aref,Saffman,Tkachenko}  may be an interesting to investigate in the context of Question \ref{queststab}.

\subsection{Wandering:  non-recurrence in infinite dimensions}\label{wanderingsec}

In this section, we discuss another truly infinite dimensional feature of perfect fluid motion.
We first recall the notion of non-wandering from finite-dimensional dynamics, which takes the form of the Poincar\'{e} recurrence theorem:

  \begin{myshade}
  \vspace{-1mm}
\begin{theorem}[Poincar\'{e} Recurrence]\label{PRthm}  
Let $(X,\Sigma,\mu)$ be a finite measure space and $\varphi: X\mapsto X$ be a measure-preserving transformation.  For any $E\in \Sigma$ (the sigma algebra of measurable subsets of $X$), the measure
\be
\mu\big(\{x\in E \ | \ \exists N\in \mathbb{N} \ \text{such that} \ \varphi^n(x) \notin E \ \text{for all} \ n>N\}\big) =0.
\ee
\end{theorem}
\end{myshade}

  This  says that in a finite measure space, the images of a positive measure set under a measure-preserving transformation will be forced to intersect the original set repeatedly. 
  This behavior should be contrasted with that of wandering:

\begin{myshade}
  \vspace{-1mm}
 \begin{definition}[Wandering in $X$]\label{wanddef}
Let $X$ be Banach space and $S_t: X\to X$ for $t\in \mathbb{R}$ be a dynamical system. 
 A given $\omega\in X$ is called a \emph{wandering point} if for all initial conditions in a neighborhood $\| \omega_0-\omega\|_X\leq \ve $, there is a moment $T>0$ such that the solution emanating from any such $\omega_0$ leaves that vicinity forever after, i.e. $\|S_t(\omega_0)-\omega\|_X> \ve$  for all $t>T$.   The dynamics has a \emph{wandering neighborhood} $U$ in $X$ if  $S_t(U)\cap U=\emptyset$ for all $t>T$ for some $T:=T(U)$. 
 \end{definition}
\end{myshade}

  There are various finite dimensional approximations to the Euler equations (see e.g. Chapter 5 of \cite{MP12}); for example, the Galerkin approximation \cite{MB} , Sullivan's fluid algebra models \cite{S11}, the Kirchhoff point vortex approximation \cite{K83}, and lattice models \cite{S20}.   An application of the Poincar\'{e} Recurrence Theorem \ref{PRthm} can be used to show that some of these finite dimensional approximations are non-wandering.  For example, energy-preserving Galerkin truncations form a dynamical system on a sphere in $\mathbb{R}^m$ where $m$ is the number of Fourier modes retained.  For some such trucations, the motion preserves the Euclidean volume on spherical shells (Liouville theorem).   Sullivan's fluid algebra models  and Zeitlin's approximation \cite{ZM} 
 also all have this property. Thus, the Poincar\'{e} Recurrence theorem applies and almost every moving point returns to a neighborhood in which it started repeatedly. We remark that this same properties allows Galerkin Euler to support a Gibbs measure (Gaussian measure) since the flow is tangent to the $L^2$ spheres and the Euclidean volume is preserved on spherical shells (which are preserved). This measure corresponds to a $k^2$
equipartition energy spectrum. This so-called Liouville theorem was proved first by Burgers \cite{burgers} and then Lee \cite{lee} and Hopf \cite{hopf}. Sullivan \cite{S14} establishes the result for fluid algebra models \cite{S11} and uses it to motivate a conjecture about almost-sure wellposedness in the zero-viscosity limit.

 The non-wandering behavior implied by Poincar\'{e} Recurrence relies crucially on the space $X$ in Theorem \ref{PRthm} being finite and finite dimensional. 
 Despite the fact that certain finite dimensional approximations to Euler are volume preserving, an example of wandering points for the infinite dimensional Euler equations in the unit $L^\infty$ ball phase space (recall from the discussion around Thm. \ref{yudothm} that $X_*$  is a compact metric space preserved by the Euler dynamics) was given by Nadirashvili.

\begin{myshade}
  \vspace{-1mm}
\begin{theorem}[Nadirashvili \cite{N}]\label{thmwand1}  Let $M=  [-\pi, \pi) \times [0,1]$ be the periodic channel.  There is a scalar field $\xi\in C^{\infty}(M)$  such that 
for any $\ve>0$ sufficiently small, there exists a finite time $T:=T(\ve)$ so that any  initial condition satisfying $\|\omega_0 - \xi\|_{L^\infty} <\ve$ wanders in $L^\infty$, i.e. $\|S_t(\omega_0) - \xi\|_{L^\infty}>\ve$ for all $t>T$.
\end{theorem}
\end{myshade}

In Nadirashvili's original paper, he works on the annulus with a circular flow having zero vorticity (streamfunction $\psi(x) = \ln |x|$). The ideas are the same.  The significance of this result is that it establishes the existence of data $\omega_0\in L^\infty$ (those satisfying $\omega_0\in\mathcal{O}_{\xi}$) such that the Euler solution $S_t(\omega_0)$ is not ergodic on $\mathcal{O}_{\omega_0}\cap \{ {\mathsf E}= {\mathsf E}_0\}$ where $\mathcal{O}_{\omega_0}$ is the orbit of $\omega_0$ in the group of area preserving diffeomorphisms:
\be
\mathcal{O}_{\omega_0}:= \{ \omega_0 \circ \varphi \ :\ \varphi  \in \mathscr{D}_\mu(M)  \},
\ee
and $ \{ {\mathsf E}= {\mathsf E}_0\}$ is the collection of vorticity fields with equal energy.   This set is the natural phase space for a given datum $\omega_0$, and incorporates all known conservation laws of the Euler equation.  Failure of ergodicity in this sense is relevant to the applicability of statistical hydrodynamic theories of long term behavior, see further discussion in \S \ref{conjectent}.  We remark that there have been recent developments establishing wandering (rather, non-approaching) in 3d Euler \cite{KKPS14,KKPS20}.

\begin{proof}
 We first construct the field $\xi$.  Let  $v(x) = (x_2,0)$ be Couette flow (any $L^2$ stable shear flow can generate such an example). 
Note that  the field is solenoidal $\nabla \cdot v=0$, tangent to the boundary $v\cdot \hat{n}|_{\partial M}=0$, with circulations $\oint_{\{y=1\}} v \cdot \rmd \ell=2\pi$ and $\oint_{\{y=0\}} v \cdot \rmd \ell=0$ and has vorticity $\eta:= \nabla^\perp\cdot v=-1$.

\begin{figure}[h!]
\label{fig:Nad}
\centering
		\begin{tikzpicture}[scale=0.8, every node/.style={transform shape}]
		\draw [thick] (-3.5,0)--(3.5,0);
		\draw [thick] (-3.5,3)--(3.5,3);
        \draw [dashed] (-3.5,0)--(-3.5,3);
        \draw [dashed] (3.5,0)--(3.5,3);
   \draw[very thick,red] (-1.75,0)--(-1.75,3);
		\draw[fill, black!30!white, opacity=0.5] (0,0)--(3.5,0)-- (3.5,3)--(0,3)--cycle;
         \draw [fill, red!50!white] plot [smooth cycle] coordinates {(1,0.8) (1.5,1) (2,1.65) (2.5,.5)};
         \draw[thin, blue, ->] (0,.5)--(.5,.5);
                  \draw[thin, blue, ->] (0,1)--(1,1);
         \draw[thin, blue, ->] (0,1.5)--(1.5,1.5);
         \draw[thin, blue, ->] (0,2)--(2,2);
          \draw[thin, blue, ->] (0,2.5)--(2.5,2.5);
          \draw[very thick, blue, ->] (0,3)--(3,3);
		\draw[thin, ->] (-3.5,0)--(3.5,0);
		\draw[thin, ->] (0,0)--(0,3);	
        \draw  (-1.75,1) node[anchor=east] {\Large $\textcolor{red}{\gamma}$};
        \draw  (3.2,1.5) node[anchor=east] {\Large $B$};
        \draw  (2,1) node { ${G}$};
       \draw  (0,2.8) node[anchor=east] {$y$};
	\draw  (3.3,0) node[anchor=north] {$x$};
\end{tikzpicture}
\caption{Setup for Nadirashvili's example. Here $g$ could be supported on $G$ and the blue lines represent the velocity field of the Couette flow. The set $B$ is the grey region.}
\label{fig:tildeomega}
\end{figure}
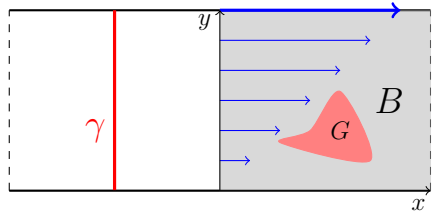
Consider the sets (see Figure \ref{fig:tildeomega})
\begin{equation}
B= [0,\pi]\times[0,1], \qquad  \gamma=\{-\pi/2\}\times [0,1].
\end{equation} 
 Let  $h\in C^\infty(M)$ be any scalar function $h: M\to \mathbb{R}$ satisfying
\begin{equation}\label{hprop}
h|_{\gamma}\geq 1/2, \qquad h|_{B}=0, \qquad  \lVert h\rVert_{L^\infty}\leq 1.
\end{equation}
Thus, $h$ is supported entirely in the white region of Figure \ref{fig:tildeomega}, and is non-trivial on the line $\gamma$.
For $ \delta \ll 1$, define $\xi\in C^\infty(M)$ as a perturbation of the Couette vorticity
\begin{equation}
\xi := \eta + \delta h.
\end{equation}
It follows that $\xi|_B =-1$.
For any $g\in L^\infty(M)$ such that $\|g\|_{L^\infty}<1$ and let 
\begin{equation}\label{om0def}
\omega_0=\xi+ \ve g, \quad \text{for} \ \ \ve \ll 1.
\end{equation}
For instance, $g$ could be supported in the region $G\subset B$ of Fig. \ref{fig:tildeomega}. It follows by \eqref{om0def}, \eqref{hprop} and $\|g\|_{L^\infty}<1$ that we have
\be\label{om0online}
\|\omega_0 - \xi\|_{L^\infty} \leq\ve, \qquad \|\omega_0 - \eta\|_{L^\infty} \leq\ve +\delta, \qquad   \omega_0|_{\gamma} > \eta +\delta/2-\ve.
\ee
Consider $\omega(t):=S_t(\omega_0)$.  The perturbed vorticity  $\tilde{\omega}:=\omega-\eta$ is transported:
\begin{align}
\partial_t \tilde{\omega}+ u\cdot \nabla \tilde{\omega}&=0,\\
\tilde{\omega}|_{t=0} = \delta h+ \ve g,
\end{align}
where $u= K_M[\omega]$.
Thus, we find that
\be
\|\omega(t)-\eta\|_{L^\infty} \leq \ve +\delta ,
\ee 
showing that Couette flow (as any constant vorticity solution) is stable in $L^\infty$ topology for vorticity.
We now appeal to the following elementary Lemma:

\begin{lemma} \label{lilem}  Let $M=  [-\pi, \pi) \times [0,1]$ and 
 $u\in C^{1}(M)$ be divergence-free $\nabla \cdot u=0$ with $u\cdot \hat{n}|_{\partial\Omega}=0$ and with $\oint_{\{y=0\}} u \cdot \rmd \ell=0$.   Then for some $C:= C(M)>0$
\be
\| u\|_{L^\infty} \leq C \|\nabla^\perp\cdot u\|_{L^\infty}.
\ee
\end{lemma}

In fact, the stronger result that $\| u\|_{C^\alpha}\leq  C(\alpha) \|\nabla^\perp\cdot u\|_{L^\infty}$ for all $\alpha\in [0,1)$ holds by Schauder theory (see e.g. \cite{gt}).
As a consequence of Lemma \ref{lilem}, 
\be\label{linfbnd}
\|u(t)-v\|_{L^\infty} \leq C ( \ve +\delta)
\ee
since by Galilean transformation, a small constant velocity can be removed to ensure the circulation on the bottom boundary be zero.  Now we take $\delta$ and $\ve$ sufficiently small so that 
\be\label{linfbnd}
\|u(t)-v\|_{L^\infty} \leq 1/4.
\ee
This velocity field $u(t)$ defines an area preserving diffeomorphism $\Phi_t$ of the periodic channel with $\Phi_0={\rm id}$. Since the velocity is tangent to the walls $u_2(t)|_{y=0,1}=0$, $\Phi_t$ maps the boundaries to themselves.  Note that
from \eqref{linfbnd} we deduce
\be\label{veldiff}
u_1(t)|_{x_2=0} \leq  1/4, \qquad u_1(t)|_{x_2=1} \geq 3/4
\ee
so that the smallest relative velocity is $\min_{\{y=1\}} u_1(t) - \max_{\{y=0\}} u_1(t)\geq1/2$ for all $t$.
We now consider the image of the line $\gamma$ under the flow $\gamma(t):= \Phi_t(\gamma)$.  Due to the relative speeds along the top and bottom boundary, we find \eqref{veldiff}
\be
{\rm dist}(\gamma(t)|_{x_2=1}, \gamma(t)|_{x_2=0})\geq  t/2.
\ee
\begin{figure}[h!]
\label{fig:evolution}
\centering
		\begin{tikzpicture}[scale=0.9, every node/.style={transform shape}]
		\draw [thick] (-3.5,0)--(3.5,0);
		\draw [thick] (-3.5,3)--(3.5,3);
        \draw [dashed] (-3.5,0)--(-3.5,3);
        \draw [dashed] (3.5,0)--(3.5,3);
		\draw[fill, black!30!white, opacity=0.5] (0,0)--(3.5,0)-- (3.5,3)--(0,3)--cycle;
   \draw[very thick,red, dashed] (-2.5,0)--(3.5,1.91);
   \draw[very thick,red, dashed] (-3.5,1.91)--(-0.1,3);
         \draw [red, very thick] plot [smooth, tension=2.1] coordinates {(-2.5,0) (-1,.75) (0,.35) (1,1) (2,1.78) (3.5,2.25) };
         \draw [red, very thick] plot [smooth, tension=1.7] coordinates {(-3.5,2.25) (-2.5,2.75) (-2,2) (-1,1.85) (-.5,2.5) (-.1,3) };
         \draw [fill, red!50!white, ] plot [smooth cycle] coordinates {(1.2,.9) (1.5,.5) (2,1) (2.5,.2) (3.2,.8) (2.7, 1.2) (2.5, .8) (2.1,1.1) (1.75,1) (1.5, .75) (1.2,.9)};
         \draw  (2.7,.6) node { ${G(t)}$};
          \draw[very thick, blue, ->] (-0.1,3)--(1.6,3);
          \draw[very thick, blue, ->] (-2.5,0)--(-1.5,0);
		\draw  (-2.5,1) node { \Large \textcolor{red}{${\gamma_1(t)}$}};
    		\draw  (3.2,2.5) node[anchor=east] {\Large $B$};
				\end{tikzpicture}
\caption{Representation of a possible evolution $\gamma(t)$. Here the dashed red line is the evolution of the line $\gamma$ if the problem were linear (with velocity $3/4$ at the top and $1/4$ at the bottom, the worst case scenario). A possible region $G(t):= \Phi_t(G)$ is also shown.}
\label{fig:imagewrap}
\end{figure}
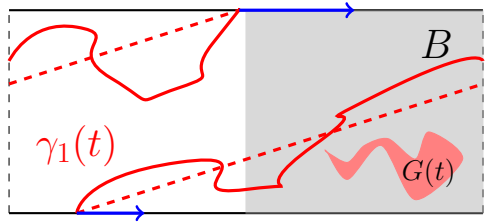
Combined with the fact that the domain is periodic in $x\in [-\pi,\pi)$, this shows that for large enough time we have that the curve must partially occupy the region $B$ forever after:
\be\label{intersectionprop}
\gamma(t) \cap B \neq \emptyset  \quad \text{for all} \ t> 2\pi.
\ee
See  Figure \ref{fig:imagewrap} for a depiction of a possible evolution.
It is enough to wait until the image of the curve wraps once around the torus -- the longer time elapses, the more the image will become wrapped. This fact can be clearly seen by lifting to a covering space of the cylinder, see Figure \ref{fig:imagewrapcover}.

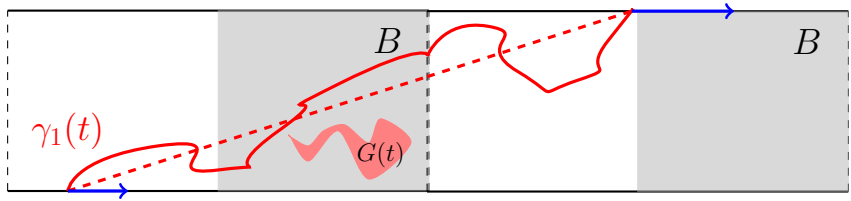
\begin{figure}[h!]
\centering
		\begin{tikzpicture}[scale=.8, every node/.style={transform shape}]
		\draw [thick] (-3.5,0)--(3.5,0);
		\draw [thick] (-3.5,3)--(3.5,3);
        \draw [dashed] (-3.5,0)--(-3.5,3);
        \draw [dashed] (3.5,0)--(3.5,3);
		\draw[fill, black!30!white, opacity=0.5] (0,0)--(3.5,0)-- (3.5,3)--(0,3)--cycle;
   \draw[very thick,red, dashed] (-2.5,0)--(3.5,1.91);
         \draw [red, very thick] plot [smooth, tension=2.1] coordinates {(-2.5,0) (-1,.75) (0,.35) (1,1) (2,1.78) (3.5,2.25) };
\draw [fill, red!50!white, ] plot [smooth cycle] coordinates {(1.2,.9) (1.5,.5) (2,1) (2.5,.2) (3.2,.8) (2.7, 1.2) (2.5, .8) (2.1,1.1) (1.75,1) (1.5, .75) (1.2,.9)};
\draw  (2.7,.6) node { ${G(t)}$};
          \draw[very thick, blue, ->] (-2.5,0)--(-1.5,0);
		\draw  (-2.5,1) node { \Large \red{${\gamma_1(t)}$}};
    		\draw  (3.2,2.5) node[anchor=east] {\Large $B$};
				\end{tikzpicture}
						\hspace{-.022\textwidth}		
		\begin{tikzpicture}[scale=.8, every node/.style={transform shape}]
		\draw [thick] (-3.5,0)--(3.5,0);
		\draw [thick] (-3.5,3)--(3.5,3);
        \draw [dashed] (-3.5,0)--(-3.5,3);
        \draw [dashed] (3.5,0)--(3.5,3);   
		\draw[fill, black!30!white, opacity=0.5] (0,0)--(3.5,0)-- (3.5,3)--(0,3)--cycle;      
   \draw[very thick,red, dashed] (-3.5,1.91)--(-0.1,3);
         \draw [red, very thick] plot [smooth, tension=1.8] coordinates {(-3.5,2.25) (-2.5,2.75) (-2,2) (-1,1.85) (-.5,2.5) (-.1,3) };
          \draw[very thick, blue, ->] (-0.1,3)--(1.6,3);
    		\draw  (3.2,2.5) node[anchor=east] {\Large $B$};
				\end{tikzpicture}
\caption{The unwrapped evolution on the covering space of the cylinder.}
\label{fig:imagewrapcover}
\end{figure}
To complete the proof, we note that since vorticity is transported $\omega(t)= \omega_0\circ \Phi_t^{-1}$, we have from \eqref{om0online} that 
\be
\omega(t)|_{\gamma(t)} > \eta+ \delta/2-\ve.
\ee
It follows from \eqref{intersectionprop} that for all $t>2\pi$, the above implies that
\be
\|\omega(t)-\xi\|_{L^\infty} \geq \|\omega(t)-\xi\|_{L^\infty(B)} > \delta/2-\ve,
\ee
since $\xi|_B=\eta =-1$.  For any $\ve< \delta/4$,  $\|\omega(t)-\xi\|_{L^\infty}>\ve$ establishing the claim.
\end{proof}

We should remark that two distinguishing features of Nadirashvili's setup is that the domain of the fluid has boundary and a non-trivial fundamental group. Elgindi and Jeong \cite{EJ20} showed in the case of the trivial fundamental group, when boundary is present, the
boundary of a large chunk of a vortex patch can play the role of an essential curve (see Figure \ref{fig:patch}) and an infinite spiral can form.  To show this, one must use the fact that the velocity vanishes at the corner of the domain and so particles cannot pass through to different sides of the square's boundary. This idea can establish existence of wandering points on simply connected domains with corners. 
\begin{figure}[h!]
\centering
  \includegraphics[height=3in]{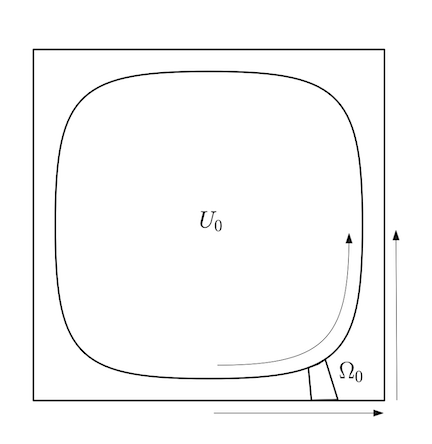}
  \caption{Spiral formation in the case of the square (Fig. 4 of \cite{EJ20}).}
  \label{fig:patch}
\end{figure}

On $M=\mathbb{T}^2$, 
Shnirelman
proved that a \textit{typical}
trajectory of the Euler equation is wandering.  He works with Lagrangian description rather than the Eulerian, rephrasing Euler with initial data $\omega_0$ as
\be\label{Lagrangiandyn}
\frac{\rmd}{\rmd t} (\Phi,{\omega_0}) =\left( K_{\mathbb{T}^2}[{\omega_0}\circ \Phi^{-1}]\circ \Phi, 0\right),
\ee
where we used ${\omega}(t):= \omega_0\circ \Phi^{-1}$ and where $K_{\mathbb{T}^2}$ is the Biot-Savart law on $\mathbb{T}^2$.  Rather than working in H\"{o}lder or Sobolev spaces, Shnirelman works with
 Besov spaces, $B^s_{2,\infty}$.  Note that $H^s \subset B^s_{2,\infty}\subset H^{s-}$ for any $s>0$. It is proved that $X^s:=B^s_{2,\infty}\times B^{s-1}_{2,\infty}$ defines a suitable phase-space for the pair $(\Phi-{\rm id},\omega_0)$ provided $s>3$ for all time. The main result is

\begin{myshade}
  \vspace{-1mm}
\begin{theorem}[Shnirelman \cite{S97}]\label{thmwand2}  
For $s > 3$, there exists an open and dense set $Y^s\subset X^s$,
such that each $(\Phi,\omega_0)\in Y^s$ is contained in a wandering neighborhood.
\end{theorem}
\end{myshade}

The main difference between the above result and the one of Nadirashvili is that Shnirelman's theorem takes place at the level of particle configurations $\Phi$, not velocity fields $u:=K_{\mathbb{T}^2}[{\omega_0}\circ \Phi^{-1}]$.  Seemingly, a wandering neighborhood in configuration space (at the level of the Lagrangian flowmap) is not inconsistent with non-wandering in phase space (at the level of the vorticity field), and so the results are complementary.  Finally, the dynamics are not uniformly-in-time bounded in the $X^s$ topology, unlike the $L^\infty$ phase space of Theorem \ref{thmwand1}.

  \begin{myshade3}
  \vspace{-1mm}
\begin{problem}\label{wanderq}
Show that there exist wandering neighborhoods in the phase space $L^\infty$ of vorticity on compact domains without boundary (e.g. $M=\mathbb{T}^2$ or $\mathbb{S}^2$).
\end{problem}
\end{myshade3}
 These domains should be somehow generic, in the spirit of Theorem \ref{thmwand2}.

  \subsection{Strife: instability and infinite-time blowup}\label{strifesec}

As discussed in \S \ref{isolation}, Arnold established a form of dynamical stability in his works \cite{arn11,arn12}.  Namely, if $\omega_*$ is a stationary solution satisfying one of two so-called Arnold stability conditions, then it is Lyapunov stable in the $L^2$ topology (Theorems \ref{thmArn}). This can be promoted to a true dynamical stability in a suitable phase space $X$ (Theorem  \ref{stabinX} herein).  
On the other hand, in stronger topologies such as $C^\alpha(M)$ or $H^s(M)$ for $s>0$, stationary solutions -- especially those identified by Arnold -- are known to be strongly unstable.  In fact, a behavior which is believed to be typical is a gradual \emph{deterioration of smoothness} or \emph{infinite-time blowup}.  Namely, solutions can have the property that $\|\omega(t)\|_{C^{\alpha}(M)} \to \infty$ as $t\to\infty$ thereby leaving the $C^{\alpha}$ phase space at infinite time.  Physically, such roughening is the result of a process of the stretching, folding and filamenting of vortex lines by coherent velocity structures and is the signature of the direct cascade of enstrophy to small scales, see Fig. \ref{fig1}.

In this direction, the Nadirashvili's argument \cite{N} of wandering  in Theorem \ref{thmwand1} can be modified to provide a simple example of sustained growth.  The construction is on the periodic channel for a constant vorticity steady Euler solution $\omega_*$.  Such a state is nonlinearly stable has the property that its velocity field exhibits a fixed `` gap," i.e. the particles on top boundary $\{x_2=1\}$ are moving strictly faster than those moving along the bottom $\{x_2=0\}$. Below we state the result for shear flows on $M = \mathbb{T}\times[0,1]$ but this is non-essential: constant vorticity steady states on a fluid vessel which is topologically annular would work.

\begin{myshade}
  \vspace{-1mm}
\begin{theorem}[Growth in $C^\alpha$ near constant vorticity states]\label{constvortinst}
Let $M = \mathbb{T}\times[0,1]$ and $\alpha>0$.
Let $\omega_*$ be an constant-vorticity steady Euler  solution  with
\be
u_*^{\mathsf{t}}:=\min_{x_1\in \mathbb{T}}u_*\big|_{x_2=1} >\max_{x_1\in \mathbb{T}}u_*\big|_{x_2=0} =:u_*^{\mathsf{b}}.
\ee
For any $0<\ve\ll 1$,  there exists $\xi\in  C^{\infty} (M)$ with $\|\xi - \omega_*\|_{C^{\alpha} (M)}<\ve $, a time $T:=T(\ve, M,u_*^{\mathsf{t}}-u_*^{\mathsf{b}})$ and a constant  $C:=C(\xi, M,u_*^{\mathsf{t}}-u_*^{\mathsf{b}})>0$ such that 
\begin{align*}
  \| S_t(\gamma_0)-\xi\|_{C^\alpha(M)} \geq \ve C t^\alpha \qquad &\text{for all} \qquad \|\gamma_0- \xi\|_{C^{\alpha} (M)}<\ve.
\end{align*}
\end{theorem}
\end{myshade}

\begin{proof}
The setup is along the same lines as the proof of Theorem \ref{thmwand1}.  For simplicity set  $u_*^{\mathsf{t}}=1$ and $u_*^{\mathsf{b}}=0$.
Define a pair of curves traversing the channel
\begin{equation}
 \gamma_1=\{-\pi/2\}\times [0,1], \qquad \gamma_2=\{-\pi/4\}\times [0,1].
\end{equation}
Let
\begin{equation}\label{hprop}
h|_{\gamma_1}\geq 1/2, \qquad h|_{\gamma_2}\leq 1/4,  \qquad  \lVert h\rVert_{C^\alpha}\leq 1.
\end{equation}
Set 
$
\xi := \omega_*+ \ve h$ for  $0<\ve \ll 1$.
We take initial data
\begin{equation}\label{om0def}
\gamma_0=\xi + \ve g,
\end{equation}
where $g\in C^\alpha(M)$ with $\|g\|_{L^\infty}<1/4$. 
Define the rectangle delimited by $\gamma_1,\gamma_2$ as 
\begin{equation}
\label{def:A}
A=[-\pi/2,-\pi/4]\times [0,1].
\end{equation}

\begin{figure}[h!]
\label{fig:Nadgrowth}
\centering
		\begin{tikzpicture}[scale=0.9, every node/.style={transform shape}]
		\draw [thick] (-3.5,0)--(3.5,0);
		\draw [thick] (-3.5,3)--(3.5,3);
        \draw [dashed] (-3.5,0)--(-3.5,3);
        \draw [dashed] (3.5,0)--(3.5,3);
   \draw[very thick,red] (-1.75,0)--(-1.75,3);
       		\draw  (-1.75,1) node[anchor=east] {\Large $\red{\gamma_1}$}; 		
   \draw[very thick,green!80!blue] (-.875,0)--(-.875,3);
       		\draw  (-.875,1) node[anchor=west] {\Large $\green{\gamma_2}$};
         \draw [fill, red!50!white] plot [smooth cycle] coordinates {(1,0.8) (1.5,1) (2,1.65) (2.5,.5)};
\draw  (2,1) node { ${G}$};
\draw  (-1.3,1) node { \Large ${A}$};
         \draw[thick, blue, ->] (0,.5)--(.5,.5);
                  \draw[thick, blue, ->] (0,1)--(1,1);
         \draw[thick, blue, ->] (0,1.5)--(1.5,1.5);
         \draw[thick, blue, ->] (0,2)--(2,2);
          \draw[thick, blue, ->] (0,2.5)--(2.5,2.5);
          \draw[thick, blue, ->] (0,3)--(3,3);
		\draw[thin, ->] (-3.5,0)--(3.5,0);
		\draw[thin, ->] (0,0)--(0,3);	
		\draw  (0,2.8) node[anchor=east] {$y$};
		\draw  (3.3,0) node[anchor=north] {$x$};
				\end{tikzpicture}
\caption{Setup for example of growth. Here $g$ could e.g. be supported on $G$ and $\xi$ has nontrivial support on $\gamma_1$ and $\gamma_2$. The set $A$ is the area contained between $\gamma_1$ and $\gamma_2$. }
\end{figure}
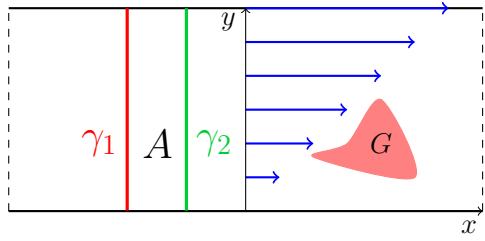

We claim that the distance between the images of the two lines $\gamma_1$ and $\gamma_2$ under the Lagrangian flow $\Phi_t$ decreases.  This will be used to deduce growth of the $C^\alpha$ norms since the curves carry different values of vorticity (since the background is constant).

\begin{myshade}
  \vspace{-1mm}
\begin{lemma}\label{distlem}
Let $\gamma_i(t) = \Phi_t(\gamma_i)$ for $i=1,2$ where $\Phi_t$ is the Lagrangian flow corresponding to the solution with initial vorticity $\gamma_0$. Then
\be\label{distresult}
{\rm dist}(\gamma_1(t),\gamma_2(t))\leq \frac{8A}{t}.
\ee
\end{lemma}
\end{myshade}

\begin{proof}
First note that for any $p>2$, interpolation and $L^2$-stability yields 
\be
\| u(t) - u_*\|_{L^\infty}  \leq C \| \omega (t) - \omega_*\|_{L^{p}} \leq 2 C \| \omega(t) - \omega_*\|_{L^{\infty}}^{\frac{p-2}{p}}  \| \omega(t) - \omega_*\|_{L^{2}}^{\frac{2}{p}}.
\ee
By choosing $\ve\ll1$ sufficiently small,  it follows that  for all $t\in \mathbb{R}$,
\be\label{veldiff}
u_1(t)|_{x_2=0} \leq  1/4, \qquad u_1(t)|_{x_2=1} \geq 3/4.
\ee
Letting $x_i^\mathsf{t}(t) =\gamma_i(t)|_{x_2=1}$ and $x_i^\mathsf{b}(t) =\gamma_i(t)|_{x_2=0}$,  on the covering space of $M$ (see Figure \ref{fig:imagewrapcover}) we have the ordering
\be
x_1^\mathsf{b}(t) < x_2^\mathsf{b}(t) <\frac{t}{4} \qquad \text{and} \qquad x_2^\mathsf{t}(t) > x_1^\mathsf{t}(t) >\frac{3t}{4}.
\ee
Assume now for contradiction that the result \eqref{distresult} is false. 
Fix $x\in (\frac{1}{4}t, \frac{3}{4}t)$. Let $y_i(x)$ be the highest \emph{$x_2$-value} of any point on the line $x=x_0$ intersecting $\gamma_i(t)$. By continuity and by the assumption, the vertical line $[y_1(x), y_1(x)+\frac{8A}{t}]$ is in the domain. See e.g. Figure \ref{fig:absurd}. On the other hand, by volume preservation and Fubini's theorem, we arrive at a contradiction 
\[|A|\geq \int_{t/4}^{3t/4}\int_{0}^1 \chi_A  \geq \int_{t/4}^{3t/4}\frac{8A}{t}dx>2|A|.
\vspace{-6mm}
\] 
\end{proof}

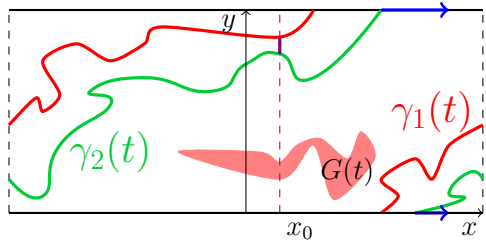
\begin{figure}[h!]
\centering
		\begin{tikzpicture}[scale=0.9, every node/.style={transform shape}]
		\draw [thick] (-3.5,0)--(3.5,0);
		\draw [thick] (-3.5,3)--(3.5,3);
        \draw [dashed] (-3.5,0)--(-3.5,3);
        \draw [dashed] (3.5,0)--(3.5,3);
         \draw [red, very thick] plot [smooth] coordinates {(2,0) (2.25,.25) (2,.5) (2.5,.8) (2.75,.5) (3,1) (3.5,1.3)};
         \draw [red, very thick] plot [smooth] coordinates {(-3.5,1.3) (-3.1,1.7) (-3,1.5) (-2.75,1.75) (-3,2.25) (-2.5,2.5) (-2,2.25) (-1.5, 2.5) (-1.25, 2.75) (0.5, 2.6)  (1,3) };
         \draw [green!80!blue, very thick] plot [smooth, tension=1] coordinates {(2.5,0) (3.2,.2) (3,.3) (3.3,.6) (3.5,.5)};
         \draw [green!80!blue, very thick] plot [smooth, tension=1] coordinates {(-3.5,.5) (-3,.25) (-2.75,1) (-2,1.6) (-2.3,1.8) (-1.5,2) (-.5, 1.8) (0,2.2) (.5,2.35) (1,2) (2,3) };   
\draw [fill, red!50!white, ] plot [smooth cycle] coordinates {(-1,.9) (.5,.5) (1,1) (1.3,.2) (1.9,.8) (1.7, 1.2) (1.5, .8) (1.1,1.1) (.75,1) (.5, .75) (.2,.9)};
\draw  (1.5,.6) node { ${G(t)}$};
          \draw[very thick, blue, ->] (2,3)--(3,3);
          \draw[very thick, blue, ->] (2.5,0)--(3,0);
          \draw[dashed, blue!30!red] (.5,0)--(.5,3);
               \draw[very thick, blue!50!red] (.5,2.35)--(.5,2.6);     
		\draw  (2.75,1.5) node { \Large \red{${\gamma_1(t)}$}};
        		\draw  (-2,.9) node { \Large \green{${\gamma_2(t)}$}};
\draw[thin, ->] (-3.5,0)--(3.5,0);
\draw[thin, ->] (0,0)--(0,3);	
\draw  (0,2.8) node[anchor=east] {$y$};
\draw  (3.3,0) node[anchor=north] {$x$};
\draw  (.8,0) node[anchor=north] {$x_0$};
				\end{tikzpicture}
				\caption{Cartoon of the evolution spiraling around the torus.  Dashed purple vertical line indicates the point where the distance between the lines is smallest. }
				\label{fig:absurd}
\end{figure}

With Lemma \ref{distlem}, the result follows after noting that for some $c>0$ we have
 \begin{equation}
 \|\omega(t)\|_{C^\alpha}\geq \frac{|\omega(t,p_1(t))-\omega(t,p_2(t))|}{|p_1(t)-p_2(t)|^\alpha}\geq  \frac{c\ve  t^\alpha}{(8A)^\alpha}.
 \end{equation}
where $\omega(t):=S_t(\gamma_0)$  and $p_1(t)$ and $p_2(t)$ are the points on the lines $\Phi_t(\gamma_1)$ and $\Phi_t(\gamma_2)$ which realize the distance \eqref{distresult} (see the purple lines in Figure \ref{fig:absurd}).
\end{proof}

There has been subsequent work exhibiting this type of behavior (often stated for the gradient of the vorticity) on domains with the boundary  \cite{BM15,IJ20,Y74,Y00,N,KS14,serre99}, on the torus \cite{Denisov1,Z15} near stationary solutions and the plane near a vortex dipole \cite{CJ21a}.    The works  \cite{Y74,Y00,N} provide linear-in-time growth rates whereas the others provide superlinear growth rates, with \cite{KS14} on the domain of a disk standing out as the sole example of double-exponential growth from smooth data of the vorticity gradient, saturating the bound \eqref{dexp}.  The result is:

\begin{myshade}
  \vspace{-1mm}
\begin{theorem}
[Kiselev-Šverák \cite{KS14}]\label{KSthm}
 Let $ M$ be the unit disk.  There exists smooth initial datum $\omega_0$ with $\|\nabla \omega_0\|_{L^\infty} /\|\omega_0\|_{L^\infty}>1$ such that the Euler solution $\omega(t)= S_t(\omega_0)$ satisfies for all $ t\in \mathbb{R}$ the lower bound
\be\label{dexplb}
\frac{\| \nabla \omega(t)\|_{L^\infty( M)}}{\|\omega_0\|_{L^\infty}}  \geq \left(\frac{\| \nabla \omega_0\|_{L^\infty( M)}}{\|\omega_0\|_{L^\infty( M)} }\right)^{c \exp\left({c\|\omega_0\|_{L^\infty( M)}  |t|}\right)}
\ee
for some constant $c>0$.
\end{theorem}
\end{myshade}
\begin{remark}
The original paper of Kiselev-Šverák \cite{KS14} finds data for which \eqref{dexplb} holds for all $t>0$.   The case of all $t\in \mathbb{R}$ follows by a simple modification.
\end{remark}

 In generality (nearby arbitrary steady states), little is known about  infinite time blowup.  However, the following theorem due to Koch \cite{K02} establishes a strong form of nonlinear instability:

 \begin{myshade}
  \vspace{-1mm}
\begin{theorem}[Koch \cite{K02}]\label{thminstab}  
Every stationary solution $\omega_*\in C^{1,\alpha}$ of the two dimensional
Euler equation whose Lagrangian flow is not periodic in time is nonlinearly unstable in $C^{\alpha}$.   Specifically, for all $M$ and $\ve$, there exists a time $T:=T(M,\ve)$  and a solution $\omega(t)=S_t(\omega_0)$ with the property that 
\be\label{nlstab}
\| \omega_0 - \omega_*\|_{C^{\alpha}} \leq \ve , \qquad while \qquad \| \omega(T) - \omega_*\|_{C^{\alpha}} \geq M.
\ee
\end{theorem}
\end{myshade}

The idea of the proof is that the vorticity of a perturbation is transported by the volume preserving flow $\Phi_t$ up to a lower order (compact) term.  If the gradient of the corresponding Lagrangian flow is large at some point, then derivatives of the perturbation can be large.    The gradient of the Lagrangian flow of the steady state $\omega_*$ grows unboundedly large unless it is periodic in time -- i.e.  $\omega_*$ is the vorticity of a so-called  \emph{isochronal} \cite{Y00,MSY}.  Towards a contradiction, stability in the $C^\alpha$ topology of the vorticity gives Lipschitz stability of the flowmap.  As such, unbounded growth of the base flow means nearby solutions have large stretching factor at a fixed late time.  The perturbation is designed to exploit this stretching, yielding instability. It would be interesting to extend Theorem \ref{thminstab} to \emph{supercritical} (in the sense that the dynamics are uncontrolled) spaces. 

  \begin{myshade3}
  \vspace{-1mm}
\begin{problem}\label{kochquest}
Let $X$ be any space compactly embedded in $L^2$. Show that  any smooth stationary solution $\omega_*$ of the two dimensional
Euler equation whose Lagrangian flow is not periodic in time is nonlinearly unstable in $X$.   
\end{problem}
\end{myshade3}
It would be interesting to prove first the above statement for $X=H^s$ with $s>0$.  We remark that the analogous instability statement in 3D would be for $X$ be any space compactly embedded in $L^2$ of velocity (rather than vorticity) and it appears to be much harder due to the lack of control on the dynamics there.

\begin{figure}[h!]
	\centering
	\begin{tikzpicture}[scale=1.5, every node/.style={transform shape}]
       \draw[ultra thick, black] plot [smooth cycle, tension=1.2] coordinates {(0,-2) (1.5,0) (0,2) (-1.5,0)}; 
    
        \draw[thick, ->, decoration={markings, mark=at position 0.625 with {\arrow{>}}},
        postaction={decorate}] plot [smooth cycle, tension=1.3] coordinates {(0,-1.33) (1,.2) (0,1.33) (-1,.1)}; 
        \draw[thick, ->, decoration={markings, mark=at position 0.625 with {\arrow{>}}},
        postaction={decorate}]   plot [smooth cycle, tension=.8] coordinates {(0.1,-.665) (.5,0) (0.05,.665) (-.5,0)}; 
        
      \draw[ultra thick, black, ->] (0,0)--(1.2,1);
            \draw[ultra thick, black, ->] (0,0)--(.25,.8);
            
       \draw[ultra thick, red, ->] plot [smooth, tension=1] coordinates {(0,0) (-.25,-.25) (-.4,0) (-.6,-.3)}; 
        \draw[ultra thick, red, ->] plot [smooth, tension=1] coordinates {(0,0) (-.3,.6) (-.7,.2) (-1.2,.6)};
        
               \draw[ultra thick, blue!80!white, ->] plot [smooth, tension=1] coordinates {(0,0) (.2,-.4) (-.2,-1) (.2,-1.5)}; 
               \draw[ultra thick, blue!80!white, ->] plot [smooth, tension=1] coordinates {(0,0) (.5,-.1) (.6,-.4) (.8,-.3)}; 
               
               
               \draw  (-4,1.5) node[anchor=west] {\Large${t=0}$};
               \draw  (-4,.8) node[anchor=west] {\Large$\textcolor{red}{t=\frac{1}{3}T}$};
               \draw  (-4,.1) node[anchor=west]{\Large$\textcolor{blue!80!white}{t=\frac{2}{3}T}$};
               \draw  (-4,-.6) node[anchor=west]{\Large$t=T$};
\end{tikzpicture}
  \caption{A fluid ``clock" which is right at least twice a day: isochronal flow with period $T=12$ hrs.  Clock hands are made of Lagrangian particles seen at the different times.}
  \label{clock}
\end{figure}
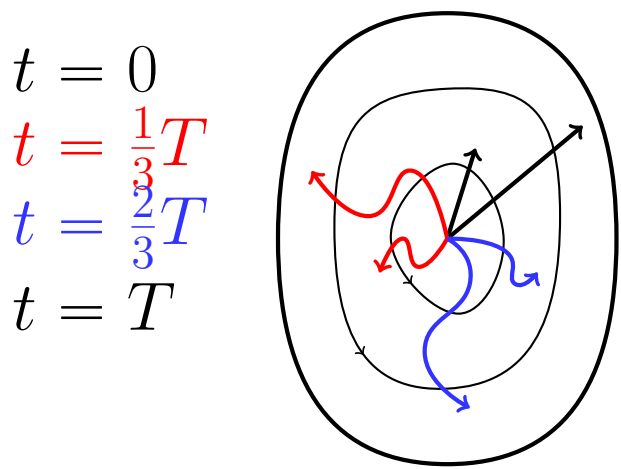

The assumption of the Lagrangian flowmap being non-periodic is not always necessary.  For example,  the trivial solution has periodic flowmap but is clearly is unstable in the sense of Theorem \ref{thminstab}.  We call those velocity fields with periodic flowmaps isochronal:

  \begin{myshade}
  \vspace{-1mm}
\begin{definition}
We say that a Hamiltonian $\psi: M\to \mathbb{R}$ or its corresponding Hamiltonian vector field $u=\nabla^\perp \psi: M\to \mathbb{R}^2$ is \emph{isochronal} if the flowmap $\Phi_t$ it generates via $\dot{\Phi}_t = u\circ \Phi_t$ is time periodic.  
\end{definition}
\end{myshade}

An interesting family of isochronal flows are elliptical vortices with constant vorticity, see \cite{MSY}.
Specifically, let $M$ be an elliptical vessel with major axis $a>0$ and minor axis $b>0$. This family is defined by the streamfunctions 
\be\label{elliptical}
\psi(x_1,x_2) = \frac{1}{2} \left(\left(\frac{x_1}{a}\right)^2+ \left(\frac{x_2}{b}\right)^2\right)
\ee
 which correspond to Euler solutions having constant vorticity 
$\Delta \psi =  {(a^2 + b^2)}/{a^2b^2}$.  One member of this family is the radial solid-body rotation vortex discussed in Remark \ref{remfail}. Having constant vorticity, these steady states are stable in $L^\infty$.
The travel time \eqref{rotation} $\mu(c) =2\pi ab$  along each streamline $\{\psi = c\}$  is manifestly independent of $c$ and thus these flows are isochronal. 

Although Theorem \ref{thminstab} does not directly apply, it can be shown that members of the family of elliptical vortices are not isolated (see Definition \ref{isodef}) from non-isochronal steady states in the $C^\alpha$ topology on vorticity \cite{MSY}.  As such, they too are nonlinearly unstable.  It seems reasonble to conjecture that \emph{all} steady states are unstable in $C^\alpha$.  
A natural starting point for addressing this question would be to characterize Euler flows having this isochronal property.  In certain cases, this can easily accomplished
\begin{myshade}
  \vspace{-1mm}
\begin{lemma}\label{isolem}
Let $M = \{ x  \ : \ |x|\leq 1\}$ be the disk.  There exists a unique smooth isochronal Euler flow (up to scaling), corresponding to solid body rotation. 
\end{lemma}
\end{myshade}
\begin{proof}
If $\psi : M \to \mathbb{R}$ is an isochronal streamfunction for a Hamiltonian vector field on $\overline{M}$, then it possesses exactly one critical point (otherwise, there will be a curve along which the velocity $u=\nabla^\perp \psi$ vanishes, and so the period would need to be zero making $u$ is everywhere trivial). Being a stationary solution of Euler imposes $\nabla^\perp\psi \cdot \nabla \Delta \psi=0$ in $M$  (vorticity equation) and $\partial_\tau \psi|_{\partial M}=0$ (non-penetration condition).  Since $\psi\in C^2(M)$ has a single critical point, there is a $C^1$ function $F: \mathbb{R}\to \mathbb{R}$ such that
\begin{align}
\Delta \psi &= F(\psi), \qquad \text{in} \ M,\\
 \psi &= 0, \qquad  \ \ \ \ \ \text{on} \ \partial M.
\end{align}
See e.g. Lemma 5 of \cite{cdgb}.  Clearly $\psi$ cannot change sign in $M$ (it otherwise would contradict $\psi$ having a unique critical point).  It follows from the symmetry results of Gidas-Ni-Nirenberg \cite{gnn} that $\psi$ is radially symmetric $\psi = \psi(|x|)$.  The only (up to scaling) circular isochronal flow is solid body rotation $\psi(|x|) = \frac{1}{2} |x|^2$.
\end{proof}

In light of Lemma \ref{isolem},  since the isochronal flow is nearby (in any topology) circular solutions which are not isochronal, we have:

  \begin{myshade}
  \vspace{-1mm}
\begin{corollary}
 All steady Euler solutions on the disk are unstable in $C^{1,\alpha}$.
\end{corollary}
\end{myshade}

Yudovich conjectured that the elliptical family, of which solid body rotation is a member, are the unique isochronal Euler solutions having constant vorticity \cite{Y00}.  However, preliminary computations of ours based on expanding the analytic Euler solutions with constant vorticity about the center critical point suggests there exists another branch of steady states that are isochronal, having domains which terminate at one with non-smooth boundary. As such, the situation may not be so simple. Nevertheless we ask

\begin{myshade2}
  \vspace{-1mm}
\begin{question}\label{isochronalconju}
For each simply connected domain $M\subset \mathbb{R}^2$, is there at most one (modulo scaling) smooth isochronal flow?
\end{question}
\end{myshade2}

A rough motivation for this question is as follows.  Steady Euler solutions can be constructed by finding a scalar function $\psi$ sharing the same level sets as its Laplacian, the vorticity $\omega=\Delta \psi$.  Imagine starting with a single level curve (say the boundary of the domain) and ``evolving" away to form a steady Euler solution. Thus there is one degree of freedom at each level of $\psi$, since the vorticity can be an arbitrary function of the streamfunction.  On the other hand, the condition of isochronality removes a degree of freedom, so by naive counting one should obtain at most one flow for 
suitable initial conditions.

More substantial support comes from the following simple fact: given an isochronal flow with streamfunction $\psi_0$, new isochronal flows can be generated by finding other stationary Euler solution on the orbit of the \emph{streamfunction}  $\psi_0$ in the group of area preserving diffeomorphisms.
\begin{myshade}
  \vspace{-1mm}
\begin{lemma}\label{isolem}
Let $M_0\subset \mathbb{R}^2$ and suppose that $\psi_0:M_0\to \mathbb{R}$ is a $C^1$ streamfunction of an isochronal Hamiltonian vector field on $M_0$. Let $M\subset \mathbb{R}^2$ with $|M|= |M_0|$  be diffeomorphic to $M_0$.  Any element $\psi$ of 
\be\nonumber
\mathcal{O}_{\psi_0}:= \{ \psi_0 \circ \varphi \ {\rm for \ any} \ \varphi:M \to M_0 \ {\rm smooth \ area\ preserving\ diffeomorphism} \}
\ee
is a streamfunction $\psi:M\to \mathbb{R}$ of an isochronal Hamiltonian vector field on $M$.
\end{lemma}
\end{myshade}

\begin{proof}
Suppose that $\psi_0$ is isochronal with travel time $\mu_0(c)= \mu_0$ (defined by \eqref{rotation}). If $M_0$ is simply connected, then $\psi_0$ has one isolated critical point so that its level sets foliate $M_0$.  If it is doubly connected, it has no critical points, and there do not exist isochronal Hamiltonian vector fields on domains with more than one hole. In any case, call ${\rm rang}(\psi_0)=[a,b]$.  By the coarea formula, we have for any $f\in L^1([a,b])$  the identity
\be\nonumber
\int_{M_0} f(\psi_0(x)) \rmd x =\mu_0\int_a^b f(s) \rmd s.
\ee
  Let $M$ be a deformation of $M_0$ having equal area.  Any $\psi\in \mathcal{O}_{\psi_0}$ takes the form $\psi= \psi_0 \circ \gamma^{-1}$ where $\gamma: M_0 \to M$. Consequently, we have
\be \nonumber
\int_M f(\psi(x))\rmd x= \int_M f(\psi_0(\gamma^{-1}(x)))\rmd x  = \int_{M_0} f(\psi_0(x)) \rmd x =  \mu_0 \int_a^b f(s) \rmd s.
\ee
Choosing $f(x) = \chi_{[a,c]}(x)$, we obtain an identity on the area enclosed by $\psi$--levels:
\be\label{defarea}
{\rm Area}({\{\psi \leq c\}}) = \mu_0 (c-a).
\ee
 Finally, we recall the general identity (see e.g. \S 4.1 of \cite{serre93} or Appendix E of \cite{cdg})
\be\label{muident}
\mu(c) =\frac{\rmd }{\rmd c} {\rm Area}({\{\psi \leq  c\}}).
\ee
It follows from \eqref{defarea} and \eqref{muident} that  $\mu(c) = \mu_0$ is constant.  
\end{proof}

Using Lemma \ref{isolem}, we can exhibit a nontrivial family of isochronal Euler solutions. Specifically, let $\psi_0$ be the streamfunction for a isochronal steady Euler solution with  $\mu_0(c)= \mu_0$  on $M_0\subset \mathbb{R}^2$ satisfying the stability condition \textbf{(H1)}. For example, one can take $M_0$ to be an ellipse and  $\psi_0 $ to be the constant vorticity isochronal streamfunction \eqref{elliptical} on $M_0$.  
  Let $M$ be a slight deformation of $M_0$ having equal area (here we can also ``wrinkle" the domain $M$ by varying the Riemannian metric).  By Theorem \ref{2dEthmarnold} there exists a steady Euler streamfunction on $M$ of the form $\psi= \psi_0 \circ \gamma^{-1}$ where $\gamma: M_0 \to M$ is an area preserving diffeomorphism  \cite{cdg}. Since $\psi\in \mathcal{O}_{\psi_0}$, this  shows that for all sufficiently small deformations $M$ of an elliptical domain $M_0$, there exists at least one isochronal flow  $\omega =\Delta \psi$ on $M$. The algorithm producing the new solution via  Theorem \ref{2dEthmarnold} gives a unique diffeomorphism $\gamma: M_0 \to M$ and thus offers some support to Conjecture \ref{isochronalconju}.  It is conceivable that all isochronal Euler solutions can be constructed in this way, tracing out a path from solid body rotation on the disk.   Since all isochronal flows constructed in this way continue to satisfy \textbf{(H1)}, they are are non-isolated from non-isochronal steady solutions and thus unstable in $C^{1,\alpha}$ according to Koch's theorem \ref{thminstab}.

We finally remark that the issue of characterizing flows with the isochronal property subject to various constraints arises in the classical subject of the constructing the tautochrone curve, in Hamiltonian dynamical systems \cite{CS99} and even has applications to problems in engineering \cite{el}.

Returning to the issue of long-term growth, it is natural to suppose further that  all steady (at least) Euler solutions  should be infinitely unstable in $C^\alpha$ and that this behavior should be generic.  This conjecture was 
 enunciated by Yudovich (1974), \cite{Y74,Y00}, quote from \cite{MSY}:
\\

\noindent \emph{``There is a `substantial set' of inviscid incompressible flows whose vorticity gradients grow without bound.  At least this set is dense enough to provide the loss of smoothness for some arbitrarily small disturbance of every steady flow."} \cite{MSY}
\\

We express this conjecture as follows

\begin{myshade2}
  \vspace{-1mm}
\begin{conjecture}[Yudovich (1974), \cite{Y74}]\label{instconj}
No steady solutions of the Euler equations are stable in $C^{1,\alpha}$.  Moreover, for any steady state $\omega_*\in C^{\alpha}$  and any $\ve>0$, there is a $\|\omega_0-\omega_*\|_{C^{\alpha}} \leq \ve$ such that  $\|\omega(t)\|_{C^{\alpha}} \to \infty$ as $t\to \infty$.
\end{conjecture}
\end{myshade2}

A stronger form of the above conjecture reads: for any steady state $\omega_*\in C^{\alpha}$  and any $\ve>0$, in any $\ve$--neighborhood $N$ of $\omega_*$ in $C^{\alpha}$ there exists an open set $B\subset N$ such that all $\omega_0\in B$ have $\|\omega(t)\|_{C^{\alpha}} \to \infty$ as $t\to \infty$.

Here we show that Koch's instability argument can be used to prove \emph{generic} long time growth nearby certain stable steady states.
 Given $\omega_0\in C^\alpha(M),$ we may consider the solution map $S_t(\omega_0)$ taking $\omega_0$ to the unique Euler solution at time $t$ that is equal to $\omega_0$ at time $t=0.$ It is not difficult to see that $S_t$ is \emph{not} continuous on $C^\alpha.$ Despite this, we have the following Lemma.

\begin{myshade}
  \vspace{-1mm}
\begin{lemma}
The map $\|S_t(\cdot)\|_{C^\alpha(M)}: C^\alpha(M)\to \mathbb{R}^+$ is lower-semicontinuous for any $t\in\mathbb{R}$.
\end{lemma}
\end{myshade}
\begin{proof}
If a family of continuous initial vorticities $\omega_n\rightarrow \omega_0$ in $C^\alpha(M)$ (even in $C(M)$), we have that $\omega_n(t)\rightarrow\omega(t)$ uniformly. The result then follows easily from the definition of the $C^\alpha$ norm.   
\end{proof}
\begin{remark}
While it is not difficult to show that $S_t(\cdot)$ is not continuous on $C^\alpha$, it is not immediately obvious whether the quantity $|S_t(\cdot)|_{C^\alpha}$ is continuous. However, lower semi-continuity is sufficient for our purposes.  
\end{remark}

In what follows, all balls are in the fixed $C^\alpha$ topology and $B_\ve$ will be a given ball in $C^\alpha.$

\begin{myshade}
  \vspace{-1mm}
\begin{theorem}\label{genthm}
Assume that for every $\omega_0\in B_\ve $ we have that the corresponding Lagrangian flow-map satisfies $\|\nabla\Phi_t\|_{L^\infty}\geq c(t),$ for some growth function $c(t)\rightarrow\infty$ as $t\rightarrow\infty.$ Fix a function $\delta(t)\rightarrow 0$ with $\delta(t)c(t)^\alpha\rightarrow\infty.$ Then, there is a dense set of $\omega_0$ in $B_\ve $ for which 
\[\sup_t\frac{\|S_t(\omega_0)\|_{C^\alpha}}{\delta(t)c(t)^\alpha}= +\infty.\]
\end{theorem}
\end{myshade}

\begin{proof}
Call the set of $\omega_0$ satisfying the conclusion $\mathcal{U}$. Now let \[\mathcal{U}_{N}=\left\{\omega_0\in B_\ve : \sup_{t} \frac{\|S_t(\omega_0)\|_{C^\alpha}}{\delta(t)c(t)^\alpha}>N\right\}.\] By lower semi-continuity of the norm, this is an open set in the $C^\alpha$ topology. Now we will show that it is dense using the result of Koch  \cite{K02} (Theorem \ref{thminstab} herein). Indeed, by Koch's main Theorem \ref{thminstab}, given $\omega_0\in B_\ve ,$ and any $T,\kappa>0$, there exists $\xi$ with $\|\xi\|_{C^\alpha}<\kappa$ for which,  for $t=T$, we have
\[\|S_t(\omega_0+\xi)\|_{C^\alpha}\geq \frac{1}{2}\kappa c(T)^\alpha.\] Thus, take $\kappa$ as small and $T$  large so that $\frac{1}{\delta(T)} \geq 4\kappa^{-1} N.$ Then, 
\[\frac{\|S_t(\omega_0+\xi)\|_{C^\alpha}}{\delta(T)c(T)^\alpha}\geq 2N.\] It follows that $\omega_0+\xi\in \mathcal{U}_N$. Thus $\mathcal{U}_N$ is dense in $B_\ve .$ Now let $B$ be any ball with $\overline{B}\subset B_\ve $. $\overline{B}$ is a complete metric space. The sets $\mathcal{U}_N\cap \overline{B}$ are open and dense in $\overline{B}$. Thus, by the Baire Category Theorem\footnote{Note that the step of intersecting with $\overline{B}$ is just to emphasize that the Baire Category Theorem applies to open subsets of complete metric spaces, even though open subsets are not usually complete. In fact, intersecting with a closed ball ``inside" of the original ball is just the first step of the proof of the Baire Category Theorem.} \[\left\{\omega_0\in \overline{B}: \sup_t \frac{\|S_t(\omega_0)\|_{C^\alpha}}{\delta(t)c(t)^\alpha} =+\infty\right\}:=\mathcal{U}\cap \bar{B}=\bigcap_{N=1}^\infty \mathcal{U}_N\cap \overline{B}\] is dense in $\overline{B}.$ The result follows. 
\end{proof}

\begin{remark}
This result can be applied \emph{within} symmetry classes.  As such, it might be used to prove generic (within symmetry) fast growth.
\end{remark}

A direct corollary of  Theorem \ref{genthm} is generic growth near stable steady states.

\begin{myshade}
  \vspace{-1mm}
\begin{corollary}
[Generic loss of smoothness near stable  steady states ]\label{stablethm}
Let $M\subset \mathbb{R}^2$ be an annular fluid domain.  Fix $\alpha>0$ and let $\omega_*\in {C^{\alpha}(\bar M)}$ be an Arnold stable Euler solution such that the boundary travel times are not equal. Then there exists $\ve>0$ such that the set  
\be\label{Dset}
\left\{\omega_0\in B_\ve(\omega_*)\ : \ \exists c>0 \ \ \text{so that} \ \ \sup_{t} \frac{\|\omega(t)\|_{C^\alpha}}{ |t|^{\alpha-}}= +\infty  \right\}
\ee
 contains a dense set in $ B_\ve(\omega_*)$.
\end{corollary}
\end{myshade}
\begin{proof}
  The argument in Thm \ref{constvortinst} shows that the gradient of the Lagrangian flowmap for any perturbation grows at least linearly.  We then apply Thm \ref{genthm}.
\end{proof}

We also announce a stronger result for a class of $\mathsf{M}$--stable stationary solutions.  These are non-isochronal $L^2$ stable stationary solutions possessing an isolated global maximum or minimum at either or single point or at the boundary. Here generic  loss of smoothness can be definitively established:

\begin{figure}[htb]\centering
        \includegraphics[width=.3\columnwidth]{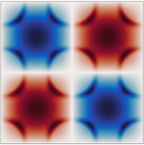} 
    \includegraphics[width=.3\columnwidth]{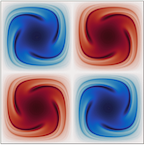} 
    \includegraphics[width=.3\columnwidth]{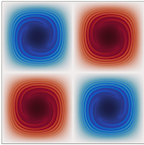} 
  \caption{\small Numerical simulation of   slightly viscous Navier-Stokes with viscosity $\nu= 10^{-6}$  on $\mathbb{T}^2$ exhibiting spiral formation.  Shown are heatmaps of the vorticity field at subsequent times, with right-most panel at $t:= 0.0025 \times  t_\nu$ where $t_\nu := 1/\nu$. Initial data $\omega_0= \omega_*+ \omega_0'$ is a perturbation $\omega_0'(x,y):= \sin(x)\sin(y) \exp\left(-100\times\left|\cos(x)^2\cos(y)^2- \tfrac{1}{4}\right|^2\right)$ of the cellular flow $\omega_*(x,y):= \sin(x)\sin(y)$. See \cite{C21,CD21}. 
 For perturbations with odd-odd symmetry, $\omega_*$ is nonlinearly stable under Euler evolution. For movie, see \url{https://www.youtube.com/watch?v=25Md9qxIReE}
  }
\label{fig1}
\end{figure}

\begin{myshade}
  \vspace{-1mm}
\begin{theorem}
[Generic loss of smoothness near $\mathsf{M}$--stable  steady states \cite{DE21}]\label{mstablethm}
Let $M\subset \mathbb{R}^2$ be a bounded fluid domain.  Fix $\alpha>0$ and let $\omega_*\in {C^{\alpha}(\bar M)}$ be a $\mathsf{M}$--stable Euler solution. Then there exists $\ve>0$ such that the set  
\be\label{Dset}
\left\{\omega_0\in B_\ve(\omega_*)\ : \ \exists c>0 \ \ \text{so that} \ \ \|\omega(t)\|_{C^\alpha}\geq c |t|^\alpha \ \ \text{for all} \ \  t\in \mathbb{R} \right\}
\ee
 contains an open and  dense set in $ B_\ve(\omega_*)$.
\end{theorem}
\end{myshade}
The basic idea of the proof is that, since the vorticity is non-constant, the Eulerian stability strongly constrains the motion of particles in the direction transverse to the vortex lines.  Thus, if the state is not isochronal, then  particles starting near different streamlines of the base steady solution will experience a differential rotation which causes that increased separation over time (in a covering of the domain). As such, what is really required is that, away from regions where the solution is locally isochronal, the vorticity is not identically constant.  In that case, provided it is not isochronal, the solution is generically infinitely unstable in $C^{\alpha}$.  
This allows us to establish indefinite shearing of the Lagrangian flowmap for a full neighborhood of time dependent solutions nearby the steady state. 
Using this we show for any $\delta>0$ there exists $\xi\in B_\delta(\omega_*)$  such that for  $\ve>0$ sufficiently small, any $\omega_0\in B_{\ve}(\xi)$ gives rise to $\omega(t) = S_t(\omega_0)$ having the property
\be
 \|\omega(t)\|_{C^\alpha}\geq c t^{\alpha}
\ee
for some $c:= c(\xi)>0$.  The result now follows.

The strongest in this line of conjectures dispenses with the assumption of proximity to steady states:

\begin{myshade2}
  \vspace{-1mm}
\begin{conjecture}\label{conjgrow2}
For each $\omega_*\in C^{\alpha}$  and any $\ve>0$, there is a $\|\omega_0-\omega_*\|_{C^{\alpha}} \leq \ve$ such that  $\|\omega(t)\|_{C^{\alpha}} \to \infty$ as $t\to \infty$.
\end{conjecture}
\end{myshade2}
In fact, a more severe form of singularity formation might be generic (see Conjecture \ref{svconj} below) and is related to the asymptotic portrait of the motion.

\subsection{Asymptotic picture: ``entropy" decrease} \label{conjectent}

In this section, we give a mathematically precise but conjectural picture of generic behavior at  infinite time (either $+\infty$ or $-\infty$) for 2D Euler solutions.   
We should take a moment to discuss what is observed from numerical simulations and physical experiment.  Starting from almost any initial data,  large scale coherent structures are robustly seen to emerge as time progresses through a process of vortex merger.  The formation and persistence of these structures are sometimes understood to be a manifestation of the inverse energy cascade.  See Figures \ref{fig2} and \ref{fig3} for two `typical' examples of such behavior on the torus $\mathbb{T}^2$ (see \cite{Modin1} for more impressive examples on the sphere $\mathbb{S}^2$). 
These observations indicate that a great deal of diversity is lost in the long time behavior -- the vorticity portrait on $\mathbb{T}^2$ always begins to resemble some very special subset of phase space consisting of a collection of  coherent vortices wandering around the domain.  It is tempting to view this apparent `contraction' in phase space as a decrease in entropy, roughly understood as a measure of diversity of a set.\footnote{We stress that this apparent decrease of diversity is in the phase space. As discussed in \S \ref{lap}, there is an alternative description of fluid motion in its configuration space in the group of area preserving diffeomorphisms $\mathscr{D}_\mu(M)$.  In principle, ``entropy" could flow from the  phase space variables (which is what we typically observe in nature and simulation) to the configuration space variables. This configuration space is known to have infinite diameter ${\rm diam} \mathscr{D}_\mu(M)=\infty$ \cite{ER91,shnir2}, and so it can absorb an arbitrary amount of entropy.  In this way, one might be able to understand irreversibility at infinite time in terms of this transfer of diversity within the two descriptions of fluid motion.  We owe this remark to A.Shnirelman (private communication).}
 Together with the detailed structure of the limiting profiles, understanding these phenomena from first principles (as properties of long time limits of solutions of the 2D Euler equations)  is a major challenge.

\begin{figure}[htb]\centering
    \includegraphics[width=.3\columnwidth]{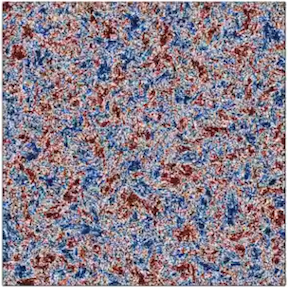} 
    \includegraphics[width=.3\columnwidth]{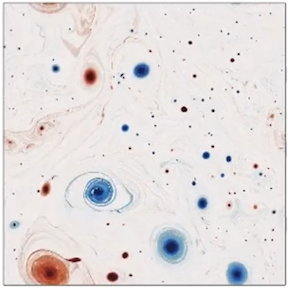} 
    \includegraphics[width=.3\columnwidth]{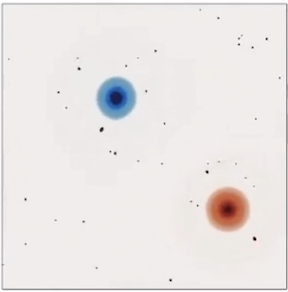} 
  \caption{Emergence of a vortex dipole at long times in a simulation a slightly viscous Navier-Stokes. 
 For movie, see \url{https://www.youtube.com/watch?v=-YdEYumSSJ0}
 }
\label{fig2}
\end{figure}

\begin{figure}[htb]\centering
    \includegraphics[width=.3\columnwidth]{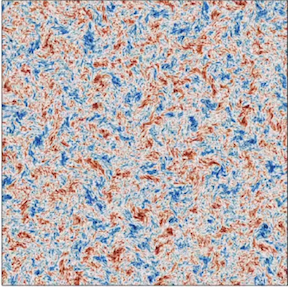} 
    \includegraphics[width=.3\columnwidth]{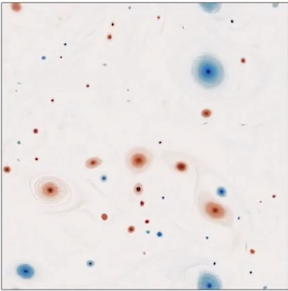} 
    \includegraphics[width=.3\columnwidth]{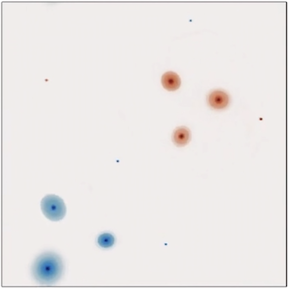} 
   \caption{Emergence of a vortex tripole pair at long times in a simulation of slightly viscous Navier-Stokes.
 For movie, see \url{https://www.youtube.com/watch?v=a3ENdIy1WL0}}
\label{fig3}
\end{figure}

There has been a great deal of activity in this direction under the name of `statistical hydrodynamics' \cite{M90,O49,R91,SSR91}. Specifically, under an implicit assumption on the uniform validity of finite dimensional approximations for long times together with an assumption of ergodicity of the finite dimensional dynamics, they predict that ``end states" are thermodynamic equilibria which arise from maximizing a certain entropy\footnote{We remark that the entropy which is maximized in these theories is associated to the distribution of vorticity in space. Were it not for the constraint imposed by conservation of energy and momentum, the maximum entropy state in this sense would be a uniform mixing, i.e. replacing the vorticity by its average over the entire flow domain. It is a very different measure than from the notion of entropy discussed above in connection to the lack of diversity of end states in the infinite dimensional phase space of available vorticity configurations.  Thus, there is no contradiction in fluid motion increasing the former while decreasing the latter. } measuring the available phase space of a given configuration on a fixed momentum and energy surface.  Lars Onsager, a pioneer in these ideas, painted the following picture in the context of the point-vortex dynamical system:
\\

\noindent \emph{``$\dots$ vortices of the same sign will tend to cluster -- preferably the strongest ones -- so as to use up excess energy at the least possible cost in terms of degrees of freedom. It stands to reason that the large compound vortices formed in this manner will remain as the only conspicuous features of the motion; because the weaker vortices, free to roam practically at random, will yield rather erratic and disorganized contributions to the flow."} \cite{O49}
\\

Indeed, Onsager's conjectured picture is a remarkably accurate description of the late stages of numerical simulations of the continuum system (see e.g. the final panels in  Figures \ref{fig2} and \ref{fig3}). However, 
in practice,  these equilibrium theories often predict that the fluid will settle down to a  steady solution which is completely determined by its momentum and energy together with possibly the vorticity distribution function of its initial data.  This prediction is in contradiction with simulations and experiment which indicate that the flow can, and generally does,  retain time dependence asymptotically. {For example, Figures \ref{fig2} and \ref{fig3} which are initiated with Gaussian random data seem to have associated end states which are time dependent and are different despite having comparable initial energy/vorticity distribution.  On the other hand, in Figure \ref{fig1}, the flow is started near a steady state and appears to relax back to equilibrium in a weak sense.}   

There are a number of possible causes for the failure of statistical hydrodynamic theories to capture this aspect of fluid flows.  
The issue mainly resides in the commutation of two limits, $N\to \infty$ (from finite to infinite dimensional) and $t\to \infty$ (long time). 
Assuming the finite dimensional dynamics are sufficiently ergodic\footnote{There are instances in which the finite dimensional dynamics are known to not be ergodic (see \cite{Kh82} for results on the point-vortex system), although the presence of multiple ergodic components in the phase space by itself may not invalidate the conclusions of equilibrium theories \cite{ES93}.}, these theories predict that the long time behavior at fixed $N$ will spend the majority of its time around the maximizer of the entropy (maintaining the aforementioned constraints).  A leap of faith is taken when it is asserted that the infinite dimensional system therefore displays the same behavior.

For the infinite dimensional dynamics, wandering neighborhoods exist \cite{N,S97} (Theorems \ref{thmwand1} and \ref{thmwand2}), along with infinite dimensional families of Lyapunov stable solutions \cite{arn11,arn12} (Theorems \ref{2dEthmarnold}, \ref{thmArn} and \ref{stabinX}).  Most severely, there are examples of fluid flows which exhibit mixing (loss of compactness) which means that the Euler system is ``open" in the sense that information (say enstrophy) can be lost through infinite frequency in the limit of long time.  This subtle behavior is intrinsically linked to the infinite dimensionality and cannot be captured by its finite dimensional approximations.    All of the above are instances of non-equilibrium behavior of a mechanical system and, as such, there is no good reason to expect theories built on equilibrium considerations to be realistic.

Shnirelman built a theory which embraced this infinite dimensional aspect of fluid motion \cite{shn01} (see also \cite{DD22}).  This theory made the concrete prediction that, if Euler is ``maximally mixing" up to constraints imposed by the conserved integrals (momentum, energy, etc), then the long time behavior is described by an Arnold stable stationary flow. Despite its more realistic setting, this prediction is also in seeming contradiction with numerical simulations starting far from equilibrium (although it may be appropriate near stable equilibria, see Fig. \ref{fig1}).  

This points toward the apparent fact that the Euler dynamics is not generally completely effective at mixing. Rather, solutions can become ``trapped" in time-dependent regimes which can indefinitely avoid further mixing.   While  the end states may be (and by many observations, are) a very rich and complicated set,  there may exist a precise mathematical framework which can  effectually capture the spirit of the above discussion.
Our considerations are motivated by the following two conjectures concerning fluid flows on bounded planar domains $ M\subset \mathbb{R}^2$:

\begin{myshade2}
  \vspace{-1mm}
\begin{conjecture}[Šverák (2011), \cite{sv11}]\label{svconj}
Generic initial data $\omega_0\in L^\infty( M)$ gives rise to inviscid incompressible motions whose vorticity orbits $\{ \omega(t)\}_{t\in \mathbb{R}}$ are not precompact in $L^2( M)$.
\end{conjecture}
\end{myshade2}

The precise meaning of ``generic" in the above statement is to be elucidated.  It may also be better to use the space $X$ introduced by \eqref{xtop} herein in regards to the genericity statement above as $X$ is a separable Banach space and has the property that a unit ball in $X$ is invariant for the Euler dynamics at finite times.
Conjecture \ref{svconj} represents an objective form of creation of `small scale' vorticity. In particular, it implies that a form of mixing or self-averaging must occur in the limit of infinite time somewhere in the flow domain, although in no way implies that the dynamics perform this mixing as efficiently as possible.

\begin{myshade2}
  \vspace{-1mm}
\begin{conjecture}[Shnirelman (2013), \cite{shn13}]\label{shnconj}
For any initial data  $\omega_0\in  L^\infty( M)$, the collection of $L^2( M)$ weak limits of the orbit $\{ \omega(t)\}_{t\in \mathbb{R}}$ consists of  vorticities which generate $L^2( M)$ precompact orbits under 2D Euler evolution.
\end{conjecture}
\end{myshade2}

Conjecture \ref{shnconj} asserts that there exists a genuine (weak) attractor for the 2D Euler equations, the 
Omega limit set defined\footnote{This definition of the Omega limit set of the phase space slightly different from the standard one used in dynamical systems which includes also the weak-$*$ limit points of sequences of initial data. This  definition has the property that the Omega limit set of the phase space is the entire phase space and as such cannot capture the irreversible features of 2D fluid motion. On the other hand, the set \eqref{omlinset} is quite possibility a meager portion of the phase space, in line with Conjectures \ref{svconj} and \ref{shnconj}.} by 
 \be\label{omlinset}
\Omega_+(X_*):= \bigcup_{\omega_0\in X_*} \Omega_+(\omega_0), \qquad \Omega_+(\omega_0) := \bigcap_{s\geq 0} \overline{\{ S_t(\omega_0), t\geq s \}}^*.
\ee 
which is made up entirely of those orbits which ``do not mix at infinity."   Examples of such orbits are stationary \cite{arn13}, time periodic \cite{HM17}, quasiperiodic  \cite{C13} and perhaps also chaotic \cite{paco}, states.

Taken together, Conjectures \ref{svconj}  and \ref{shnconj} give a precise notion of ``entropy decrease" for perfect fluids in the long time limit -- a ``second law of thermodynamics" for 2D Euler.  By this we mean that it gives some quantification of the apparent fact that the velocity fields of long lived states (whose orbits according to Conjecture \ref{shnconj} are precompact) represent a `sparse' set in the entire $L^\infty$ phase space  (whose orbits according to Conjectures \ref{svconj} are generally not precompact).  See Figure \ref{fig:longtime} for a cartoon involving the Omega limit set  of the phase space $X_*$ defined in \eqref{Xstar}.

\begin{figure}[h!]
	\centering
\begin{tikzpicture}[scale=.4, line cap=round, line join=round]
	\shade[ball color = gray!40, opacity = 0.4] (0,0) circle (2cm);
	\draw[thick] (0,0) circle (2cm);
	\draw[thick] (-2,0) arc (180:360:2 and 0.6);
	\draw[thick, dashed] (2,0) arc (0:180:2 and 0.6);
	
	         	 \draw  (-4.5,0) node[anchor=west] {\Large $S_t:$}; 
	\draw[thick, black, ->] (3.0,0)--(4.5,0);
	\draw  (1.0,2.0) node[anchor=west] { $X_*$};
	\draw  (8.5,2.0) node[anchor=west] { $X_*$};
	
	\shade[ball color = gray!40, opacity = 0.4] (7.5,0) circle (2cm);
	\draw[thick] (7.5,0) circle (2cm);
	\draw[thick] (5.5,0) arc (180:360:2 and 0.6);
	\draw[thick, dashed] (9.5,0) arc (0:180:2 and 0.6);
\end{tikzpicture}
	\hspace{10mm}
	\begin{tikzpicture}[scale=.4, line cap=round, line join=round]
\draw[dashed] (0,0) circle (2cm);
\draw[dashed] (-2,0) arc (180:360:2 and 0.6);
\draw[ dashed] (2,0) arc (0:180:2 and 0.6);
		\draw[thick, gray] plot [smooth, tension=1.5] coordinates {(-1.5,.9) (-1.3,-.8) (-1,0) (-1.2,.5) (-.9,.9) (-1.3,-.6) (-.9,-1.2) (-.4,-1.5) (0,0) (-.2,-.2) (-.4, -.8) (-.6,-1.2) (.5,-.5) (0,1) (-.2,.8) (-.4,.2) (1,1) (1.5,-.8) (.3,0)};
		\draw  (-8.5,0) node[anchor=west] {\Large $\Omega_+(X_*) = $}; 
	\end{tikzpicture}
	\caption{Euler preserves the unit $L^\infty$ ball $X_*$ \eqref{Xstar} for all finite time $S_t:X_*\to X_*$.  However the Omega limit set $\Omega_+(X_*)$  \eqref{omlimset} of the entire phase space can be sparse.}
	\label{fig:longtime}
\end{figure}
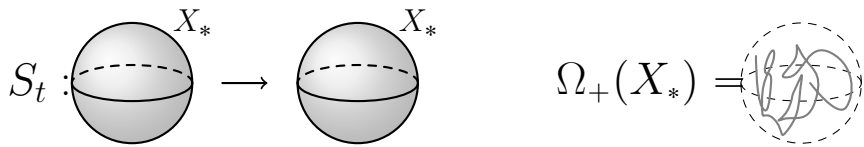

{A beautiful refinement to Shnirelman's Conjecture \ref{shnconj} is given by the recent work Modin and Viviani \cite{Modin0,Modin1} which postulates some structure of the precompact orbits which form $\Omega_+(X_*)$.  They conjecture, based on careful numerical simulations of Euler on the sphere using a refinement of Zeitlin's method \cite{ZM}, that the portrait of fluid motion in the limit $t\to\infty$ consists generically of $N$ blobs where $N$ is the maximum number of point vortices on the fluid vessel $M$ for which the dynamics is integrable (given the macroscopic characteristics; energy, circulation, momentum).  The subsequent motion of the center of masses of these blobs is described approximately by the point vortex system. See Figures \ref{figKM1} and \ref{figKM2} (kindly provided by Klas Modin)  for an empirical support of this conjecture.  We remark also that Conjecture 1 of \cite{Modin2} posits necessary conditions for the point vortex dynamics on a  2-dimensional Riemannian manifold with symmetry group $G$ to be integrable.  }

\begin{figure}[h!]\centering
    \includegraphics[width=1\columnwidth]{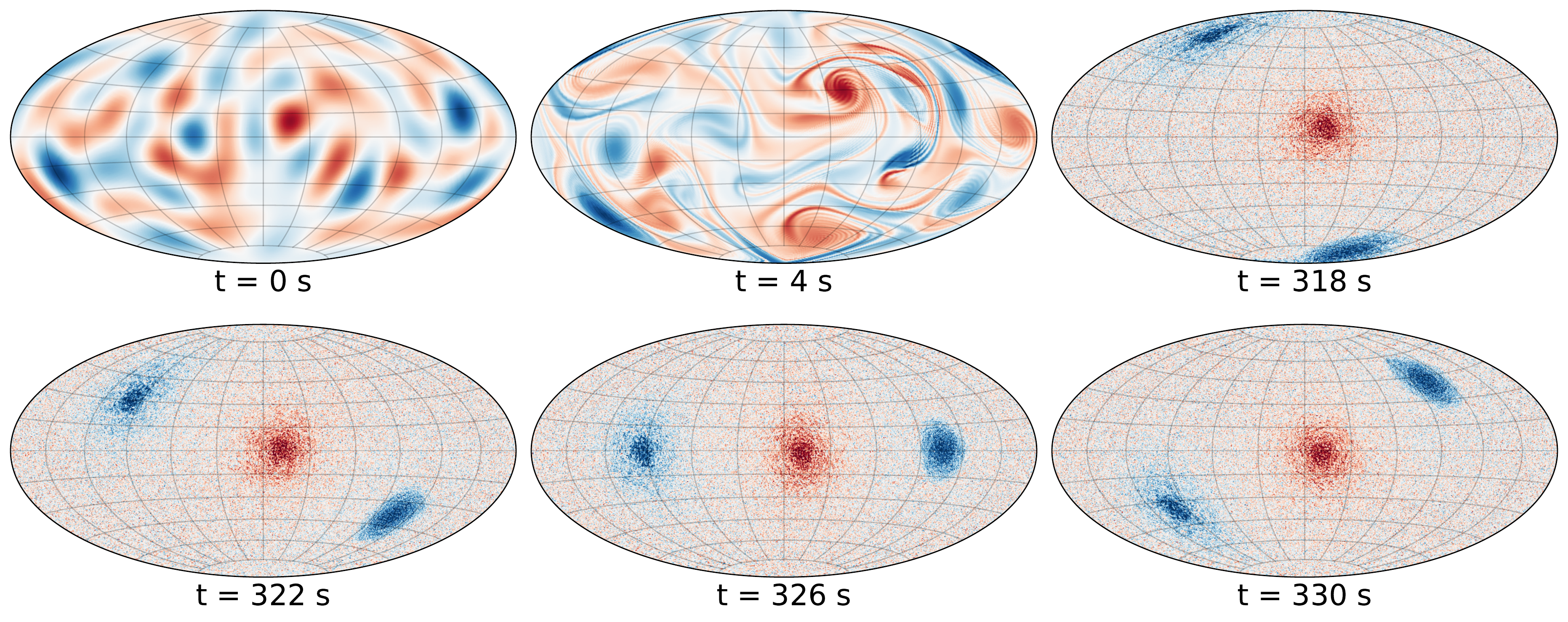} 
   \caption{Evolution of vorticity for Euler equations on the sphere, with randomly generated, smooth initial data. After a transient stage, where vorticity regions of equal sign undergo weak mixing, a dynamically stable configuration of 3 vortex blobs emerge. These blobs continue to interact, yielding quasi-periodic trajectories well approximated by 3 point vortices on the sphere. That 3 blobs emerge seems to correspond to integrability conditions: on the sphere N point vortex dynamics is integrable for $N\leq 3$, and in general non-integrable for $N\geq 4$.  For movie, see \url{https://play.chalmers.se/media/Sphere+simulation+smooth+data/0_sqc7dy3t}}
\label{figKM1}
\end{figure}

\begin{figure}[h!]\centering
    \includegraphics[width=1\columnwidth]{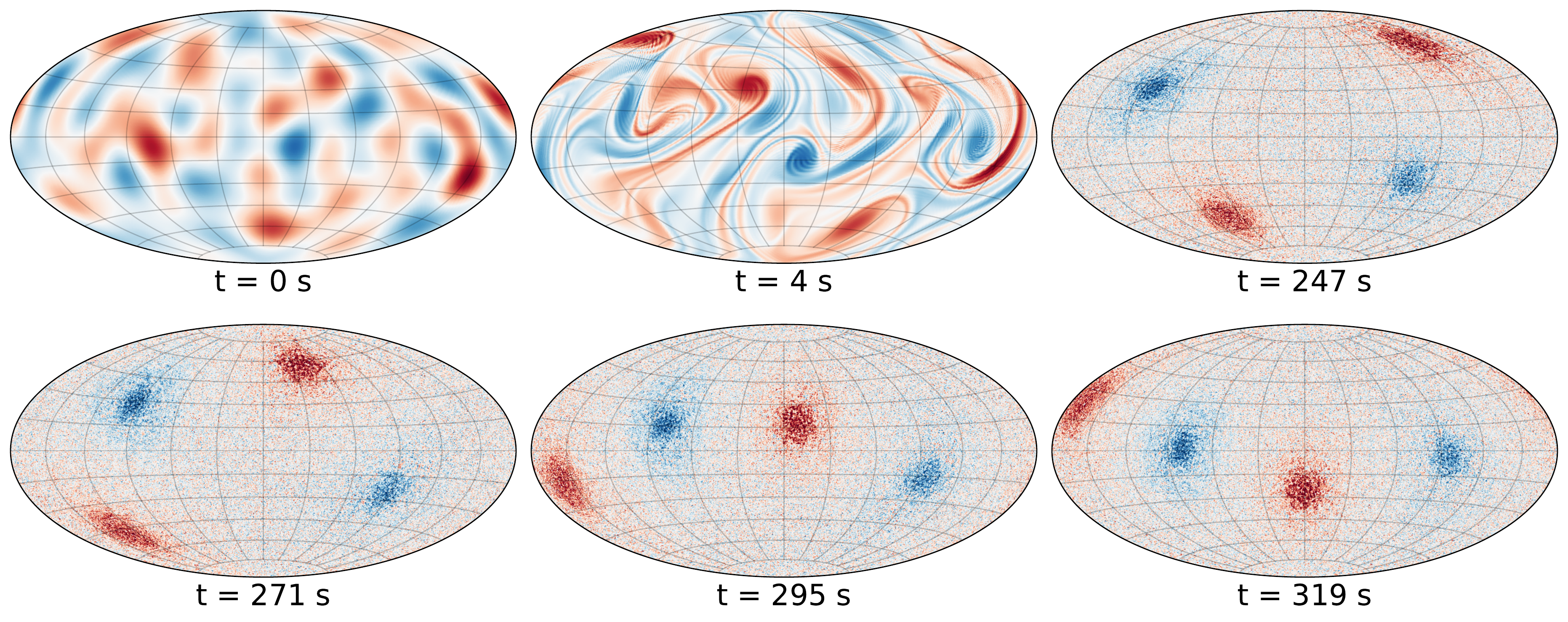} 
   \caption{Evolution of vorticity for Euler equations on the sphere, with randomly generated, smooth initial data chosen so that the \emph{total angular momentum vanishes}. A dynamically stable configuration of 4 vortex blobs emerge. The trajectories of these blobs are well approximated by 4 point vortices on the sphere. That 4 blobs emerge seems to correspond to integrability conditions: on the sphere N point vortex dynamics with vanishing momentum is integrable for $N\leq 4$.  For movie, see \url{https://play.chalmers.se/media/Sphere+simulation+vanishing+momentum/0_eik8wnab}}
\label{figKM2}
\end{figure}

A much simpler question than Conjectures \ref{svconj}  and \ref{shnconj} is the following
  \begin{myshade3}
  \vspace{-1mm}
\begin{problem}\label{omlimprob}
Show that $\Omega_+(X_*)\neq X_*$.
\end{problem}
\end{myshade3}
Namely, show that there are some motions which cannot persist indefinitely.  The point $\xi$ and its neighborhood from Nadirashvili's wandering Theorem \ref{thmwand1} appears a good candidate for exclusion from the Omega limit set (i.e. show that $\xi$ is a weak-$*$ wandering point).   A resolution of Problem \ref{omlimprob} would exclude the existence of a weak-$*$ ergodic orbit on leaves of equal energy, i.e. a datum $\omega_0\in X_*$ for which $\Omega_+(\omega_0)=X_*\cap \{ \mathsf{E}=\mathsf{E}_0\}$ (which, by itself, appears to be a open issue), and would definitively establish the irreversibility of Euler at long times.

While there are very few rigorous results concerning 2D Euler which survive for all time, there are two  which must be mentioned in connection with these Conjectures. First is a result of Šverák which establishes that, for each initial datum $\omega_0\in L^\infty( M)$, there exists at least one element of its Omega limit set (one long time limit) which corresponds to a $L^2$ precompact orbit \cite{sv11}, lending some support for Conjecture \ref{shnconj}.
Next, there is one setting for which both of these Conjectures are proved to be correct in some sense.  At this moment, these results pertain only to flows near certain equilibria on the periodic channel $\mathbb{T}\times [0,1]$ or its unbounded analogue $\mathbb{T}\times \mathbb{R}$. This work on \emph{inviscid damping}, initiated with the celebrated paper of Bedrossian and Masmoudi \cite{BM15}, show that in an open neighborhood (in a Gevrey-$\frac{1}{s}$ for $s\in(0,1]$ topology with regularity quantified by the scale of norms $\|f\|_{\mathcal{G}^{\lambda;s}} = \left\|e^{\lambda |\nabla|^s} f\right\|_{L^2}$) of any sufficiently smooth monotone shear flow, the solutions starting from any initial data (with the sole exception of other stationary states) converge weakly but not strongly in $L^2$ back to equilibrium as time goes to infinity \cite{IJ20}. 
The  following is the (rough) statement

\begin{myshade}
  \vspace{-1mm}
\begin{theorem}[Bedrossian and Masmoudi  \cite{BM15}]\label{id} Let $M = \mathbb{T}\times[0,1]$ or $\mathbb{T}\times\mathbb{R}$.  There exists an $\ve_0:= \ve_0(\lambda, \lambda')$ such that if 
\be
\left\|u_0- \begin{pmatrix} x_2 \\0\end{pmatrix} \right\|_{L^2} + \|\omega_0-1\|_{\mathcal{G}^{\lambda;\frac{1}{2}}} = \ve<\ve_0
\ee
then there is a profile $\omega_\infty\in\mathcal{G}^{\lambda';\frac{1}{2}}$  corresponding to a (possibly different) shear flow such that the solution converges weakly-$*$ in $L^\infty(M)$ to this equilibrium, i.e. $\omega(t) \wsc \omega_\infty$ as $t\to \infty$.   Moreover there is an open set of data such that 
\begin{itemize}
\item
 $\|\omega(t)\|_{C^\alpha} \approx \ve t^\alpha$ for all $\alpha>0$, $t\in \mathbb{R}$,
 \item $\|\omega_\infty\|_{L^2} < \|\omega_0\|_{L^2}$.
 \end{itemize}
\end{theorem}
\end{myshade}

This has been generalized by Ionescu and Jia \cite{IJ20} and Masmoudi and Zhao \cite{MZ20} to hold for the class of monotone shear flows i.e. $u_*:= (v(x_2),0) $ where $v\in C^\infty([0,1])$ is a monotone function.  
To some extent, similar phenomena can be established also in the radial vortex setting \cite{BCV,IJ}.
 Theorem \ref{id} and the like are the only to fully characterize the Omega limit sets \eqref{omlimset} for Euler, albeit for very smooth perturbations of special equilibria.  Specifically, if $\omega_*$ is a monotone shear on $\mathbb{T}\times [0,1]$ or $\mathbb{T}\times \mathbb{R}$, then for $\| \omega_0- \omega_*\|_{\mathcal{G}^{\lambda',\frac{1}{2}}} \ll 1$, one has
\be
\Omega_+(\omega_0) = \{ \omega_\infty\} 
\ee
where $\omega_\infty:=\omega_\infty({\omega_0})$ is a (slightly modified) shear flow nearby $ \omega_*$. The convergence happens weakly, not strongly, in $L^2$ for an open set of data, so some amount of mixing definitively occurs in accord with Conjecture \ref{svconj}. Moreover, they show that a certain full neighborhoods in the Gevrey phase space relax to equilibrium at long time, confirming Conjecture  \ref{shnconj} in this context. Of course, in this setting much more can be said,  but far from equilibrium one may hope that  two Conjectures in concert are as robust as seem the phenomena which they describe.

\section{3D fluids: finite-time singularity formation}\label{3dsec}
We now move to discuss the issue of finite-time singularity in solutions to the Euler equation. This problem has been considered by many authors in different contexts. For the purposes of this article, we will discuss analytical results in the following categories:
\begin{itemize}
\item Necessary conditions for blow-up 
\item Lower dimensional models
\item Blow-up of infinite energy solutions
\item Self-similar Blow-up
\end{itemize}
As we discuss each of the above issues, we will mention relevant numerical simulations that have helped to guide the analytical studies. 

\subsection{Background, 1D models and special solutions}

\subsubsection{Necessary Conditions for Blow-up} The local existence and uniqueness of solutions in $C^{1,\alpha}$ and a wide variety of function spaces is not difficult to establish using the basic energy estimate
\begin{equation}\label{BasicEE}\frac{\rmd}{\rmd t}\|u\|_{X}\leq C_{X} \|\nabla u\|_{L^\infty} \|u\|_{X}, \end{equation} enjoyed by smooth solutions to the Euler equation. It follows also that a necessary condition for blow-up at some time $T_*$ is that \begin{equation}\label{BasicBUC}\lim_{t\rightarrow T_*}\int_{0}^t \|\nabla u\|_{L^\infty}=+\infty. \end{equation} This result is somewhat unsatisfying since to determine whether there is a blow-up in a particular scenario, or likewise to find a good candidate for a blow-up, one must study the behavior of the whole matrix $\nabla u$. This approach is limited by our poor understanding of the corresponding pressure term $D^2p$. 

An important improvement on the basic blow-up criterion \eqref{BasicBUC} is that of Beale, Kato, and Majda \cite{bkm}. In particular, they observed that there is a little more than a logarithm of room in the basic energy estimate \eqref{BasicEE} while there is only a logarithmic loss in replacing $\|\nabla u\|_{L^\infty}$ by $\|\nabla\times u\|_{L^\infty}$, since we know that $\div(u)=0$. While this improvement appears on the surface to be merely technical, it is actually paradigmatic in that we now only need to study the evolution of the vorticity $\omega=\nabla\times u$. Indeed, $\omega$ now evolves by the equation:
\[\partial_t\omega + u\cdot\nabla\omega = \omega\cdot\nabla u,\] and in fact the equation is closed due to the fact that
\[u=(-\Delta)^{-1}\nabla\times \omega.\] There are, strictly speaking, ``better" analytical blow-up criteria than the BKM (see, for example, \cite{Plan}). However, the BKM criterion is the \emph{best} in that it gives us something concrete to look for, pointwise growth, on a quantity, $\omega$, whose evolution is understandable. 

A second important improvement is due to Constantin, Fefferman, and Majda \cite{CFM96}. This criteria answers the basic question: \emph{Can a 3d Euler solution blow up in a two-dimensional way?} Since 2d Euler solutions are global, the answer should be ``no" in a general setting. In particular, the authors of \cite{CFM96} proved that if the velocity field $u$ remains uniformly bounded up to the blow-up time and if the \emph{direction} of vorticity remains $C^1$, then no blow-up can occur. In other words, in order that a singularity occur, the vorticity vector must change direction very quickly. We remark that if a singularity were to occur at a point where the vorticity vanishes for all $t<T_*$ (as occurs in \cite{E_Classical}), the direction of vorticity becomes discontinuous at the time of blow-up. If the vorticity is nowhere vanishing initially, it is possible that the criteria of \cite{CFM96} could be helpful in ruling out blow-up. See Question \ref{Sullivan}. 

A natural question one could ask when discussing blow-up criteria is whether they are sharp. In fact, it is possible to show that the BKM criterion is actually sharp in several senses. In particular, if we have a solution with $|\omega|_{L^\infty}\approx \frac{1}{T_*-t}$ as $t\rightarrow T_*$, then the BKM criterion is clearly borderline. This was done in setting of $C^{1,\alpha}$ solutions in \cite{E_Classical} and it is not difficult to deduce that the BKM criterion is actually \emph{sharp} in the scale of $L^p$ spaces for general $C^{1,\alpha}$ solutions of the Euler equation. In other words, given $p<\infty$, there are classical $C^{1,\alpha}$ solutions that remain bounded in $L^p$ uniformly up to the blow-up time. This phenomenon is present in many of the basic models where we know finite-time singularity. This is not difficult to show in the 1d Burgers equation, for example. However, an interesting consequence of \cite{CGM} and a direct calculation gives the following

\begin{myshade}
  \vspace{-1mm}
\begin{lemma}
If $u$ is a smooth solution to the 1d Burgers equation:
\[u_t +uu_x=0.\] Then, if $u$ becomes singular at $t=T_*$,  we have
\[\sup_{t\in[0, T_*)} \|u_x(t)\|_{L^{3/2}}=+\infty.\]
\end{lemma}
\end{myshade}

Note that there are smooth solutions to the Burgers equation that becomes singular so that $\|u_x\|_{L^p}<\infty$ for any $p<3/2$ uniformly up to the blow-up time. 

We thus close the section on blow-up criteria with the following conjecture, the meaning of which is that any smooth Euler solution that becomes singular in finite time must do so with a minimal intensity. Ideally, one would hope that this conjecture can be established \emph{without} a full resolution of the blow-up problem.

\begin{myshade2}
  \vspace{-1mm}
\begin{conjecture}\label{bkmconj}
There exists a universal $p_*<\infty$ so that if a smooth solution to the Euler equation $\omega\in C^\infty_c(\mathbb{R}^3\times [0,T_*))$ becomes singular as $t\rightarrow T_*$, then 
\[\sup_{t\in[0, T_*)}  \|\omega(t)\|_{L^{p_{*}}}=+\infty.\]
\end{conjecture}
\end{myshade2}

\begin{remark}
If we are looking at $C^{1,\alpha}$ solutions, necessarily $p_*$ must depend on $\alpha$ and $p_*(\alpha)\rightarrow\infty$ as $\alpha\rightarrow 0$. See \cite{EGM3dE}.
\end{remark}

\subsubsection{Lower Dimensional Models: Advection and Vortex Stretching}\label{AdvectionAndVS}

Aside from proving bounds, understanding the \emph{dynamics} of solutions of the Euler equation presents a significant challenge in any dimension. In dimension two, the vorticity is transported by the velocity field that is determined through the Biot-Savart law $u=(-\Delta)^{-1}\nabla^\perp \omega$. Because of the simple nature of transport, there are a number of special types of solutions one can study to gain intuition about general solutions: one can consider point vortices, vortex sheets, vortex patches, etc. Such solutions allow us to visualize vortex motion in two dimensions. In three dimensions the situation is a bit more grim and our understanding is much weaker, the main difference being that vorticity is no longer transported but is now also ``stretched:"
\begin{equation}
\label{GeneralVorticityEqn}
\partial_t\omega + u\cdot\nabla\omega=\omega\cdot\nabla u.
\end{equation} 
A natural approach to understanding the vorticity equation is thus to simply understand the equation \eqref{GeneralVorticityEqn} in various settings, forgetting about the link between $u$ and $\omega$. It turns out that a number of the classical models of the Euler equation such as the SQG equation \cite{CMT} and other active scalar equations \cite{KReview}, the De Gregorio equation \cite{DG}, and many others can be recast in the general form \eqref{GeneralVorticityEqn}, the defining feature of the equation being the relationship between $u$ and $\omega$. Even in one dimensional renditions, the combination of transport \emph{and} stretching is significantly richer than just transport. 

Before discussing specific examples, it is important to remark that there can be no model that describes the dynamics of \emph{all} solutions of the vorticity equation. At best, we should hope for a model that accurately describes the behavior of some solutions perhaps in a symmetry class or even more restrictive class. In fact, the more universal of a model we try to give, the less likely we are to understand it any better than we do the Euler equation itself. At the same time, there are a few guiding principles we can use to devise models that may be relevant to some Euler solutions:
\begin{itemize}
\item The vorticity is transported and stretched by the same velocity field $u$. It is not necessary that $u$ should always be divergence-free since the model may only describe lower dimensional dynamics. 
\item That the velocity field $u=(-\Delta)^{-1}\nabla\times \omega$ is determined through an operator of degree $-1$ is important though it could be of even less degree depending on the scenario. 
\item The parity between $u$ and $\omega$ is important, but also may change depending on the scenario. 
\end{itemize}

\subsubsection{General Ideas}

We now discuss a little bit the general intuition behind the \emph{linear} equation 
\begin{equation}
\label{1dVort}\partial_t \omega + u\partial_x\omega=\omega\partial_x u,
\end{equation} where $u$ is a given time-independent mean-zero function on $\mathbb{S}^1.$ 
Let $x_0, x_1$ be two consecutive zeros of $u$ and let us assume that $u>0$ in $(x_0, x_1)$. while $\partial_x(x_0)>0$ and $\partial_x u(x_1)<0$. Then, using a simple calculation, it is easy to show that smooth solutions of \eqref{1dVort} in $(x_0,x_1)$ can be written as:
\[\omega(x,t)= \omega(x_0,0)F_*(x,t)+\partial_x \omega(x_0,0)u(x)+ \Omega(x,t),\] where 
\[\|\Omega\|_{L^2}\leq C\exp(-ct)\] and 
$F_*(x,t)$ is just the solution of \eqref{1dVort} with data $\omega\equiv 1$ on $[x_0,x_1]$. Note that it is not difficult to see that $F_*$ grows exponentially. 

The above decomposition on $(x_0,x_1)$ tells us that if $u$ has only finitely many zeros that are all non-degenerate, the smooth solutions of \eqref{1dVort} all approach a finite dimensional family of solutions exponentially fast. The key to the proof is that $\frac{\omega}{u}$ is transported and any smooth solution to the transport equation 
\[\partial_t f+u\partial_x f=0\] vanishing initially at $x_0$ will decay exponentially in $L^2([x_0,x_1])$ as $t\rightarrow\infty$. What this means is that smooth solutions to \eqref{1dVort} evolve almost completely based on the Taylor expansion of the data near the zeros of $u$: the constant part is magnified at stationary points where $\partial_x u>0$, the first derivative term is preserved, and the higher order terms are dissipated. This (linear) behavior appears to be similar to what has been found even in the nonlinear case and it is tempting to conjecture that this behavior is ``universal" even among nonlinear 1d vorticity equations, the main difference being that what happens in infinite time for the linear problem may happen in finite time in the non-linear problem. We remark that the presence of advection is what leads to a relatively simple dynamics in \eqref{1dVort}. Without advection, the dynamics of the linear problem would be completely infinite dimensional.  
Let us now discuss two simple non-linear examples.

\subsubsection{Examples and a conjecture} Consider the following projection operator \[\mathbb{P}_1(f)=\sin(x)\int_{\mathbb{S}^1} f(x)\sin(x)\rmd x.\] Now consider the non-linear problem:
\begin{equation}\label{1dEqn1}\partial_t\omega + u\partial_x\omega=\omega\partial_x u,\end{equation}
\[u=\mathbb{P}_1(\omega).\]
Interestingly, we can give a complete characterization of blow-up and global regularity in this problem. We may write:
\[u=\lambda(t)\sin(x),\qquad \lambda(t)=\int_{\mathbb{S}^1} \omega(x,t) \sin(x)\rmd x.\]
Thus, \eqref{1dEqn1} becomes:
\[\partial_t \omega +\lambda(t) \sin(x)\partial_x\omega =\lambda(t)\cos(x)\omega.\]
It follows that 
\[\partial_t \Big(\frac{\omega}{\sin(x)}\Big)+\lambda(t)\sin(x)\partial_x\Big(\frac{\omega}{\sin(x)}\Big)=0.\]
Thus, letting $\dot{X}_t(a) = \lambda(t)\sin(X_t(a))$ with $X_0(a)=a$,
we get that
\[\frac{\omega(X_t(a),t)}{\sin(X_t(a))} = \frac{\omega_0(x)}{\sin(x)}, \qquad \implies \qquad \omega(x,t)=\omega_0(X_t^{-1}(x),t)\frac{\sin(x)}{\sin(X_t^{-1}(x))}.\]
It follows that 
\[\lambda(t)=\frac{1}{2\pi}\int_{0}^{2\pi} \omega_0(X_t^{-1}(x),t)\frac{\sin^2(x)}{\sin(X_t^{-1}(x))}\rmd x.\]
Now let's solve for $X$. We get $\frac{\dot{X}_t(a)}{\sin(X_t(a))} = \lambda(t).$ If we let $\Lambda(t)=\int_{0}^t\lambda(s)ds$ we get:
\[\frac{\rmd}{\rmd t}\log(\tan(X_t(a)/2))=\lambda(t) \qquad \implies \qquad \log\frac{\tan(x/2)}{\tan(X_t^{-1}(x)/2)}=\Lambda(t).\]
Thus, we obtain
\[X_t^{-1}(x)=2\arctan(e^{-\Lambda(t)}\tan(x/2)).\] 
We now only need to consider
\[\Lambda'(t)=\frac{1}{2\pi}\int_{-\pi}^{\pi} \omega_0(2\arctan(\exp^{-\Lambda(t)}\tan(x/2))\frac{\sin^2(x)}{\sin(2\arctan(\exp^{-\Lambda(t)}\tan(x/2))}\rmd x\]
\[=\frac{1}{2\pi}\int_{-\pi}^\pi \omega_0(2\arctan(\exp^{-\Lambda(t)}\tan(x/2))\frac{\exp(\Lambda(t))\sin^2(x)\Big(1+\exp(-2\Lambda(t))\tan^2(x/2)\Big)}{2\arctan(x/2)}.\]
Therefore we arrive at the expression:
\[\Lambda'(t)=\frac{\exp(\Lambda(t))}{4\pi} \int_{-\pi}^{\pi} \omega_0 (...) \frac{\sin^2(x)}{\tan(x/2)}\rmd x+\frac{\exp(-\Lambda(t))}{4\pi} \int_{-\pi}^{\pi} \omega_0(...)\sin^2(x)\tan(x/2)\rmd x.\]
Notice, however, that the functions $\frac{\sin^2(x)}{\tan(x/2)}$ and $\sin^2(x)\tan(x/2)$ are smooth $2\pi$-periodic functions with zero mean on $[-\pi,\pi]$. This is a key observation. 
It is now easy to see that if $\omega_0\in C^1$ we have global regularity. Indeed, in the first integral, we use the mean-zero property of $\frac{\sin^2(x)}{\tan(x/2)}$ and we get:
\[|\omega_0(2\arctan(\exp(-\Lambda(t)\tan(x/2))))-\omega_0(0)|\leq C|\omega_0|_{C^1}\exp(-\Lambda(t))\tan(x/2).\] Arguing similarly in the second integral (subtracting $\omega_0(\pi)$ this time) we get that
\[|\Lambda'(t)|\leq C\|\omega_0\|_{C^1}, \] for some universal constant $C$. This bound is sharp since $\lambda(t)=\Lambda'(t)$ can be constant for all time.
We arrive at the following theorem. 

\begin{myshade}
  \vspace{-1mm}
\begin{theorem}
$C^1$ solutions to \eqref{1dEqn1} are globally regular and $\lambda(t)$ is uniformly bounded for all time in that case.
\end{theorem} 
\end{myshade}

It is also clear from the above that we can construct $\omega_0\in C^\alpha$ for any $\alpha<1$ for which $\Lambda(t)$ satisfies the bound:
\[\Lambda'(t)\geq c \exp((1-\alpha)\Lambda(t)).\] This leads to a finite-time singularity and the following

\begin{myshade}
  \vspace{-1mm}
\begin{theorem}
For any $\alpha<1$ there are $C^\alpha$ solutions to \eqref{1dEqn1} that become singular in finite time. 
\end{theorem} 
\end{myshade}

It is interesting that the sharp regularity threshold is not $C^1$ but is rather given by a (sharp) Osgood-type condition on the modulus of continuity of $\omega_0$. For example, even when $\omega_0$ has a modulus of continuity like $r|\log r|$ there is still global regularity. 

\subsubsection{$\partial_x u=\mathbb{P}_1(\omega)$}
We now consider the problem 
\begin{equation}\label{1dEqn2}\partial_t\omega + u\partial_x\omega=\omega\partial_x u,\end{equation}
\[\partial_x u=\mathbb{P}_1(\omega).\] This problem may seem to be a trivial modification of the preceding one since the system simply becomes (after a translation): 
\[\partial_t \omega + \lambda(t)\sin(x)\partial_x\omega=\lambda(t)\cos(x)\omega,\]
where
\[\lambda(t)=\int_{0}^{2\pi}\omega(x,t)\cos(x)\rmd x.\] The strange fact is that smooth solutions to \emph{this} problem now can develop a singularity in finite time. This issue is that the parity of the velocity is the \emph{opposite} of the parity of the vorticity, in a sense. 
Indeed, for this problem we arrive at the equation:
\begin{align*}
\Lambda'(t)&=\frac{\exp(\Lambda(t))}{4\pi} \int_{-\pi}^{\pi} \omega_0 (2\arctan(\exp^{-\Lambda(t)}\tan(x/2)) \frac{\sin(x)\cos(x)}{\tan(x/2)}\rmd x\\
&\qquad +\frac{\exp(-\Lambda(t))}{4\pi} \int_{-\pi}^{\pi} \omega_0(2\arctan(\exp^{-\Lambda(t)}\tan(x/2))\sin(x)\cos(x)\tan(x/2)\rmd x.
\end{align*}
The key difference now is that $\frac{\sin(x)\cos(x)}{\tan(x/2)}$ and $\sin(x)\cos(x)\tan(x/2)$ are \emph{not} mean-zero. The following theorem is an exercise.

\begin{myshade}
  \vspace{-1mm}
\begin{theorem}
There exist analytic solutions to \eqref{1dEqn2} on $\mathbb{S}^1$ that blow up in finite time.
\end{theorem}
\end{myshade}

While the above models appear to be somewhat overly-simplified, we do gain some very important pieces of information about the dynamics of the general equation 
\[\partial_t\omega+u\partial_x\omega=\omega\partial_x u,\] 
with some functional relation between $u$ and $\omega.$ Generally speaking, the picture we get is that:
\begin{itemize}
\item Growth of vorticity is coming from the vortex stretching term $\omega\partial_x u.$
\item The transport term depletes the growth. 
\item Whether the growth or depletion wins depends on the regularity of the solution \emph{and} the parity of the mapping relating $u$ and $\omega$. 
\end{itemize}
The above observations along with the existing results on various models of this type motivates us to make the following guess about the nature of solutions of the 1d vorticity equation for general Biot-Savart law $\omega\rightarrow u$. As we will discuss below, some aspects of the conjecture have already been established on a number of non-trivial models.

\begin{myshade2}
  \vspace{-1mm}
\begin{conjecture}\label{1dModelConjecture}
Consider the 1d vorticity equation on $\mathbb{S}^1$:
\[\partial_t\omega+u\partial_x\omega=\omega\partial_x u,\qquad u=\mathcal{K}(\omega),\] where $\mathcal{K}$ is a non-trivial Fourier integral operator with  bounded symbol $m$ satisfying $|m(k)|\leq \frac{C}{1+|k|}$. 
\begin{itemize}
\item In general, $H^s$ solutions should develop singularities when $s<3/2$. 
\item If $m$ is odd, there exist analytic solutions that become singular in finite time. 
\item If $m$ is even, $H^s$ solutions must be global whenever $s>\frac{3}{2}$. 
\end{itemize}
\end{conjecture}
\end{myshade2}

\noindent We have written the bullet points of the conjecture in order of perceived difficulty. Note that the two models we discussed above exhibit exactly the behavior described in the above conjecture. 
\begin{remark}
It is important to emphasize that the spatial domain in the above conjecture is $\mathbb{S}^1$. When the spatial domain has a boundary such as the case of domains like $[a,b]$, the results of the conjecture do not necessarily hold. In particular, there may be singularities no matter the parity of $m$--see \cite{SarriaSaxton} as an example. 
\end{remark}

Let us now move to discuss several of the existing results on non-trivial models that seem to support the above conjecture. 

\subsubsection{Overview of Results on 1d Models of the Vorticity Equation}

As mentioned above, the first 1d model of the vorticity equation that we are aware of is the model of Constantin, Lax, and Majda \cite{CLM}, where only the effect of vortex stretching is studied.  A related model is the De Gregorio model:
\[\partial_t\omega+u\partial_x\omega=\omega\partial_x u \]
\begin{equation}\label{DGBSLaw}\partial_x u = H(\omega),\end{equation} where $H$ is the Hilbert transform defined on $L^2(\mathbb{R})$ by:
\[H(f)(x)=\frac{1}{\pi}P.V. \int \frac{f(y)}{x-y}\rmd y.\] Alternatively, we can write:
\[\mathcal{F}\Big(Hf\Big) (\xi)=-i\text{sgn}(\xi)\mathcal{F}(f)(\xi),\] where $\mathcal{F}$ is the Fourier transform. The choice of $H$ in \eqref{DGBSLaw} is due to the desire to have the relation between $\partial_x u$ and $\omega$ to be translation invariant, zeroth order, and truly a singular integral operator (as opposed to just the identity on the right hand side, for example). Numerical computations of De Gregorio himself \cite{DG} indicated that solutions to his equation may be global, but major theoretical advances in this direction were only made just a few years ago. Since the symbol relating $u$ and $\omega$ in \eqref{DGBSLaw} is even, Conjecture \ref{1dModelConjecture} indicates that solutions to the De Gregorio model with high smoothness should be globally regular while low regularity solutions should become singular. All of the recent advances in the study of the De Gregorio equation have confirmed these predictions. 

In the direction of singularity formation for low regularity data, advances were made in \cite{EJDG} where self-similar blow-up profiles were constructed via perturbative methods. These blow-up profiles were studied further in the works \cite{EGM} and \cite{CHHDG} where asymptotic stability was also established. In the very recent work \cite{ChenDGMain}, J. Chen proved the first bullet point in \ref{1dModelConjecture} for the De Gregorio model using a different type of perturbative method that perturbs a steady state to construct a blow-up profiles with smoothness $C^{1-\epsilon}$.  In the direction of global regularity, the authors of \cite{JSS} observed a very interesting mechanism for asymptotic stability of steady states in the De Gregorio model and gave the first results on global regularity for non-trivial solutions. Thereafter, it was shown that non-negative regular solutions are global \cite{LeiLiuRen}. Most recently, this was extended to sign changing solutions under some symmetry hypothesis \cite{ChenDGMain}. 

Let us close by remarking that there are several other 1d models of fluid problems related to 3d Euler \cite{KReview, CKY, Choi2017, DKX, HK, KT}.  While these models are not models of the vorticity equation \emph{per se}, they tend to exhibit several similar features to the models we described above. One important example that satisfies the assumptions and conclusions of Conjecture \ref{1dModelConjecture} is the equation:
\[\partial_t\omega+u\partial_x\omega=\omega\partial_x u,\qquad \partial_{xx}u=\omega.\]
See \cite{SarriaSaxton} for more details.

\subsubsection{Symmetry Reductions and Special Ansaztes}

A natural first step to understand the dynamics of real solutions to the Euler equation is to consider solutions in symmetry. The solution map for the Euler equation enjoys a number of important symmetry properties. In particular, whenever $u$ is a solution, $\lambda, \tau>0$ are constants, and $\mathcal{O}$ satisfies $\mathcal{O}\mathcal{O}^t={\rm I}$, we have that
\[\frac{1}{\lambda}u(\lambda x, t),\qquad \mathcal{O}^t u(\mathcal{O}x,t),\qquad \tau u(x, \tau t),\] are automatically solutions. It follows that solutions belonging to an existence and uniqueness class must retain symmetries of the above type. Up to now, there are three principle types of symmetries:

\begin{itemize}
\item Rotationally symmetric and reflection symmetric solutions

\item Scale-Invariant solutions

\item Self-Similar solutions.

\end{itemize}

There are also solutions that are a combination of various symmetries. For example, helically symmetric solutions \cite{ET} are gotten through a combination of the translation and rotation symmetries. 

\subsubsection{Rotation and Reflection Symmetries}

Rotation and reflection symmetries are helpful to reduce the complexity of the Euler equation. A particular example is the axi-symmetry assumption. We let \[r=(x_1^2+x_2^2)^{1/2}.\]

\begin{myshade}
  \vspace{-1mm}
\begin{definition}
A velocity field $u$ is said to be axi-symmetric (with respect to $x_3$) if for all rotation matrices $\mathcal{O}$ fixing the $x_3$ axis we have:
\[u(\mathcal{O}x)=\mathcal{O}u(x).\]
\end{definition} 
\end{myshade}

This immediately implies that $u_3$ is a function of $r$ and $x_3$ only. Moreover, expanding $u_h=(u_1,u_2)$ in a Fourier series in $\theta=\arctan\frac{x_2}{x_1}$, we see that 
\[u_h(r,x_3,\theta)=\sum_{k}u_h^k(r,x_3)\exp(i k\theta).\] Thus, we have that 
\[\sum_k u^k_h(r,x_3)\exp(ik (\theta+\alpha)) = \exp(i\alpha)\sum_k u^k_h(r,x_3)\exp(ik \theta)\] for all $\alpha\in\mathbb{R}$. It follows that $u_h^k\not =0$ if and only if $k=1$. We have established the following Lemma.

\begin{myshade}
  \vspace{-1mm}
\begin{lemma}
Any $L^2$ axi-symmetric velocity field can be written as:
\[u=u_r(r,x_3)e_r+u_\theta(r,x_3) e_\theta+u_3(r,x_3) e_3,\] where 
$e_r=\frac{1}{r}(x_1,x_2,0)$, $e_\theta = \frac{1}{r}(-x_2,x_1,0)$, and  $e_3=(0,0,1)$.
\end{lemma}
\end{myshade}

Now via some simple identities\footnote{$e_r\cdot\nabla  = \partial_r, e_\theta\cdot\nabla = \partial_\theta, e_3\cdot\nabla=\partial_3,$ while
$\partial_i e_r =\partial_i e_\theta=\partial_i e_3=0,$ for $i\in \{r,3\},$ and
$\partial_\theta e_r=e_\theta, \partial_\theta e_\theta = -e_r, \partial_\theta e_3=0.$}, we find that axi-symmetric solutions of the 3d Euler equation satisfy a simplified equation.

\begin{myshade}
  \vspace{-1mm}
\begin{corollary}
Axi-symmetric solutions $u(\cdot,t)$ of the Euler equation satisfy
\begin{equation}\label{ASE}\partial_t v+v\cdot\nabla v+ \nabla p =(\frac{(u_\theta)^2}{r},0),\end{equation}
\begin{equation}\nabla\cdot(r v)=0,\end{equation}
\begin{equation}\label{swirl}\partial_t(ru_\theta)+v\cdot\nabla (ru_\theta) = 0,\end{equation}
where $v=(u_r, u_3)$ and $\nabla=(\partial_r,\partial_3)$.
\end{corollary} 
\end{myshade}

Even though axi-symmetric solutions are governed by a system of two evolution equations in two spatial variables $(r,x_3),$ the qualitative behavior of solutions remains largely open. As we shall see below, even in the case where $u_\theta\equiv 0$ we are able to make highly non-trivial statements about the dynamics. 

\subsubsection{Scale-Invariant Solutions}\label{SISolutions}

We now move to discuss scale-invariant solutions. We start with a formal definition of scale-invariance, where the ``scale" we choose is obviously motivated by the Euler equation.

\begin{myshade}
  \vspace{-1mm}
\begin{definition}[Scale Invariance]
A velocity field $u$ is said to be scale invariant if for every $\lambda>0$ we have that 
\[u(\lambda x)=\lambda u(x).\]
\end{definition}
\end{myshade}

 Observe that this level of homogeneity is the only one that could possibly be propagated by solutions except perhaps in trivial situations.  Surprisingly, these solutions seem to have only been discovered and analyzed in recent years \cite{EJSymm,E1,EJB,EJE,EIS}. One reason appears to be that such solutions are automatically linearly growing at spatial infinity and are not smooth except in special situations or possibly in domains that exclude the origin. We remark that it is an interesting problem to classify those solutions that satisfy the no-penetration boundary condition on annular domains in $\mathbb{R}^3$, though it is known that there is a large number of non-trivial solutions of this type. Let us remark here that the assumption of scale invariance essentially reduces the complexity of full 3d solutions to essentially a 2d problem. If it were also possible to combine this with another symmetry like axi-symmetry, then we would have reduced the problem to a one-dimensional one that would be more amenable to analysis. 

The key difficulty in analyzing scale-invariant solutions is to define the pressure in the Euler equation:
\[\partial_t u+u\cdot\nabla u+\nabla p=0.\] In particular, to restrict the Euler equation to scale-invariant solutions we would need to show that if $u$ is scale-invariant then $\nabla p$ is scale invariant, where $p$ is determined by:
\begin{equation}\label{pEq}-\Delta p = \div(u\cdot\nabla u). \end{equation} In terms of scaling, it would seem that if $u$ were scale-invariant, then the right hand side of \eqref{pEq} is 0-homogeneous spatially so that the solution $p$ should be $2$-homogeneous. If this "scaling counting" were rigorous, then we would have no trouble to directly analyze scale invariant solutions. Unfortunately, when it comes to the Poisson equation, integer homogeneity counting is in general false and there need to be logarithmic corrections. In particular, there exist $u$ that are scale-invariant for which \emph{any} solution to \eqref{pEq} grows like $|x|^2\log|x|$ when $|x|\rightarrow\infty$. The extra logarithm can be viewed as a "resonance" due to the harmonic polynomials of degree two. In particular, the scaling problem can be solved by imposing further restrictions on $u$ that ensure that this resonance does not occur. To make this precise, let us first make some definitions. First we define \[\mathcal{H}_2=\left\{\frac{P(x)}{|x|^2}: P\,\, \text{is a harmonic polynomial of degree 2}\right\}.\]

\begin{myshade}
  \vspace{-1mm}
\begin{definition}
A function $f:\mathbb{S}^{d-1}\rightarrow\mathbb{R}$ is said to belong to $\mathcal{H}_2^\perp$ if \[\int_{\mathbb{S}^{d-1}} f(\sigma)h(\sigma)d\sigma=0,\] for all $h\in\mathcal{H}_2$. 
\end{definition}
\end{myshade}

The following Lemma was established in \cite{EJSymm} in the case $d=2$, though the proof carries over word by word to general $d\geq 2$.

\begin{myshade}
  \vspace{-1mm}
\begin{lemma}
Let $f\in L^\infty(\mathbb{R}^d)$ and that $f(r,\cdot)\in \mathcal{H}_2^\perp$ for all $r\in [0,\infty)$. Then, there exists a {unique} $\psi\in W^{2,p}_{loc}$ for all $p<\infty$ with \[\Delta\psi=f,\]  $|\psi(x)|\leq C|x|^2$ on $\mathbb{R}^d,$ and with $\psi(r,\cdot)\in \mathcal{H}_2^\perp$ for all $r\in [0,\infty)$. 
\end{lemma}
\end{myshade}

An important corollary of this Lemma is.

\begin{myshade}
  \vspace{-1mm}
\begin{corollary}\label{PropagationHomogeneity}
Given a $0$-homogeneous function $f\in\mathcal{H}_2^\perp$, there is a unique $2$-homogeneous function $\psi\in\mathcal{H}_2^\perp$ solving
\[\Delta\psi=f.\]
\end{corollary}
\end{myshade}

From the above, it is natural to conjecture that we can give a sense to scale-invariant solutions as long as the right hand side of \eqref{pEq} belongs to $\mathcal{H}_2^\perp$ automatically. It does not seem that $\mathcal{H}_2^\perp$ itself is an invariant set for the solution map; however, using considering solutions that also satisfy rotational and/or reflection symmetries, we can force the solutions we consider to be under the domain of applicability of Corollary \ref{PropagationHomogeneity}. In two dimensions, there is a very nice analogue of the Yudovich theory that allows for the analysis of scale-invariant solutions.

\begin{myshade}
  \vspace{-1mm}
\begin{theorem}\label{YudovichNoDecay}(\cite{EJSymm})
Given an initial vorticity $\omega_0\in L^\infty$ that is $m$-fold symmetric for some $m\geq 3$, there exists a unique $L^\infty$ and $m$-fold symmetric solution to the 2d Euler equation.  
\end{theorem}
\end{myshade}

A corollary of this theorem is that every scale invariant initial data that is $\frac{2\pi}{m}$ periodic in $\theta$ on $\mathbb{R}^2$ gives rise to a unique scale invariant solution. In vorticity form, the equation takes the form of an active scalar equation on $\mathbb{S}^1$:
\begin{equation}\label{1dEuler}
\partial_t g + 2G\partial_\theta g=0,
\end{equation}
\[4G+\partial_{\theta\theta}G=g.\]
This equation is globally well-posed in $L^\infty$. Though some aspects of the long-time behavior of solutions of \eqref{1dEuler}, the long-time dynamics is not yet resolved.

\begin{myshade3}
  \vspace{-1mm}
\begin{problem}\label{inftime3d}
Establish generic infinite-time singularity formation and convergence to compact orbits for smooth solutions to \eqref{1dEuler} in the infinite-time limit. 
\end{problem}
\end{myshade3}

When $d\geq 3$, there does not seem to be any analogue of Theorem \ref{YudovichNoDecay} due to issues of ill-posedness \cite{EM1}. However, if we add a bit of regularity we can still give a sense to scale-invariant solutions. To do this, we need to make two definitions, the first describing admissible symmetry classes and the second regarding regularity.

\begin{myshade}
  \vspace{-1mm}
\begin{definition}
A collection $\mathcal{M}$ of orthogonal matrices is said to be an \emph{admissible symmetry class} if whenever $u\in C^\infty$ is $\mathcal{M}$-invariant, we have that 
\[\div(u\cdot\nabla u)\in\mathcal{H}_2^\perp\] on every sphere centered at the origin.  
\end{definition}
\end{myshade}

\begin{myshade}
  \vspace{-1mm}
\begin{definition}
Given $0\leq \alpha\leq 1$, a function $f\in L^\infty(\mathbb{R}^d)$ is said to belong to $\ring{C}^{0,\alpha}$ if \[|f|_{\ring{C}^{0,\alpha}}:=\sup_{x\not=y}\frac{||x|^\alpha f(x)-|y|^\alpha f(y)|}{|x-y|^\alpha}<\infty.\]
\end{definition}
\end{myshade}

Now the main theorem in this direction is

\begin{myshade}
  \vspace{-1mm}
\begin{theorem}
Fix $0<\alpha<1$ and an admissible symmetry class $\mathcal{M}$. Then, the Euler equation is locally well-posed in the class of velocity fields $u$ that are $\mathcal{M}$-invariant with $\nabla u\in\ring{C}^{0,\alpha}$. 
\end{theorem}
\end{myshade}

\begin{remark}
The interested reader can look at the thesis of I. Jeong \cite{JThesis} and the papers \cite{EJSymm,E1,EJB,EJE}
\end{remark}
As we have already remarked, the utility of this theorem is that scale-invariant solutions will satisfy a lower dimensional equation, which should be easier to analyze. This gives an interesting approach to the blow-up problem that yielded the first examples of blow-up for finite-energy strong solutions to the 3d Euler equation \cite{EJB,EJE,EJEO}. We must remark, however, that there remain numerous interesting problems in this direction. One important one is

\begin{myshade3}
  \vspace{-1mm}
\begin{problem}\label{scaleinvprob}
Does there exist a scale invariant solution with $\nabla u\in \ring{C}^\infty$ belonging to an admissible symmetry class that develops a singularity in finite time? 
\end{problem}
\end{myshade3}
Here, the space $\ring{C}^\infty$  (as defined in \cite{EJB}) is the space of functions $f\in C^\infty(\mathbb{R}^3\setminus \{0\})$ for which $|x|^k D^k f$ is bounded for all $k=0,1,2,...$


\subsubsection{Self-Similar Solutions}
Another type of solution that is more amenable to analysis than full time-dependent 3d solutions are self-similar solutions. In vorticity form, these are solutions of the form:
\[\omega(x,t)=\frac{1}{T_*-t} \Omega\Big(\frac{x}{(T_*-t)^\lambda}\Big),\] where $\lambda\in\mathbb{R}$. We remark that the power of $1$ on the coefficient $\frac{1}{T_*-t}$ is the only self-consistent one, as can be readily checked. The value of $\lambda$ is free and there may be different "branches" of self-similar solutions for different values of $\lambda$. There are a few important ideas to keep in mind when studying self-similar solutions:
\begin{itemize}

\item $\lambda>0$ corresponds to a concentrating profile in which the singularity occurs only at the origin at $t=T_*$. $\lambda<0$ corresponds to an expanding profile in which the singularity occurs everywhere $t=T_*$. The fact that the time-dependent solution looks like it is "expanding" or "contracting" does not violate the divergence-free condition. 

\item It should not be expected that there are self-similar solutions where $\Omega$ decays fast. In fact, the decay of $\Omega$ \emph{should} be precisely $|x|^{-\lambda}$ when $\lambda>0$. Indeed, the profile would satisfy an equation of the form:
\[\Omega+\lambda z\cdot\nabla_z \Omega = B(\Omega,\Omega),\] where $B$ is bilinear. Thus, if a decaying solution exists, it must be such that $\Omega+\lambda z\cdot\nabla_z\Omega$ decays \emph{faster} than $\Omega$ itself. Under some quite soft assumptions, this implies that the profile decays precisely like $|x|^{-\lambda}$. We do not make this precise (and it may not be true in full generality since $\Omega$ may oscillate at spatial infinity), but see \cite{EJDG} for example for some settings where this is made precise.  

\item Even though self-similar solutions will decay slowly and will have infinite-energy in general, if we have a self-similar solution with $\lambda>0$, it is likely that their do exist finite-energy solutions nearby that behave asymptotically self-similar near the blow-up time. This even be the case when the solution is "unstable" since the instability generically would be only finite-dimensional.

\end{itemize}

Until the recent work \cite{E_Classical}, it was not known whether self-similar solutions to the 3d Euler equation exist and a complete theory remains elusive. What is clear is that their analysis is the most promising one available to attack the singularity problem in general. It is important to emphasize that self-similar solutions are not only useful for the blow-up problem, but also for numerous other problems related to non-trivial dynamics of solutions \cite{Elling,V18,DallasEtAl}. 
Note that there are several interesting no blow-up results under various assumptions on self-similar or nearly self-similar blow-ups (see \cite{CC21,CC212,ChaeShv}).

\subsection{Taming the Nonlocal Effects}

We now move to discuss the issue of finite-time singularity for the actual 3d Euler equation, where we recall it in vorticity form:
\begin{equation}\label{3dE2}\partial_t\omega+u\cdot\nabla\omega= \omega\cdot\nabla u,\end{equation}
\[u=\nabla\times(-\Delta)^{-1} \omega.\]
 Though this problem looks significantly more difficult than the corresponding one-dimensional equations we discussed before, we can still build quite a bit of intuition from the one-dimensional problems. For the benefit of the reader, we repeat some of these ideas:
 \begin{itemize}
 \item Growth tends to occur near a stagnation point where $\nabla u$ has a positive eigenvalue. 
 \item The transport term will deplete the growth though this depletion can be overcome when the vorticity is non-smooth in the expanding direction of $u$, in the presence of a physical boundary, \emph{or} for smooth vorticity without a boundary when parity allows. 
 \end{itemize}

In the coming pages we will describe the simplest setting in which the above principles, coupled with a deeper understanding of the Biot-Savart law
\[u=\nabla\times(-\Delta)^{-1}\omega,\]
can be successfully applied to give a singularity for "strong" and "classical" solutions of \eqref{3dE2}. We will also discuss the limitations of the constructions in those settings to achieve a singularity for smooth data without boundary and move on to discuss one avenue to get a smooth blow-up solution. 

\subsubsection{A first look at non-locality}\label{NonlocalEffect}
The nonlocal nature of incompressible fluids is captured in the pressure term in the velocity formulation and in the Biot-Savart law in the vorticity formulation. In both cases, the nonlocal operator involved is the solution to \[\Delta \psi =f.\] Let us make precise the meaning of this non-locality though the trivial observation that the Taylor expansion of $f$ at $0$ does not determine the Taylor expansion for $\psi$ at $0$. To be more precise, the terms in the Taylor expansion of $\psi$ corresponding to Harmonic polynomials are determined by the global behavior of $f$ while the other terms only depend on $f$ locally. We have already alluded to this in the discussion around Corollary \ref{PropagationHomogeneity}. In fact, as can be expected, the presence of these harmonic polynomials in the Taylor expansion of $f$ leads to a type of resonance that yields a logarithmic correction. The above considerations hint at a decomposition of the following type:
\[\Delta^{-1} = \mathcal{S}+\mathcal{R},\] where $\mathcal{S}$ is a singular part coming only from projection to certain the spherical harmonics and a more regular part $\mathcal{R}$. 

We will start with some examples.

\subsubsection{Examples illustrating the behavior on $L^\infty$}

We start with some examples illustrating the behavior of solutions of $\Delta\psi=f$. \\

\noindent \textbf{Radial Case.} Observe that if 
\[\psi(x)=|x|^{\beta},\] then
\[\Delta \psi = (\partial_{rr} + \frac{1}{r}\partial_r) (r^\beta)= \beta^2 |x|^{\beta-2}.\] 
Thus, up to a harmonic part, the solution to 
\[\Delta \psi = |x|^{-\gamma}\] is 
\[\psi=\frac{1}{(2-\gamma)^2}|x|^{2-\gamma},\] and we are interested in the case $\gamma\rightarrow 0$. Thus we see that:
\[|D^2\psi|\leq C|x|^{-\gamma},\] for $C$ independent of $\gamma$ as $\gamma\rightarrow 0$. 
It is in fact possible to show that $\mathcal{K}$ restricted to radial functions is bounded on $L^p$ \emph{independent} of $p$ and, in particular, on $L^\infty$. \\

\noindent \textbf{General homogeneous vorticity.} Now we look at more general negative homogeneous data:
\[\psi(r,\theta)= r^{\beta}m(\theta),\] where we leave $m$ free for now.
Then we see that
\[\Delta \psi(r,\theta)=(\partial_{rr}+\frac{1}{r}\partial_r + \frac{1}{r^2}\partial_{\theta\theta}) \psi =\Big(\beta^2m(\theta) + m''(\theta)\Big)r^{\beta-2}. \]
At the same time, 
\[|D^2\psi|\approx (|m|+|m'|+|m''|)r^{\beta-2}.\]
Thus, singular behavior occurs only when we have
\[|\beta^2 m + m''|\ll |m|+|m'|+|m''|.\]
This only occurs when:
$\beta$ is, or is very close to, some $n\in\mathbb{N}$ \emph{and} when $\beta^2 m(\theta)+m''$ is very small. 
For example, 
\[\psi(r,\theta)=r^\beta \sin(2\theta),\] and $\beta$ close to $2$ gives:
\[\Delta\psi(r,\theta)=(4-\beta^2)\sin(2\theta) r^{2-\beta}.\]
Thus, as $\beta\rightarrow 2$, we see that $|D^2\psi|=O(1)$ but $\Delta\psi=O(2-\beta)$.

\subsubsection{Expansion of the Biot-Savart Law}
We now place the preceding calculations in their proper context. The first is through a Lemma about the Taylor expansion of $\psi$ and can be seen as a pointwise estimate and the second is in an $L^2$ context.

\subsubsection{Pointwise Estimates}

\begin{myshade}
  \vspace{-1mm}
\begin{lemma}\label{PtWiseKeyLemma}
 Let $\mathbb{P}_2$ be the orthogonal projection on $L^2(\mathbb{S}^{d-1})$ to the spherical harmonics of order 2. Assume that $\Delta\psi=f$. Then, 
 \[\Big|\psi(x)-\psi(0)-x\cdot\nabla\psi(0)-|x|^2\int_{|x|}^\infty\frac{\mathbb{P}_2f(\rho,\sigma)}{\rho} \rmd \rho\Big|\leq C|f|_{L^\infty}|x|^2,\] where $C=C(d)$ is a dimensional constant. 
\end{lemma}
\end{myshade}

\begin{remark}
This Lemma may be stated similarly at higher orders of the Taylor expansion of $\psi$. A proof can be found in \cite{EJVP} in the case $d=2$, the general case being similar. Note also that the corresponding expansion for $\nabla \psi$ holds (again with a bound on $|f|_{L^\infty}$ only). This can be seen as a generalization of the key lemma of \cite{KS14}.
\end{remark}
Observe that the term
\[L(f):=\int_{|x|}^\infty \frac{\mathbb{P}_2f(\rho,\sigma)}{\rho} \rmd \rho\]
could be unbounded, even if $f$ is bounded. As we have remarked, a version of this Lemma was used in \cite{KS14} to establish double exponential growth for the gradient of some solutions of 2d Euler in domains with boundary. An important open problem is to study domains without boundary and it is tempting to try to use the above Lemma to give a simplified model to determine whether double exponential growth occurs without boundary. Toward this end, we may consider the model (written in polar coordinates):
\begin{equation}\label{2dFM}\partial_t\omega + \frac{1}{r}(\partial_\theta \psi \partial_r\omega-\partial_r\psi \partial_\theta\omega)=0, \qquad \psi(r,\theta) = r^2 \int_{r}^\infty \frac{\mathbb{P}_2(\omega)(\rho,\theta)}{\rho}\rmd \rho.\end{equation}

\begin{myshade3}
  \vspace{-1mm}
\begin{problem}\label{compsup}
Determine whether a smooth compactly supported solution to \eqref{2dFM} on $\mathbb{R}^2$ can have $\nabla\omega$ grow double exponentially. 
\end{problem}
\end{myshade3}

\begin{myshade}
  \vspace{-1mm}
\begin{lemma}\label{PtWiseKeyLemma}
 Let $\mathbb{P}_2$ be the orthogonal projection on $L^2(\mathbb{S}^{d-1})$ to the spherical harmonics of order 2. Assume that $\Delta\psi=f$. Then, 
 \[\Big|\psi(x)-\psi(0)-x\cdot\nabla\psi(0)-|x|^2\int_{|x|}^\infty\frac{\mathbb{P}_2f(\rho,\sigma)}{\rho} \rmd \rho\Big|\leq C|f|_{L^\infty}|x|^2,\] where $C=C(d)$ is a dimensional constant. 
\end{lemma}
\end{myshade}

\begin{remark}
In the model we have only retained the "singular part" of the stream function $\psi$ near $r=0$. It is natural to believe that the answer to the question above should be the same for 2d Euler, though it is not clear how to pass from one to the other.  
\end{remark}

\subsubsection{$L^2$ based estimates}
 
 We now move to understand the above expansion in an $L^2$ context that also allows us to give an expansion for $D^2\Delta^{-1}$ rather than just $\Delta^{-1}$ or $\nabla\Delta^{-1}$. We consider $\Delta\psi=f,$ where we take $f$ to be of the form
\[f(r,\theta)=\Omega(R,\theta),\] where
\[R=r^\alpha,\] and $\alpha>0$ is {\bf small.} Note that even if $\Omega$ is smooth, $\omega$ will generally only be $C^\alpha$ in space (though it may be $C^\infty$ away from the origin). 
So, we want to study the solution $\psi$ of
$\Delta \psi =f,$ which in polar coordinates reads
\[\partial_{rr}\psi + \frac{1}{r}\partial_r\psi+\frac{1}{r^2}\partial_{\theta\theta}\psi =f.\] 
Now let us make the above changes:
\[\psi(r,\theta)=r^2\Psi(R^\alpha,\theta),\qquad \omega(r,\theta)=\Omega(R,\theta).\]
Thus, 
\begin{equation}\label{BSAlpha}\alpha^2R^2\partial_{RR}\Psi+ (4\alpha+\alpha^2) R\partial_R\Psi+4\Psi +\partial_{\theta\theta}\Psi =\Omega.
\end{equation}

\subsubsection{The Regular Part}

We begin with a discussion of the regular part of $D^2\Delta^{-1}$ as it acts on $\Omega$.

\begin{myshade}
  \vspace{-1mm}
\begin{theorem}
Assume that \[\int_{0}^{2\pi}\Omega(R,\theta)\exp(in\theta)d\theta=0,\] for all $R$ and for $n=0,1,2$. 
Then, given $\Omega\in H^s$, there is a unique solution to \eqref{BSAlpha} satisfying 
\[\alpha^2|R^2\partial_{RR}\Psi|_{H^s}+\alpha|R\partial_R\Psi|_{H^s}+|\partial_{\theta\theta}\Psi|_{H^s}\leq C_s |\Omega|_{H^s},\] with the constant $C_s$ {\bf independent} of $\alpha$. 
\end{theorem}
\end{myshade}

\begin{proof}
We will only prove the $L^2$ estimate. The case $s\geq 1$ is easier after that. First, observe that $\Psi$ also satisfies the orthogonality conditions. 
Indeed, if we define \[\Psi_n(R)=(\Psi(R,\theta),\exp(in\theta))_{L^2_\theta}.\] We see that
\[\alpha^2 R^2\partial_{RR}\Psi_n+(4\alpha +\alpha^2)R\partial_R\Psi_n +(4-n^2)\Psi_n=0.\]
Consequently, 
\[\Psi_n(R)= c_{1,n}R^{\lambda_1}+c_{2,n}R^{\lambda_2},\] where $\lambda_1,\lambda_2$ are the roots of
\[\alpha^2 \lambda(\lambda-1)+(4\alpha+\alpha^2)\lambda+(4-n^2)=0,\] which is reduces to 
$\alpha^2 \lambda^2+4\alpha\lambda+(4-n^2)=0.$
Thus, 
\[\lambda=\frac{-2\pm n}{\alpha}.\] It is then easy to see that since we want $\Psi\in L^2$, we need $\Psi_{n}\equiv 0$ for $n=0,1,2$. 
Now coming back to the $L^2$ estimates, take
\[\alpha^2R^2\partial_{RR}\Psi+ (4\alpha+\alpha^2) R\partial_R\Psi+4\Psi +\partial_{\theta\theta}\Psi =\Omega,\]
multiply by $\partial_{\theta\theta}\Psi$, and integrate. 
Then we see:
\[|\partial_{\theta\theta}\Psi|_{L^2}^2-4|\partial_\theta\Psi|_{L^2}^2+\int \alpha^2 R^2\partial_{RR}\Psi \partial_{\theta\theta}\Psi +(4\alpha+\alpha^2)\int R\partial_R\Psi \partial_{\theta\theta}\Psi\leq |\Omega|_{L^2}|\partial_{\theta\theta}\Psi|_{L^2}.\]
The crucial point is that because of the orthogonality condition, we have that
\[|\partial_\theta\Psi|_{L^2}^2\leq \frac{1}{9}|\partial_{\theta\theta}\Psi|_{L^2}^2.\]
Moreover, 
\[(4\alpha+\alpha^2)\int R\partial_R\Psi \partial_{\theta\theta}\Psi=-(4\alpha+\alpha^2)\int R \partial_{R}\partial_\theta\Psi \partial_\theta\Psi=\frac{4\alpha+\alpha^2}{2}\int (\partial_\theta\Psi)^2.\]
Additionally, 
\begin{align*}
\int \alpha^2 R^2\partial_{RR}\Psi \partial_{\theta\theta}\Psi&=-\alpha^2\int R^2\partial_{RR\theta}\Psi \partial_\theta\Psi
= \alpha^2\int R^2|\partial_{R\theta}\Psi|^2+2\alpha^2\int R\partial_{R\theta}\Psi\partial_\theta\Psi\\
&= \alpha^2\int R^2|\partial_{R\theta}\Psi|^2-\alpha^2\int (\partial_{\theta}\Psi)^2.
\end{align*}
Thus we get:
\[\frac{1}{2}|\partial_{\theta\theta}\Psi|_{L^2}^2+ \alpha^2\int R^2|\partial_{R\theta}\Psi|^2+\frac{4\alpha-\alpha^2}{2}\int (\partial_\theta\Psi)^2 \leq |\Omega|_{L^2}|\partial_{\theta\theta}\Psi|_{L^2}.\]
The prove the rest, multiply by $R^2\partial_{RR}\Psi$ and integrate. This is left as an exercise. 
\end{proof}

\subsubsection{The Singular Part}
Obviously, the singular part must come from functions which are \emph{not} orthogonal to $\exp(in\theta)$ for some $n\in \{0,1,2\}$.
\[\alpha^2R^2\partial_{RR}\Psi+ (4\alpha+\alpha^2) R\partial_R\Psi+4\Psi +\partial_{\theta\theta}\Psi =\Omega.\]

\begin{myshade}
  \vspace{-1mm}
\begin{theorem}
Assume that \[\Omega(R,\theta)=F(R)\exp(2i\theta)\] and that $F\in L^2$. 
Then, the unique solution to \eqref{BSAlpha} with $\Psi\in L^2$ is
\[\Psi(R,\theta)=\frac{1}{4\alpha}L(F)(R)\exp(2i\theta)+\mathcal{R}(F),\] where
\[L(F)(R)=\int_{R}^\infty \frac{F(s)}{s}\rmd s\] and 
$\|\mathcal{R}\|_{H^s\rightarrow H^s}\leq C_s$ independent of $\alpha$. 
\end{theorem}
\end{myshade}

\begin{proof}
As in the proof above, we can write:
\[\Psi(R,\theta)=G(R)\exp(2i\theta)\] so that $G$ satisfies:
\[\alpha^2 R^2G''(R)+(4\alpha+\alpha^2)RG'(R)= F(R).\] Now we just have to solve this ODE. Observe that 
\[G''+\frac{4+\alpha}{\alpha R}G'(R)=\frac{F}{\alpha^2 R^2}.\]
It follows that, 
$\Big(G'R^{\frac{4+\alpha}{\alpha}}\Big)'=R^{\frac{4+\alpha}{\alpha}}\frac{F}{\alpha^2 R^2}$ and
thus
\[G'(R)=R^{-\frac{4+\alpha}{\alpha}}\int_{0}^Rs^{\frac{4+\alpha}{\alpha}}\frac{F(s)}{\alpha^2 s^2}\rmd s.\]
Integrating, we find
\begin{align*}
G(R)&=\int_{R}^\infty \tau^{-\frac{4+\alpha}{\alpha}}\int_{0}^\tau s^{\frac{4+\alpha}{\alpha}}\frac{F(s)}{\alpha^2 s^2}\rmd s=-\int_{R}^\infty \frac{\alpha}{4}\frac{d}{d\tau}\tau^{-\frac{4}{\alpha}}\int_{0}^\tau s^{\frac{4+\alpha}{\alpha}}\frac{F(s)}{\alpha^2 s^2}\rmd s\\
&=\frac{1}{4\alpha}\int_{R}^\infty \frac{F(s)}{s}\rmd s+\frac{1}{4\alpha}R^{-\frac{4}{\alpha}}\int_{0}^Rs^{\frac{4}{\alpha}} \frac{F(s)}{s}\rmd s\\
&=:\frac{1}{4\alpha}L(F)+\mathcal{R}(F).
\end{align*}
\end{proof}

\begin{myshade}
  \vspace{-1mm}
\begin{lemma} \label{lemmaRf} There is a constant $C$ independent of $0\leq\alpha\leq 1$ so that
\[|\mathcal{R}(F)|_{L^2}\leq C|F|_{L^2}.\] 
\end{lemma}
\end{myshade}
In Lemma \ref{lemmaRf}, we can take $C={1}/{56}$.

\begin{proof} The proof is similar to the proof of the Hardy Inequality. 
\begin{align*}
\int_{0}^\infty \mathcal{R}(F) F&= \frac{1}{4\alpha}\int_{0}^\infty F(R) R^{-\frac{4}{\alpha}}\int_{0}^Rs^{\frac{4}{\alpha}} \frac{F(s)}{s}\rmd s \\
&=\frac{1}{4\alpha}\int_{0}^\infty R^{-\frac{8}{\alpha}+1} R^{4/\alpha}\frac{F(R)}{R} \int_{0}^{R} s^{4/\alpha}\frac{F(s)}{s}\rmd s\\
&=\frac{1}{8\alpha}\int_{0}^{\infty} R^{-\frac{8}{\alpha}+1} \frac{d}{dR}\Big(\int_{0}^{R} s^{4/\alpha}\frac{F(s)}{s}\rmd s\Big)^2\\
&=\frac{1}{8\alpha}\Big(\frac{8}{\alpha}-1\Big)\int_{0}^\infty R^{-8/\alpha}\Big(\int_{0}^{R} s^{4/\alpha}\frac{F(s)}{s}\rmd s\Big)^2=\frac{64\alpha^2(8-\alpha)}{8\alpha^2}|\mathcal{R}(F)|_{L^2}^2.
\end{align*}
Thus, we obtain
$|\mathcal{R}(F)|_{L^2}\leq {\frac{1}{8(8-\alpha)}}|F|_{L^2}$.
\end{proof}

\subsubsection{Discussion}

We showed above two senses in which we can see that for merely bounded or almost merely bounded $f$ (like $C^\alpha$ for small $\alpha$), the leading order behavior of the solution of $\Delta\psi=f$ is given by:
\[\psi\left(|x|,\frac{x}{|x|}\right) \approx |x|^2\int_{|x|}^\infty\frac{\mathbb{P}_2(f)(\rho,\frac{x}{|x|})}{\rho}\rmd \rho.\] As we have remarked several times, similar expansions can be done at higher order and it would be interesting to see in what sense these higher expansions can lead higher order singular expansions in higher regularity spaces. We now move to discuss stable self-similar blow-up in a general context before applying these ideas to the Euler equation. 
 
\subsection{Self-Similar Singularity}\label{selfsimilar}
In the preceding section, we were able to show that under a certain change of variables (i.e. by looking at the correct type of functions), we can re-write the Euler equation as follows:
\[\partial_t \omega + u\cdot\nabla \omega= \nabla u\omega.\]
\[u=\frac{1}{\alpha}u_\mathcal{S}+u_\mathcal{R},\] for $\alpha$ small. 
In particular, by scaling time by $\alpha$, we see that
\[\partial_t\omega + u_{\mathcal{S}}\cdot\nabla \omega = \nabla u_{\mathcal{S}} \omega + \alpha N(u_{\mathcal{R}},\omega),\]
where $u_{\mathcal{S}}$ has a very simple angular dependence and can be thought of as a function of one variable (the radial variable). Thus, for the problem when $\alpha=0$, there is some hope to actually establish a finite-time singularity. We now want to discuss, on toy examples, what kind of techniques one could use to pass from a blow-up when $\alpha=0$ to a blow-up for $\alpha>0$. 
We will discuss two techniques to approach this problem: the \emph{fixed point method} and \emph{compactness method}.

\subsubsection{Fixed point method}

For an example of the fixed point approach that led to a great deal of work in this direction see  \cite{EJDG}. 
Consider the following problem for $(x,t)\in [0,\infty)\times [0,\infty)$:
\[\partial_t f = f^2 +\epsilon N(f),\] where $N$ could be a functional satisfying the following propeties:
\[N(a f)=a^2 N(f)\qquad N(f(a\,\cdot))=N(f)(a\,\cdot),\] 
for $a\in\mathbb{R}$. Let us also assume that $N$ is a bounded operator, for example, on $H^k$ for some $k\in\mathbb{N}$. 
\\

\noindent\textbf{The case $\epsilon=0$.} When $\epsilon=0$, we have
$\partial_t f = f^2.$
This has many self-similar profiles:
\[f(x,t)=\frac{1}{1-t}F\Big(\frac{x}{1-t}\Big),\]
where the $F$ are solutions to 
\[F+z\partial_z F=F^2.\]
From this, it is easy to see that $F\equiv 1$ is a solution (though this is quite unstable). The stable solution is:
\[F_0(z)=\frac{1}{1+z}.\] 

\noindent\textbf{The case of $\epsilon$ small.} Let us try to find a solution like this when $\epsilon>0$ for
\[\partial_t f = f^2 +\epsilon N(f).\] We seek a solution  of the form:
\[f(x,t)=\frac{1}{1-t}F_\epsilon\Big(\frac{x}{(1-t)^{1+\delta(\epsilon)}}\Big),\] where $\delta(\epsilon)$ is to be determined. 
Then, 
\[F_\epsilon + (1+\delta_\epsilon)z\partial_z F_\epsilon = F_\epsilon^2 +\epsilon N(F_\epsilon).\]
Thus, letting $F_\epsilon=g+F_0$ we see:
\[g+z\partial_z g -\frac{2}{1+z}g=-\delta_\epsilon z\partial_z F_0+g^2+\epsilon N(F_0+g),\]
which we write as:
\[\mathcal{L}(g)=-\delta_\epsilon z\partial_z F_0+g^2+\epsilon N(F_0+g).\]

\begin{myshade}
  \vspace{-1mm}
\begin{lemma}
Let $f\in H^k,$ $k\geq 2$. $\mathcal{L}(g)=f$ is solvable in $C^1$ if and only if $f'(0)+2f(0)=0$. 
Moreover, in this case we can write
\[g(z)=-f(0)+\frac{z}{(z+1)^2}\int_{0}^{z} \frac{(t+1)^2}{t^2}\left(f(t)+f(0)\frac{t-1}{t+1}\right)\rmd t\]
 and
$|g|_{H^k}\leq C_k|f|_{H^k}.$
\end{lemma}
\end{myshade}

\begin{remark}
Assuming the truth of this Lemma, the idea is just to use $\delta_\epsilon$ to make sure that the right hand side satisfies the solvability condition. 
\end{remark}
\begin{proof}
Evaluating $\mathcal{L}(g)=f$ and its derivative at $z=0$ gives 
\[-g(0)=f(0),\qquad 2g(0)=f'(0).\]
This gives the condition 
$f'(0)+2f(0)=0.$
When $f$ satisfies this, we write:
\[\frac{z-1}{z+1}g+z\partial_z g=f,\] so that
\[\frac{z-1}{z+1}(g-g(0))+z\partial_z (g-g(0))=f+f(0)\frac{z-1}{z+1}.\]
Let 
\[G=g-g(0)\qquad F=f+f(0)\frac{z-1}{z+1}.\]
Observe that
$F(0)=F'(0)=0$ using the condition.  Thus we have
\[\frac{z-1}{z+1}G+z\partial_z G=F, \qquad \implies \qquad \frac{(z-1)}{(z+1)z} G+ \partial_z G =\frac{F}{z}.\]
Note that
\[\int \frac{z-1}{(z+1)z}= \int \frac{2z-z-1}{(z+1)z}=2\log(z+1)-\log(z). \]
Thus, 
\[\partial_z \Big(\frac{(z+1)^2}{z}G\Big)=\frac{(z+1)^2}{z^2}F, \qquad \implies \qquad G(z)=\frac{z}{(z+1)^2} \int_{0}^{z}\frac{(t+1)^2}{t^2}F(t)\rmd t\] where we chose that $G'(0)=0$ (since we have one compatibility condition, we have one free condition on $G$ and this is how we chose it). 
We obtain
\[g(z)=-f(0)+\frac{z}{(z+1)^2}\int_{0}^{z} \frac{(t+1)^2}{t^2}F(t)\rmd t.\]

The proof of the $H^k$ bound follows from small generalizations of the Hardy inequality and we leave it as an exercise. There may be some concern about the $f(0)$ terms for $L^2$ bounds so we will explain this part. Let us observe that when $z\geq 1$, 
\[-f(0)+\frac{z}{(1+z)^2}\int_{1}^{z} \frac{(t+1)^2}{t^2}F(t)\rmd t=-f(0)+\frac{z}{(1+z)^2}\int_{1}^z \frac{(t+1)^2}{t^2}f(0)\frac{t-1}{t+1}dt\]\[=f(0)\Big(-1+\frac{z}{(1+z)^2}\int_{1}^z\frac{t^2-1}{t^2}dt\Big)=f(0)\text{O}(\frac{1}{z}).\]
\end{proof}

Now, back to finding $g$:
\[\mathcal{L}(g)=-\delta_\epsilon z\partial_z F_0+g^2+\epsilon N(F_0+g),\]
Observe that the compatibility condition is:
\[2g(0)^2+2\epsilon N(F_0+g)(0)-\delta_\epsilon F_0'(0)+2g(0)g'(0)+\epsilon N(f_0+g)'(0)=0.\]
Since $F_0'(0)=1$ we thus set:
\[\delta_\epsilon=2g(0)^2+2\epsilon N(F_0+g)(0)+2g(0)g'(0)+\epsilon N(f_0+g)'(0)\]
Then we can write:
\[g=\mathcal{L}^{-1}\Big(-\delta_\epsilon z\partial_z F_0+g^2+\epsilon N(F_0+g)\Big):=\mathcal{K}(g).\]

\begin{myshade}
  \vspace{-1mm}
\begin{theorem}
Since $N$ is bounded on $H^k$ for some $k\geq 2$, there is a $\delta>0$ (which is just a constant multiple of $\epsilon$) so that 
\[\mathcal{K}:\overline{B_{\delta}(0)}\rightarrow \overline{B_{\delta}(0)} \] is a contraction (where $B_{\delta}(0)$ is the ball of radius $\delta$ around $0$ in $H^k$).  
\end{theorem}
\end{myshade}

\begin{remark}
We don't actually need that $N$ is bounded on $H^k$. All we need is that $\mathcal{L}^{-1}(N)$ is bounded on $H^k$, which is possible in some cases even when $N$ contains derivatives (see, for example, \cite{EJDG}). We now discuss a more flexible approach given by a compactness method. 
\end{remark}

\subsubsection{Compactness method}

The compactness approach follows essentially the same line of reasoning as the fixed point approach but instead of using invertibility of the linearized operator, we use its ellipticity on certain spaces. The point of this is to deal with possible unboundedness of $N$ in 
\be\label{feqn}
\partial_t f=f^2+\epsilon N(f).
\ee

\noindent\textbf{General philosophy.} Suppose we are trying to solve an equation of the form:
\[L(f)=g+\epsilon N(f).\] When $\mathcal{L}^{-1}N$ is a bounded operator, we can do as above and solve it via a clear iteration scheme. In the case that $L$ is a positive operator and $N$ satisfies:
\[|(N(f),f)|\leq C|f|^2,\] while
\[(L(f),f)\geq c|f|^2,\] we are still able to construct a solution. Indeed, in this case we can add a ``fake" time variable $\tau:$
\[\partial_\tau f + L(f)=g+\epsilon N(f).\]
In the above, the reader should interpret $(,)$ as an $H^k$ inner-product or similar. 
Since $g$ is independent of $\tau$, it is usually then possible to show that $\partial_\tau f$ decays exponentially in, say $H^{k-1}$ while $f$ is bounded in $H^k$. The existence of a solution then follows. 
\\

\noindent\textbf{Simple application of the compactness method.} Consider equation \eqref{feqn}.
We now search for a solution of the form 
\[f(x,t)=\frac{1}{1-(1+\mu(\epsilon))t}F_\epsilon\Big(\frac{x}{(1-(1+\mu(\epsilon))t)^{1+\delta(\epsilon)}}\Big).\] Essentially the point we will see is that while for invertibility of $\mathcal{L}$ we needed only one free parameter, we will need two to establish ellipticity. 
Then we get:
\[(1+\mu(\epsilon))F_\epsilon +z(1+\mu(\epsilon))(1+\delta(\epsilon))\partial_z F_\epsilon=F_\epsilon^2+\epsilon N(F_\epsilon).\]
Writing
$F_\epsilon =F_0+g$, we obtain
\[g +z\partial_z g-\frac{2}{1+z}g=-\mu F_0-(\mu+\lambda+\mu\lambda)z\partial_z F_0+g^2-\mu g-(\mu+\lambda+\mu\lambda)z\partial_z g+\epsilon N(F_0+g).\]
Thus, 
\begin{equation}\label{MainEqn}\mathcal{L}(g)=-\mu F_0-(\mu+\lambda+\mu\lambda)z\partial_z F_0+g^2-\mu g-(\mu+\lambda+\mu\lambda)z\partial_z g+\epsilon N(F_0+g).\end{equation}
Our goal is now to find a space $X$ so that
\begin{equation}\label{LB}(\mathcal{L}(g),g)_{X}\geq c|g|_{X}^2\end{equation} while
\begin{equation}\label{UB}(RHS,g)_{X}\leq C\epsilon (|g|_{X}+|g|_{X}^2)+C|g|_{X}^3.\end{equation}
Recall that 
\[\mathcal{L}(g)=g+z\partial_z g-\frac{2}{1+z}g.\]
Consider a $g$ supported in $[0,1]$ which satisfies $g(0)>0$ and then decreases monotonically to $0$. Obviously, for such $g$, $(\mathcal{L}(g),g)_{L^2}<0$, while for $g$ supported at large $z$ we see that $(\mathcal{L}(g),g)>0$. This means that it will not be able to take the space $X$ to be a normal Sobolev space. 

If we wish to show that $\mathcal{L}$ is positive, we should try to penalize mass at $0$. It turns out that if one just defines the weight
\[w(z)=\frac{(1+z)^4}{z^4},\] then
\[(\mathcal{L}(f),f)_{L^2_w}=\frac{1}{2}|f|_{L^2_w}^2.\]
Note that it isn't really necessary to have have an exact weight which will give us equality as above, but the point is to take a strong weight near $z=0$. There are now a number of ways to get around this issue and choosing the weight for a given problem can be quite involved. 
\\

\noindent \textbf{Choice of $\mu$ and $\lambda$.} Since we intend to solve for $g$ in the space $L^2_w$ with weight $w=\frac{(1+z)^4}{z^4}$, we need to ensure that $g$ and the right hand side vanish quadratically at $z=0$. 
Observe that, as in the fixed point approach, we can use $\mu$ and $\lambda$ to satisfy these two conditions since we have on the RHS or \eqref{MainEqn} the term
\[-\mu F_0-(\mu+\lambda+\lambda\mu)z\partial_z F,\] which when evaluated at $z=0$ gives $\mu$ and when differentiated and evaluated at $0$ gives $2\mu+\lambda+\mu\lambda$. This means that we can choose $\mu$ to ensure $RHS(0)=0$ and then choose $\lambda$ to ensure $RHS'(0)=0$.

\begin{myshade}
  \vspace{-1mm}
\begin{lemma}
It is possible to find constants $c_1,c_2>0$ so that the inner product
\[(f,g)_{X}=(f,g)_{L^2_w}+c_1(z\partial_z f, z\partial_z g)_{L^2_w}+c_2((z\partial_z)^2 f, (z\partial_z)^2g)_{L^2_w}\] is such that \eqref{LB} holds. 
\end{lemma}
\end{myshade}

\begin{remark}
Additionally, \eqref{UB} holds with the above choices of $\mu$ and $\lambda$ so long as $N$ satisfies reasonable assumptions (basically that we can do standard energy estimates). 
\end{remark}

\begin{myshade}
  \vspace{-1mm}
\begin{corollary} (A-priori estimate)
There exists a $C_*>0$ and an $\epsilon>0$ sufficiently small so that if a solution to \eqref{MainEqn} $g\in X$ exists with $|g|_{X}\leq C_*\epsilon$, then actually $|g|_{X}\leq C_*\frac{\epsilon}{2}$. 
\end{corollary}
\end{myshade}

 \subsection{Application to axi-symmetric no swirl solutions}
We now discuss a particular situation where the perturbation arguments discussed above as well as the expansion of the non-local operator of Section \ref{NonlocalEffect} can be applied to give a singularity for $C^{1,\alpha}$ solutions of 3d Euler as in \cite{E_Classical}. We emphasize that this is just one application of these ideas. Another application appeared in \cite{CHB}, where the authors establish a finite-time singularity for $C^{1,\alpha}$ solutions to the 3D Euler equation in a cylinder in the precise scenario of the numerical work \cite{HouLuo}. Recall the axi-symmetric Euler equation derived above \eqref{ASE}-\eqref{swirl}: 
\[\partial_t v+v\cdot\nabla v+ \nabla p =(\frac{(u_\theta)^2}{r},0),\]
\[\nabla\cdot(r v)=0,\]
\[\label{swirl}\partial_t(ru_\theta)+v\cdot\nabla (ru_\theta) = 0.\] Recall that $v=(u_r, u_z)$ and all derivatives are $(r,z)$ derivatives. 
We restrict further our class of solutions to the ones where \[u_\theta\equiv 0.\] These are called no-swirl solutions. Geometrically, particles are allowed to move toward/away and from and up/down the $z$-axis but cannot ``swirl" around the axis. Such solutions are effectively 2d though they are three dimensional exactly at the axis of symmetry. While this restriction appears quite restrictive, it turns out that solutions can still exhibit non-trivial behavior and even blow-up. 

Our first task will be to pass to the vorticity equation. We define 
\[\omega=\partial_r u_z-\partial_z u_r.\]
We obtain the system
\begin{equation}\label{ASVorticity}\partial_t \omega + v\cdot\nabla_{r,z}\omega =\frac{u_r}{r} \omega, \end{equation}
\begin{equation}\label{ASBSLaw}\div(r v)=0,\,\, \omega=\partial_r u_z-\partial_z u_r.\end{equation}

\subsubsection{Boundary Conditions and Growth Mechanism}
The physical domain of the solution is:
\[(r,z)\in [0,\infty)\times \mathbb{R}.\]
When $r=0$, the system does not actually require any boundary condition on $\omega$. However, in order than the original $u$ be smooth, $\omega_0$ must vanish at least linearly on $r=0$. Of course, if we only care that $u\in C^{1,\alpha}$ all it means is that $\omega$ must vanish like $r^\alpha$ as $r\rightarrow 0$. The system does impose a boundary condition on $u_r$. We need $u_r$ to vanish at $r=0$ linearly (just for the solution to be $C^1$). 

In this particular setting, it is easy to see that in order for $\omega$ to grow, we need to be working in a scenario where $\frac{u_r}{r}>0$ (at least near the maximum of $\omega$). It is easy to see that if we put an odd symmetry in $z$ and also make $\omega$ positive in $[0,\infty)\times [0,\infty)$, $u_r$ will be positive on $z=0$ and $u_z$ will be negative on $r=0$.  This follows from the Hopf Lemma. We now see that $0$ is a hyperbolic point and that $\omega$ grows along trajectories. Whether $\omega$ actually grows pointwise depends on the degree of vanishing of $\omega$ along $r=0$, just as we saw in our analysis of the vorticity equation in 1d in Section \ref{AdvectionAndVS}.

\subsubsection{Passage to the fundamental model}
Toward understanding the behavior of weakly vanishing solutions near the origin, we introduce the new variables
\[R=\rho^\alpha, \theta=\arctan(\frac{z}{r})\] where $\rho=\sqrt{r^2+z^2}$. Following a similar computation to what we did in Section \ref{NonlocalEffect} (see \cite{E_Classical} for details), we get that \eqref{ASVorticity}-\eqref{ASBSLaw} can be re-written as
\[\partial_t \Omega - \frac{3}{2}L_{12}(\Omega)\sin(2\theta)\partial_\theta \Omega=L_{12}(\Omega)\Omega +\alpha N(\Omega),\]
\[L_{12}(\Omega)(R)=\int_{R}^\infty\int_{0}^{\pi/2}\frac{\Omega(s,\theta)}{s}K(\theta)d\theta ds,\]
where $K(\theta)=3\cos^2(\theta)\sin(\theta)$ and $N$ is a quadratic non-linearity that is amenable to energy estimates independent of $\alpha$. 

{\bf Let us now set $\alpha=0$.} 
Then we see that
\begin{equation}\label{FM1}\partial_t \Omega - \frac{3}{2}L_{12}(\Omega)\sin(2\theta)\partial_\theta \Omega=L_{12}(\Omega)\Omega. \end{equation}
\begin{equation}\label{L12} L_{12}(\Omega)(R)=\int_{R}^\infty\int_{0}^{\pi/2}\frac{\Omega(s,\theta)}{s}K(\theta)d\theta ds. \end{equation}
This is the system \[\partial_t\omega + u_{\mathcal{S}}\cdot\nabla \omega = \nabla u_{\mathcal{S}}\omega,\] which we mentioned before, restricted to axi-symmetric solutions without swirl. 

It is not difficult to check that \eqref{FM1}-\eqref{L12} is, in fact, globally regular for data which is smooth in $\theta$ and $R$ and vanishing on $\theta=0$.

\begin{myshade}
  \vspace{-1mm}
\begin{lemma}
Assume that $\Omega_0\in C^{1/3}_c([0,\frac{\pi}{2}]\times [0,\infty))$ and $\Omega_0(R,0)=0$, then the unique local solution of \eqref{FM1}-\eqref{L12} is global. 
If, however, $\Omega_0(\theta,R)\geq (\pi/2-\theta)^\alpha$ for $\theta$ near $\pi/2$ and for any $\alpha<\frac{1}{3},$ the solution can blow-up in finite time. 
\end{lemma}
\end{myshade}

\begin{remark}\label{AngularRegularity}
This Lemma implies that even when we formally set $\alpha=0$ smooth (in $\theta$ and $R$) solutions are global. These solutions are merely bounded solutions in the original coordinates; however, because of the regularity in $\theta$ they have the key property that $\omega$ vanishes linearly on the axis of symmetry. In this regime, the transport term is strong enough to prevent blow-up just as we saw in the 1d models. 

In particular, this means that the small $\alpha$ regime in the change of variables $r\rightarrow r^\alpha$ itself does not actually change the regularity properties of the Euler equation. The only use of this change of variables is to magnify certain effects that, in principle, could be magnified in the smooth case as well. To get a singularity, we need to minimize the effect of the transport term. This leads us to introduce non-smoothness in the direction of the flow.  
\end{remark}

\subsubsection{Dropping the Angular transport term for solutions that are almost $\theta$ independent}
When $\Omega$ is independent of $\theta$, \eqref{FM1}-\eqref{L12} becomes:
\[\partial_t\Omega(R,t)=\Omega(R,t)\int_{R}^\infty \frac{\Omega(s,t)}{s}\rmd s.\]
This is like a non-local Ricatti equation. It turns out that very similar arguments to what we described in Section 3 above (the compactness approach), yield that the self-similar blow-up profile:
\[\Omega(R,t)=\frac{1}{1-t} F(\frac{R}{1-t}), \qquad F(z)=\frac{z}{(1+z)^2},\qquad L_{12}(F)=\frac{1}{(1+z)}\] is stable.

\subsection{Analysis of the linearized operator $\mathcal{L}$ and design of angular weights}
When interpreted as a solution of \eqref{FM1}-\eqref{L12}, the linearized operator becomes:
\[\mathcal{L}g=g+z\partial_z g - \frac{2}{1+z}g-\frac{z}{(1+z)^2}L_{12}(g)+\frac{3}{2(1+z)}\sin(2\theta)\partial_\theta g.\]
The way to deal with $\mathcal{L}$, we have to be very careful in how we deal with the angular transport term. Our goal is to design a reasonable inner product space $X$ for which we have
\[(\mathcal{L}g,g)_{X}\geq c|g|_{X}^2.\]
Let us adopt the notation:
\[D_\theta=\sin(2\theta)\partial_\theta, \qquad D_z=z\partial_z.\]
Now, observe that
\[D_\theta\mathcal{L}g= D_\theta g +z\partial_z D_\theta g-\frac{2}{1+z}D_\theta g+\frac{3}{2(1+z)}\sin(2\theta)\partial_\theta (D_\theta g).\]
We already have a weight $w_{z}(z)={(1+z)^4}/{z^4}$ which gives us positivity from the first three terms. For the last term, we can simply add the weight
\[w_{\theta}=\frac{1}{\sin(2\theta)^{1+\delta}},\] and if $\delta$ is sufficiently small, we will see that
\[(D_\theta \mathcal{L} g, D_\theta g)_{L^2_w}\geq (\frac{1}{2}-\delta)|D_\theta g|_{L^2_w}^2,\] where 
$w:=w_z w_\theta.$ This allows us to get control of $D_\theta g$ in terms of itself. We then define an inner product of order $1$ as follows:
\[(f, g)_{\mathcal{H}^1}=(D_\theta f, D_\theta g)_{L^2_w}+c_1(f, g)_{L^2_w}+c_2(D_zf,D_z g)_{L^2_w}\] and $c_1$ can be taken small and then $c_2$ smaller so that we have:
\[(\mathcal{L}g,g)\geq c|g|_{\mathcal{H}^1}^2.\] The choice of taking $\delta>0$ is to ensure that $\mathcal{H}^k$ embeds in $C^\alpha$ for some $\alpha>0$. 

\section{Perspectives and problems on singularity formation}\label{perspect}

We now move on to give precise statements of the existing blow-up results on the 3d Euler equation and discuss various possible extensions. Most of the discussion below will be regarding axi-symmetric solutions. The reason for this is that fully 3d solutions are significantly more complicated and there are very few existing results. An interesting direction of study in the full 3d setting is the high-symmetry setting introduced in \cite{Kida,BP} and studied analytically in \cite{EJEO}. An interesting feature of \emph{all} of the current results (numerical and analytical) on singularity formation in the Euler equation is that they seem to occur near stationary points of the flow where the vorticity is also vanishing. This motivates the following interesting question of D. Sullivan. 

\begin{myshade3}
  \vspace{-1mm}
\begin{question}[D. Sullivan]\label{Sullivan}
Do there exist blow-up solutions with non-vanishing vorticity?
\end{question}
\end{myshade3}

\begin{remark}
Even in the 1d models, the vanishing of vorticity seems to be a necessary condition for blow-up. It is not clear at all how one could approach this question with the current technology; however, it would be interesting to develop arguments in this direction applied to the 1d models.
\end{remark}

\subsection{The Boussinesq system}

Let us now recal the axi-symmetric Euler equation and Boussinesq system. We already derived the axi-symmetric Euler equation above in \eqref{ASE}-\eqref{swirl}: 
\[\partial_t v+v\cdot\nabla v+ \nabla p =(\frac{(u_\theta)^2}{r},0),\]
\[\nabla\cdot(r v)=0,\]
\[\label{swirl}\partial_t(ru_\theta)+v\cdot\nabla (ru_\theta) = 0\] where that $v=(u_r, u_z)$ and all derivatives are $(r,z)$ derivatives. Here, the two-dimensional velocity field $v$ essentially solves the 2d Euler equation except that it is forced on the right-hand side by an effect of the swirl $u^\theta$. A multiple of the swirl $u^\theta$ is then transported. There are two quite different regimes for this equation. Near $r=0$ even the zero-swirl solutions can behave in a complicated way. Away from $r=0$, axi-symmetric Euler solutions can be effectively modeled by solutions to the Boussinesq system:
\begin{equation}\label{B1}\partial_t v + v\cdot\nabla v +\nabla p=(\rho,0),\end{equation}
\begin{equation}\label{B2}\nabla\cdot v=0,\end{equation}
\begin{equation}\label{B3}\partial_t \rho+v\cdot\nabla \rho=0.\end{equation}
Definitive links between the singularity problem for the Boussinesq and axi-symmetric Euler equations were established in \cite{EJB,EJE,CHB}. While it may not always be the case, it is reasonable to guess that a blow-up in the Boussinesq system essentially implies a corresponding blow-up of axi-symmetric Euler solutions.

\subsubsection{Corner domains}

The first results on blow-up for 3d Euler with finite energy were established in domains with corners following a series of papers by I. Jeong and the second author \cite{EJB,EJE,EJEO}. The key idea of these works is to introduce and analyze scenarios in which the Biot-Savart law can be \emph{localized.} The key is to impose geometric conditions, such as spatial symmetries or boundary conditions, that ensure that we can recover the entire velocity gradient at a point $x_*=0$ from just the vorticity vector at this point. In the setting of solutions on the whole space where we impose only symmetry conditions, we have already discussed this in Section \ref{SISolutions}. In the setting of domains with boundary, this phenomenon is understood from an analysis of the Taylor expansion of solutions to the Poisson equation. We explain this in the following example. 

\begin{example}
Assume that $\psi\in C^2$ is any solution to a Dirichlet problem $\Delta\psi =f$ on the sector $\{(x,y): 0\leq x\leq y\}$. Then, the entire matrix $D^2\psi(0)$ can be determined uniquely from $f(0)$. The same can be said in any sector or corner domain of angle strictly less than $\pi/2$. The reason for this is that on this domain we have that $\psi(x,0)=0$ and $\psi(x,x)=0$ so that $\partial_{xx}\psi(0,0)=0$ and $2\partial_{xy} \psi (0)+\partial_{yy}\psi(0)=0$. Since $f(0)=\partial_{xx}\psi(0)+\partial_{yy}\psi(0)$, we now recover $D^2\psi(0)$. 
\end{example}

When considering solutions on domains without a corner or with a corner of size $\pi/2$ or more, further symmetries should be imposed on the data. One idea considered in \cite{EJSymm,EJB} is to impose a reflection symmetry. This is the setting of the following problem.

\begin{myshade3}
  \vspace{-1mm}
\begin{problem}\label{bousprob}
Prove that there exist scale-invariant blow-up solutions to the Boussinesq system \eqref{B1}-\eqref{B3} on the corner domain $\Omega=\{(r,\theta): -\pi/4<\theta<\pi/4\}$ with vanishing vorticity and density gradient on $\partial\Omega$. 
\end{problem}
\end{myshade3}

\begin{remark}
The interested reader should look at \cite{EJB} for previous results and ideas in this direction. Additionally, a scalar model of this system was studied in \cite{EIS} where this type of result was established. 
\end{remark}

\noindent The above problem is a prelude to the more striking possibility that solutions supported entirely inside of such corner domains could become singular in finite time.

\subsection{Axi-symmetric solutions on the cylinder: the Hou-Luo scenario}

We now move to discuss a very promising scenario for blow-up of smooth solutions at the boundary of a smooth domain proposed by Guo Luo and Tom Hou \cite{HL}.  The singularity mechanism is most clearly explained on the Boussinesq system \eqref{B1}-\eqref{B3} on the domain $[0,\infty)\times \mathbb{R}$.  Observe that the two-dimensional vorticity $\omega$ is forced by the vertical derivative of $\rho$:
\[\partial_t\omega + u\cdot\nabla\omega = \partial_y \rho,\] while $\rho$ is transported by $u$:
\[\partial_t \rho+u\cdot\nabla \rho=0.\]
Now let us assume that the initial data for $\rho$ even in $y$ and decreasing as a function of $y$ on $[0,\infty)$ for fixed $x$. Under this condition, positive vorticity is produced for short time in the quadrant $[0,\infty)\times [0,\infty)$ and the vorticity remains odd in $y$ for all time. For short time, the generated vorticity produces a velocity field that pushes particles down in $y$ and "out" in $x$ on $[0,\infty)\times [0,\infty)$. This, in turn leads to the growth of the magnitude of $\partial_y\rho$. The growth in $\partial_y \rho$ leads to more vorticity growth and stronger advection. If this growth "loop" is sustained for all time, a finite time singularity will form. Remarkably, there does not appear to be any clear "enemy" to blow-up in this situation and the only difficulty is to capture the blow-up rigorously. The difficulty is compounded by the fact that \eqref{B1}-\eqref{B3} is a system and is non-local. Despite the difficulties associated to rigorously establishing the blow-up, upon studying the simulations provided by Luo and Hou \cite{HL} and reflecting upon the mathematical mechanism described above, one is naturally led to believe that a singularity does indeed occur in this scenario:

\begin{myshade2}
  \vspace{-1mm}
\begin{conjecture}\label{HLConj}
There exists a smooth solution to the axi-symmetric Euler equation on the cylindrical domain $\{(r,z): 0\leq r\leq 1\}$ that becomes singular in finite time. 
\end{conjecture}
\end{myshade2}

\begin{remark}
There may be different approaches to solve this conjecture. If it be through constructing first a self-similar profile, it would suffice to establish the corresponding result on the Boussinesq system. In the category of $C^{1,\alpha}$ solutions, a singularity was constructed in the work \cite{CHB} via this strategy.  It is important to emphasize that singularity formation at the boundary is qualitatively different from singularity formation in free space. This is due to the suppression, via the boundary, of a powerful regularizing mechanism in the vorticity equation. We know that this regularizing mechanism is present in 1d equations and \emph{prevents} singularity formation for smooth data. This regularizing mechanism is completely absent when the vorticity does not vanish on the boundary. 
\end{remark}

The study of conjecture \ref{HLConj} has led to a number of outstanding results in the analysis of the Euler equation and related models. The first such result was given by Kiselev and \v{S}ver\'ak \cite{KS14} where the authors established the qualitatively optimal double-exponential growth rate for the gradient of vorticity for 2d Euler solutions on the disk. The result was established using two essential ingredients: the Key Lemma, which is contained in Lemma \ref{PtWiseKeyLemma}, which gives the leading order term in the expansion of the velocity at a hyperbolic point in the case of uniformly bounded vorticity. This is used to give large scale control on the motion of particles. The second ingredient is a clever argument to control the evolution of a bump of vorticity on the boundary which also uses a type of comparison principle given by the Key Lemma that particles closer to the hyperbolic point tend to have higher relative velocity. An important extension of the methods introduced in \cite{KS14} was given in \cite{KRYZ} where a singularity was established in a related two dimensional model. Finally, following the work of Luo and Hou \cite{HL}, there were a number of model equations that were proposed to get better understand the singularity in the Hou-Luo scenario. Singularity formation in these models is established mostly using contradiction and/or barrier arguments (see, for example, \cite{CKY, Choi2017, TDo, DKX, HK, KRYZ, KT, Z15}). Recently, Chen, Hou, and Huang \cite{CHHHL} established the existence of a self-similar blow-up in a one-dimensional ``boundary layer model" of the Hou-Luo scenario via computer assistance. The authors of \cite{CHHHL} also suggested the existence of approximate self-similar profiles in the in the original Hou-Luo scenario. More recently, Wang, Lai, G\'omez-Serrano, and Buckmaster \cite{WLGB} proposed a new method to find self-similar profiles in the Hou-Luo scenario using neural networks.

\subsection{Axi-symmetric solutions without swirl}

The focus of this subsection is on possible extensions and follow-up problems related to the blow-up result \cite{E_Classical}. We remark that all of the problems we discuss below could also be asked in the context of the singularity result of Chen and Hou \cite{CHB} (though there should be no regularity limitations there).
Let us first recall the main Theorem of \cite{E_Classical}:

\begin{myshade}
  \vspace{-1mm}
\begin{theorem}\label{EClassical}
There is a continuum $(0,\alpha_0)$ of self-similar blow-up axi-symmetric no-swirl Euler solutions. 
\end{theorem} 
\end{myshade}

The solutions have the following character: they can be charactured as concentrating and colliding oppositely signed vortex rings. See Figure \ref{fig1tarek}

\begin{figure}[h!]\centering
    \includegraphics[width=.32\columnwidth]{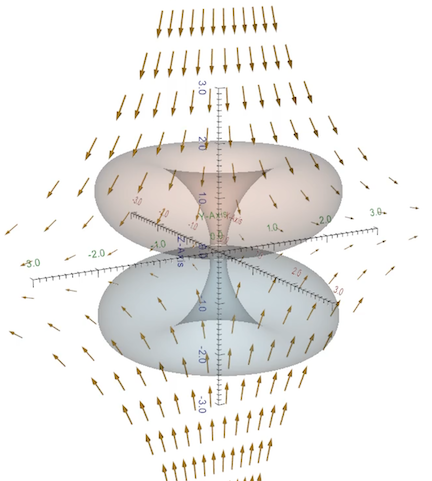} 
    \includegraphics[width=.32\columnwidth]{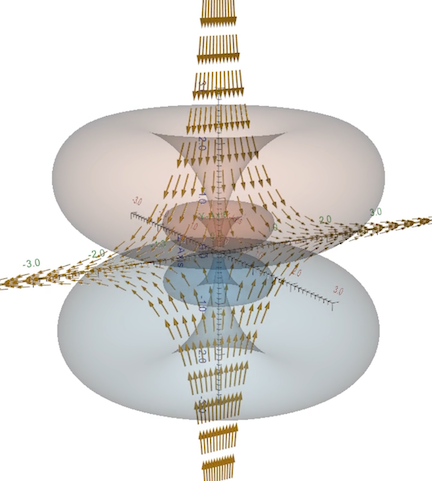} 
    \includegraphics[width=.32\columnwidth]{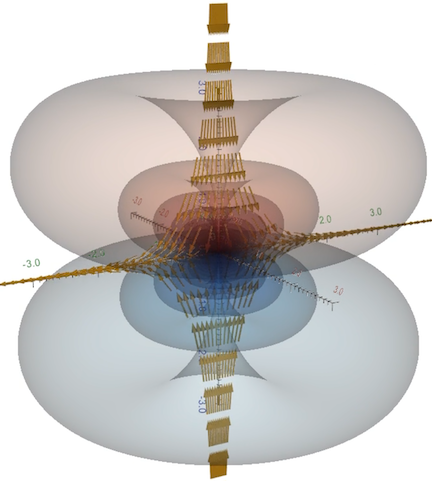} 
  \caption{\small Snapshots of blowup solution of Theorem \ref{EClassical} at times approaching singularity. }
\label{fig1tarek}
\end{figure}

\noindent

A first obvious problem is to give a reasonable estimate on the size of $\alpha_0$. This can be done by checking the constants in \cite{E_Classical}, though many of the estimates done there were suboptimal. It is possible that by sharpening the arguments in \cite{E_Classical}, one can show that $\alpha_0>10^{-4}$. To motivate the next problem, let us observe that the blow-up constructed in \cite{E_Classical} is essentially one-dimensional. Indeed, one of the key ideas is to show that in some regime the Euler equation is well approximated by solutions to the so-called ``fundamental model:"
\[\partial_t F = F L_{12}(F),\qquad L_{12}(F)= \int_{r}^\infty\int_{0}^{2\pi}\frac{F(s,\theta)K(\theta)}{s}d\theta ds.\]  
This model problem has an infinite-dimensional family of self-similar blow-up profiles starting from data of the form:
\[F_0(R,\theta)=\frac{R}{(1+R)^2} \Gamma(\theta),\] for \emph{any} $\Gamma$ satisfying $\int_{0}^{2\pi} K(\theta)\Gamma(\theta)d\theta>0$. Despite this, in \cite{E_Classical} the author was only able to establish singularity formation in the Euler equation near \emph{one} particular profile $\Gamma_*$. Whether there is actually a singularity near other angular profiles is a very interesting problem.

\begin{myshade3}
  \vspace{-1mm}
\begin{problem}\label{axiblowupfund}
Establish blow-up for axisymmetric no-swirl solutions in the vicinity of a large class of blow-up profiles to the fundamental model. 
\end{problem}
\end{myshade3}

\begin{remark}
It is possible that most of the tools to solve this problem are essentially contained in \cite{E_Classical}, though it is possible that singularities near profiles other than $\Gamma_*$ will exhibit non-self-similar behavior. Studying this problem would also shed some light on the universality profiles. Indeed, in a number of other similar models it is observed that corresponding to each blow-up speed there is a unique self-similar profile. 
\end{remark}

In the spirit of the preceding problem, we also mention the following general problem.

\begin{myshade3}
  \vspace{-1mm}
\begin{problem}\label{instabquest}
Study the stability/instability of the self-similar profiles of Theorem \ref{EClassical} with respect to general 3d perturbations. 
\end{problem}
\end{myshade3}

\begin{remark}
In the work \cite{EGM3dE}, stability of those profiles was established with respect to certain types of perturbations including perturbations with swirl. 
\end{remark}

Another problem that may appear to be technical at first is the following.

\begin{myshade3}
  \vspace{-1mm}
\begin{problem}\label{selfsimcoer}
Establish the existence of self-similar profiles relying only on the invertibility of the linear operator rather than its coercivity. 
\end{problem}
\end{myshade3}

\begin{remark}
Another way to say this is to establish existence using the ``fixed point method" rather than the ``compactness method" as discussed in Section \ref{selfsimilar}. It would be interesting to see whether this can be done even in the setting of \cite{E_Classical}. We remark that if it were possible to establish existence of a profile based only on invertibility, it would allow for a much more robust approach to the construction of singularities, both analytically and numerically. In particular, it would give a basis for studying the continuation of curves of self-similar profiles as we are about to discuss. 
\end{remark}

Another important direction of study is to determine the behavior of self-similar solutions as we increase the continuation parameter $\alpha$ used in the paper \cite{E_Classical}. Such a study would shed further light on the nature of singularity formation in this setting. It may be possible to establish analytical results in the spirit of \cite{HMW}, though even a better understanding through numerical simulations of the solution curves would be illuminating.

\begin{myshade3}
  \vspace{-1mm}
\begin{problem}\label{curvequest}
Study the behavior of the curve of self-similar solutions in $\alpha$ constructed in \cite{E_Classical}. Does the curve continue for all $\alpha\leq 2$ or does it stop at some $\alpha_*<2$? Establish continuation criteria for the curve of solutions.
\end{problem}
\end{myshade3}

\begin{remark}
Note that the curve does not necessarily stop until $\alpha=2$ (after which the change of variables involving the parameter $\alpha$ ceases to make sense). The parameter $\alpha$ only controls the radial regularity, which could be $C^\infty$ as we will remark below. 
\end{remark}

\begin{myshade2}
  \vspace{-1mm}
\begin{conjecture}\label{alpha13}
For any $\alpha<\frac{1}{3}$, there exists a $C^\alpha_c$ axi-symmetric no-swirl local Euler solution that becomes singular in finite time. 
\end{conjecture}
\end{myshade2}

\begin{remark}
Note that $C^{1/3+}_c$ solutions are globally regular and so solving this conjecture would close the regularity gap (see \cite{UY, Serfati3D, ShirotaY, SaintRaymond, Danchin, AHK}). It is likely that there is global regularity in the borderline case as well. 
\end{remark}

A sharper version of the preceding conjecture is the following one where the real non-smoothness required is in the flow direction as we remarked above in Remark \ref{AngularRegularity}.

\begin{myshade2}
  \vspace{-1mm}
\begin{conjecture}[Smooth in $R$, limited regularity in $\theta$] \label{radsmooth}
In the preceding conjecture, $\omega_0$ can be $C^\infty$ in $\rho=\sqrt{x_1^2+x_2^2+x_3^2}$. 
\end{conjecture}
\end{myshade2}

Finally, we remark that the key to the blow-up seems to be the degree of vanishing of the vorticity on the axis of symmetry and not non-smoothness \emph{per se}. The following conjecture makes this precise.

\begin{myshade2}
  \vspace{-1mm}
\begin{conjecture}[Artificial Boundary at the axis of symmetry]\label{artbound}
There exist smooth solutions to the axi-symmetric Euler equation non-vanishing on the axis of symmetry that become singular in finite time. 
\end{conjecture}
\end{myshade2}

\begin{remark}
Of course, "smooth" in the preceding conjecture means smooth in the axi-symmetric coordinates; however, once they are lifted to 3d solutions they are not smooth on the axis since the vorticity does not vanish on the symmetry axis. 
\end{remark}

\subsubsection{Continuation After Singularity?}

An important problem is to study extensions of classical solutions beyond the time of singularity. Since solutions leave the classical uniqueness classes at the singularity time, it seems that "anything" could occur after the singularity. Despite this, we have the following

\begin{myshade}
  \vspace{-1mm}
\begin{theorem}\label{ContinuationTheorem}
Let $\omega_0$ be a compactly supported initial axi-symmetric no-swirl vorticity that gives a local classical solution that becomes singular in finite time. Also assume that $\frac{1}{r}\omega_0\in L^{2+\delta}$ for some $\delta>0$.  
Consider a family $\omega^\epsilon$ of axi-symmetric no-swirl solutions from regularized data $\omega_0^\epsilon$. Now let $\omega_*$ be any subsequential limit of $\omega^\epsilon$ as $\epsilon\rightarrow 0$. Then, $\omega_*$ is a weak solution to the 3d Euler equation that is necessarily smooth outside of $B_{C\cdot(t-T_*)}(0)$ for some constant $C=C(\omega_0)>0$, where $T_*=T_*(\omega_0)$ is the blow-up time of the Euler solution emanating from $\omega_0$.   
\end{theorem}
\end{myshade}

The theorem says that the singular set after the blow-up time remains inside a growing ball around the origin.  The structure of the flow-lines seems to indicate that the singular set is actually much smaller. This is the content of the following conjecture.

\begin{myshade2}
  \vspace{-1mm}
\begin{conjecture}\label{propring}
The singular set of the solutions of \cite{E_Classical} propagates within the natural symmetry class as an expanding ring $\mathcal{S}(t)=\{(r,x_3)=(\delta(t),0)\}$, where $\delta(T_*)=0$ while $\delta$ is increasing in $t$ for all $t\geq T_*$. 
\end{conjecture}
\end{myshade2}

\begin{remark}
This may be related to corresponding results and conjectures in two dimensions \cite{Elling, BM}. 
\end{remark}

\begin{proof}[Sketch of the Proof of Theorem \ref{ContinuationTheorem}]

By the assumptions and properties of axi-symmetric no-swirl solutions, we have a time and $\epsilon$ independent bound on \[\left\|\omega^\epsilon/r\right\|_{L^{2+\delta}}.\] Standard elliptic estimates give a uniform bound on $u$:
\[\|u^\epsilon\|_{L^\infty}\leq C(\omega_0)\] for all $\epsilon,t> 0$. At the same time, we have that  ${\omega}/{r}$ is transported. This implies 
\[\|{\omega^\epsilon}/{r}\|_{L^\infty(A)}(t_1)\leq \|{\omega^\epsilon}/{r}\|_{L^\infty(B_{\delta(t)}(A))}, \] whenever $t_2\geq t_1$, for any open set $A$ where $B_{\delta(t_1,t_2)}(A)$ is the set of points within distance $\delta(t)$ of $A$ and $\delta(t_1,t_2)=C\cdot (t_2-t_1)$. This implies local $L^\infty$ bounds on $\omega^\epsilon$ independent of $\epsilon$ away from the axis of symmetry. The global $L^2$ bound on $\omega$ along with the local $L^\infty$ bounds away from the axis of symmetry gives local propagation of regularity away from the axis of symmetry. This is because the Biot-Savart law is smoothing off the diagonal. 
That $\omega_*$ is a weak solution follows from a standard compactness argument using the a-priori bound. 
\end{proof}

\begin{remark}
This result has a 2d analogue \cite{DEL22}: if we consider data  $\omega_0\in L^p(\mathbb{R}^2)\cap C^\infty(\mathbb{R}^2\setminus\{0\})$ for $p>2$, then any limit of regularized solutions will be smooth outside of a growing ball. In particular, possible non-uniqueness can only come from inside that ball. 
\end{remark}

\section{$\infty$D fluids: finite time singularity from smooth data}\label{infdsec}

Here we consider \emph{pressureless solutions}, i.e. those with $\nabla p=0$ for all time on the domain $M= \mathbb{T}^d$ or $\mathbb{R}^d$.  These solve the multidimensional Burgers' equation
  \begin{alignat}{2}\label{eeb1}
\partial_t u + u \cdot \nabla u &= 0,  \\
u|_{t=0} &=u_0 . \label{eef1}
\end{alignat}
To be consistent with incompressibility, the initial data $u_0$ must  be very special as we shall see.  From the (infinite dimensional) geometric viewpoint of \S \ref{lap}, such Euler solutions correspond to flowmaps $\Phi_t$ which are \emph{asymptotic geodesics} \cite{KM}.  Namely, those curves which are geodesic both in the group $\mathscr{D}_\mu(M)$ of smooth volume-preserving  diffeomorphisms of $M$ with the $L^2$ pseudo-Riemannian metric  \eqref{action} as well as in the full diffeomorphism group  $\mathscr{D}(M)$ for either the non-invariant or right invariant $L^2$ metrics (the distinction of non-invariant and right invariant is immaterial for $\mathscr{D}_\mu(M)$ since the Jacobian determinant is unity).  We remark also that, in view of Theorem \ref{theoremgeo}, they are \emph{minimizers} of the energy for all time.

 The first statement characterizes for which initial data the Burgers solution can maintain incompressibility.
This result can be found in  Chapter III the thesis of  Serfati  \cite{s92} and the paper of Yudovich \cite{Y00}.  That pressureless solutions are globally regular is a consequence of the continuation criterion given by Chae and Constantin \cite{CC21}.

\begin{myshade}
  \vspace{-1mm}
\begin{theorem}\label{pressthm}
A velocity field $u: M\times \mathbb{R}\to \mathbb{R}^d$ is an incompressible solution of \eqref{eeb1} with $u|_{t=0}=u_0: M\to \mathbb{R}^d$ for all time $t\in \mathbb{R}$ if and only if $\nabla u_0:M\to \mathcal{M}^{d\times d}(\mathbb{R})$ is a nilpotent matrix for each $x\in M$.
If $\nabla u_0$ is a nilpotent matrix, then the solution exists for all time and the gradient satisfies:
\be\label{sum}
\nabla u (\Phi_t(x),t) =\sum_{m=0}^{d-2} (-t)^m (\nabla u_0(x))^{m+1}
\ee
where $\Phi_t$ is the Lagrangian flowmap.
\end{theorem}
\end{myshade}

Theorem \ref{pressthm} gives examples of solutions on $\mathbb{T}^d$ and $\mathbb{R}^d$ for any $d>2$ which exist for all time and grow with a specific polynomial rate $\| u(t)\|_{C^1} \sim t^{d-2}$.   This point is related to the fact that these flows are unstable in Sobolev spaces in their (Lagrangian) configuration space \cite{mis1}.

\begin{proof}
Let $\Phi_t$ be the flowmap corresponding to the solution of \eqref{eeb1}--\eqref{eef1}.
Eqn. \ref{eeb1} says that $u$ is transported by its own flow, thus trajectories are straight lines
\be
\Phi_t(x) = x + t u_0(x).
\ee
The velocity field $u$ is incompressible if and only if the flow is volume preserving, i.e. $\det(\nabla \Phi_t(x))=1$.  
Thus, the incompressibility condition is equivalent to 
\be \label{nilpcond}
\det\left( I + t \nabla u_0(x) \right) = 1 \quad \text{for all} \quad x\in M,\  t\in \mathbb{R}.
\ee
Recall that if a matrix $N$ is nilpotent, then $\det(I+ t N)=1$ for any $t\in \mathbb{R}$.  Conversely, if $\det(I+ tN)=1$ for $d+1$ distinct values of $t$, then $N$ is nilpotent. It follows that the condition of \eqref{nilpcond} holding for all $t\in \mathbb{R}$ is equivalent to the matrix $\nabla u_0(x)$ being nilpotent for every $x\in M$.

The equation for the gradient of the Burgers solution is
\be
(\partial_t  + u \cdot \nabla)\nabla u = - (\nabla u )^2.
\ee
Thus $A(t) =\nabla u (\Phi_t(x),t) $ satisfies the following matrix ODE\footnote{See Sullivan \cite{S11ode} for a discussion of conditions for long time existence and genericity of finite time blowup in quadratic ODEs of this type.} for each $x\in M$:
\begin{align}
\label{matrixeqn}
\dot{A} &= - A^2,\\ \label{matrixeqn2}
 A|_{t=0}&= A_0:= \nabla u_0.
\end{align}
 The solution of \eqref{matrixeqn}--\eqref{matrixeqn2} is readily seen to be
\be
A(t)= A_0 (I + t A_0)^{-1}.
\ee
If $N\in \mathcal{M}^{d\times d}$ is nilpotent then
$
(I + N)^{-1}= \sum_{m=0}^{d-1} (-N)^m$.
The claim follows.
\end{proof}

\begin{example}[example of Nilpotent data]\label{exampledata} Let $u_0: M\to \mathbb{R}^d$ be defined by
\begin{align*}\nonumber
u_0(x_1, x_2, \dots, x_d) &= \begin{pmatrix}u_1 (x_2, \dots, x_d) \\
u_2 (x_3, \dots, x_d)\\
\vdots\\
u_{d-2} (x_{d-1}, x_d)\\ 
u_{d-1} (x_d)\\ 
0
\end{pmatrix}, \\ 
\nabla u_0(x_1, x_2, \dots, x_d)  &= \begin{pmatrix}0 & \partial_2 u_1 & \partial_3 u_1  &\dots  &\dots & \partial_d u_1 \\
0 & 0& \partial_3 u_2  &\dots  &\dots  &\vdots\\
\vdots & \dots & \ddots & \ddots &\dots  &\vdots  \\
\vdots& \dots & \dots & \ddots & \partial_{d-1} u_{d-2}&  \partial_{d} u_{d-2} \\ 
\vdots &\dots & \dots & \dots& \ddots & \partial_{d} u_{d-1} \\
0 & \dots & \dots & \dots & \dots & 0
\end{pmatrix}.
\end{align*}
It is easy to check that $(\nabla u_0)^{d-1} $ has only one non-trivial component $\prod_{i=1}^{d-1} \partial_{i+1} u_i$ which is located on the first row and final column. Thus
whenever $ \prod_{i=1}^{d-1} \partial_{i+1} u_i \neq 0$, then $(\nabla u_0)^d=0$ but $(\nabla u_0)^n\neq 0$ for any $n<d$. 
\end{example}
\vspace{2mm}

We now discuss the possibility of singularity formation in ``infinite" spatial dimensions.  
Although there are no examples, it is expected that the higher the dimension, the easier it is for a perfect fluid to form a singularity in finite time.  In this direction, Khesin and Yang  constructed a finite-time blow-up for the binormal equation in 5D,
which is a localized induction approximation (LIA) of the 5D Euler \cite{KY}.  It is also hoped  \cite{ffo,ffr} that more complicated phenomena, such as turbulence, simplify in the large dimension limit due to the decreased influence of pressure (the reason being the incompressibility becomes one constraint amongst infinitely many components, and thus becomes trivially satisfied -- at least in some generic sense -- without need of pressure).  Here we show, in some sense, that infinite dimensional perfect fluids form a singularity in finite times from smooth initial conditions.

\begin{myshade}
\begin{theorem}\label{infdimbu}
The following two statements hold:
\begin{itemize}
\item 
For any smooth vector field $u_0^{(d)}:M\to \mathbb{R}^d$,    we have
\be\label{bound}
\|\nabla u^{(d)} (\cdot ,t)\|_{L^\infty} \leq \frac{ \|\nabla u_0^{(d)}\|_{L^\infty} }{1- t \|\nabla u_0^{(d)}\|_{L^\infty} },
\ee
where $\|A\|_{L^\infty} = \sup_{x\in M} \|A(x)\|$ and $\|A\|:= \max_{1\leq i \leq d} \sum_{j=1}^d  |a_{ij}|$ is the  induced $L^\infty$ norm (maximum absolute row sum).
\item There exists a family $\{ u_0^{(d)}\}_{d\in \mathbb{N}}$ of smooth vector fields $u_0^{(d)}:M\to \mathbb{R}^d$ whose gradients $A_0^{(d)} = \nabla u_0^{(d)}: M \to  \mathcal{M}^{d\times d}$ are nilpotent for every $x\in M$ such that  $\|\nabla u_0^{(d)} \|_{L^\infty} = 1$ and for all $t\in[0,1]$ we have
\be \label{infinitedimension}
\lim_{d\to\infty} \|\nabla u^{(d)} (\cdot ,t)\|_{L^\infty} =  \frac{1 }{1- t  }.
\ee
\end{itemize}
\end{theorem}
\end{myshade}

\begin{proof}
The bound \eqref{bound} follows from \eqref{matrixeqn}--\eqref{matrixeqn2} noting that by sub-multiplicativity of the $L^\infty$ norm, one has the differential inequality for $ A(t)= \nabla u^{(d)}( \Phi_t(x),t)$
\be
\frac{\rmd}{\rmd t} \| A\|_{L^\infty} \leq  \| A\|_{L^\infty}^2.
\ee
The lower bound \eqref{infinitedimension} follows from  bounding the sum \eqref{sum} for data of the form
\be
u_0(x_1, x_2, \dots, x_d) = \begin{pmatrix}u_1 (x_2) \\
u_2 (x_3)\\
\vdots\\
u_{d-2} (x_{d-1})\\ 
u_{d-1} (x_d)\\ 
0
\end{pmatrix},
\ee
with $u_1, \dots, u_{d-1}$ to be chosen.
This is a special case of the data of Remark \ref{exampledata}.
For such initial data, the matrix $(\nabla u_0)^n$ has non-zero entries only on the $n$th superdiagonal.  In particular, there are no cancellations when the powers are summed up to construct
the Lagrangian velocity gradient $A(t):=\nabla u (\Phi_t(x),t) $ via the formula \eqref{sum}.  Then $A(t)$ takes the explicit upper triangular form
\be\nonumber
\begin{pmatrix}0 & \partial_2 u_1 &-t \prod_{i=1}^2 \partial_{i+1} u_i   &t^2 \prod_{i=1}^3 \partial_{i+1} u_i   &\dots &(-t)^{d-2} \prod_{i=1}^{d-1} \partial_{i+1} u_i  \\
0 & 0& \partial_3 u_2  &\ddots  &\ddots  &(-t)^{d-3} \prod_{i=2}^{d-1} \partial_{i+1} u_i \\
\vdots & \dots & \ddots &\ddots & \ddots  &\vdots  \\
\vdots& \dots & \dots & \ddots & \partial_{d-1} u_{d-2}& -t \prod_{i=d-2}^{d-1} \partial_{i+1} u_i \\ 
\vdots &\dots & \dots & \dots& \ddots & \partial_{d} u_{d-1} \\
0 & \dots & \dots & \dots & \dots & 0
\end{pmatrix}.
\ee
We now choose initial data such that the absolute maxima of all the partials is unity and this is achieved at the point $x=0$, i.e. $\partial_{i+1} u_i(0) = -1$ while $|\partial_{i+1}u_i(x)|\leq 1$ for all $x\in M$ and all $1\leq i\leq d-1$ . It follows that $\|\nabla u_0\|_{L^\infty}=1$ and thus $  \|\nabla u (\cdot ,t)\|_{L^\infty} \leq  \frac{1}{1- t}$.  On the other hand, we note 
\be\nonumber
\|\nabla u (0 ,t)\|_{L^\infty} = \sum_{i=0}^{d-2} t ^i =  \frac{1- t^{d-1}}{1- t}.
\ee
This concludes the proof.
\end{proof}

We remark that this blowup is generic within the class of nilpotent initial conditions in the sense that perturbations also exhibit finite time singularities at slightly modified times. One might hope that in  large but finite spatial dimensions, $1/d$ could serve as a small parameter with which to perform a controlled expansion.

  \begin{myshade3}
  \vspace{-1mm}
\begin{question}\label{questinf}
Is it possible to construct a blowup for smooth solutions in  large but finite dimension $d$ perturbatively from a suitable infinite dimensional blowup?
\end{question}
\end{myshade3}

We note also that, for axisymmetric no swirl solutions in dimensions larger than three, it remains unclear whether global existence and uniqueness holds for smooth solutions.

  \begin{myshade3}
  \vspace{-1mm}
\begin{question}\label{questinf2}
Can singularities form from smooth data for the axisymmetric no swirl Euler equations on $\mathbb{R}^d$ when $d\geq 4$?
\end{question}
\end{myshade3}

\section{Acknowledgements}

 First and foremost, we would like to thank D. Sullivan for generously sharing with us his unique vision of fluid dynamics and Mathematics at large. {This review is dedicated to Dennis with admiration and affection.}  We would like to thank A. Shnirelman for many inspiring remarks and insights. We are grateful to M.  Dolce and K.Modin for illuminating discussions and help creating some of the figures.  We thank C. Lalescu for providing the visualization in Figure \ref{fig1tarek}.
 We thank V.Vicol  for a careful reading of the manuscript and insightful remarks which greatly improved the paper. We thank also D. Ginsberg, B. Khesin, G. Misiolek, K. Modin, R. Shvydkoy V. Šverák,  F. Torres de Lizaur, and Y. Yao for very  useful comments and discussions on the manuscript. The research of T.D. was partially supported by NSF-DMS grant 2106233 and the Charles Simonyi Endowment at the Institute for Advanced Study.
 The research of T.E. was partially supported by the NSF grants DMS-2043024 and DMS-2124748  as well as the Alfred P. Sloan Foundation.

\section{Collection of problems, questions and conjectures}
We collect, for the reader's convenience, the problems, questions and conjectures stated throughout the paper (with perceived difficulty in that order).
We remark that a good number of the problems and conjectures below are either explicit, or are simplifications of, statements by other authors.
For the others, we have followed Arnold's principle: \emph{``There is a general principle that a stupid man can ask such questions to which one hundred wise men would not be able to answer. In accordance with this principle I shall formulate some problems."}

\vspace{4mm}

\noindent \textbf{Section I.}\phantom{Adsf}
\vspace{2mm}

\noindent \textbf{Question \ref{quest1}.} (Shnirelman (1985), \cite{shnir1}) Let $M\subset \mathbb{R}^2$ be a domain with smooth boundary.   Does there exist a perfect fluid flow connecting any two isotopic states $\gamma_1,\gamma_2\in \mathscr{D}_\mu(M)$?
\vspace{4mm}

\noindent \textbf{Section II.}\phantom{Adsf}
\vspace{2mm}

\noindent \textbf{Problem \ref{conj1}.} 
Show that there is no steady 2D Euler solution $\omega\in C^{1}$ that is isolated in $C^{1}$ from other stationary solutions (modulo symmetries).
\vspace{2mm}

\noindent \textbf{Question \ref{yudoquest}.} (Yudovich (2003), \cite{Y03}) 
Does there exist a continuous $L^2$ stable steady state $\omega_*$  that is unstable in $L^\infty$?
\vspace{2mm}

\noindent \textbf{Question \ref{queststab}.} Let $(M,g)$ be a compact two-dimensional Riemannian manifold without boundary (e.g. $\mathbb{S}^2$ with the round metric or $\mathbb{T}^2$ with the flat metric).  Do there exist any non-trivial stable steady Euler solutions? \vspace{2mm}

\noindent \textbf{Problem \ref{wanderq}.} 
Show that there exist wandering neighborhoods in the phase space $L^\infty$ of vorticity on compact domains without boundary (e.g. $M=\mathbb{T}^2$ or $\mathbb{S}^2$).
\vspace{2mm}

\noindent \textbf{Problem \ref{kochquest}.} 
Let $X$ be any space compactly embedded in $L^2$. Show that  any smooth stationary solution $\omega_*$ of the two dimensional
Euler equation whose Lagrangian flow is not periodic in time is nonlinearly unstable in $X$.   
\vspace{2mm}

\noindent \textbf{Question \ref{isochronalconju}.} 
For each simply connected domain $M\subset \mathbb{R}^2$, is there at most one (modulo scaling) smooth isochronal flow?
\vspace{2mm}

\noindent \textbf{Conjecture \ref{instconj}.} (Yudovich (1974), \cite{Y74})
No steady solutions of the Euler equations are stable in $C^{1,\alpha}$.  Moreover, for any steady state $\omega_*\in C^{\alpha}$  and any $\ve>0$, there is a $\|\omega_0-\omega_*\|_{C^{\alpha}} \leq \ve$ such that  $\|\omega(t)\|_{C^{\alpha}} \to \infty$ as $t\to \infty$.
\vspace{2mm}

\noindent \textbf{Conjecture \ref{conjgrow2}.} 
For each $\omega_*\in C^{\alpha}$  and any $\ve>0$, there is a $\|\omega_0-\omega_*\|_{C^{\alpha}} \leq \ve$ such that  $\|\omega(t)\|_{C^{\alpha}} \to \infty$ as $t\to \infty$.
\vspace{2mm}

\noindent \textbf{Conjecture \ref{svconj}.} (Šverák (2011), \cite{sv11})
Generic initial data $\omega_0\in L^\infty( M)$ gives rise to inviscid incompressible motions whose vorticity orbits $\{ \omega(t)\}_{t\in \mathbb{R}}$ are not precompact in $L^2( M)$.
\vspace{2mm}

\noindent \textbf{Conjecture \ref{shnconj}.} (Shnirelman (2013), \cite{shn13})
For any initial data  $\omega_0\in  L^\infty( M)$, the collection of $L^2( M)$ weak limits of the orbit $\{ \omega(t)\}_{t\in \mathbb{R}}$ consists of  vorticities which generate $L^2( M)$ precompact orbits under 2D Euler evolution.
\vspace{2mm}

\noindent \textbf{Problem \ref{omlimprob}.} 
Show that $\Omega_+(X_*)\neq X_*$.
\vspace{4mm}

\noindent \textbf{Section III.}\phantom{Adsf}
\vspace{2mm}

\noindent \textbf{Conjecture \ref{bkmconj}.} 
There exists a universal $p_*<\infty$ so that if a smooth solution to the Euler equation $\omega\in C^\infty_c(\mathbb{R}^3\times [0,T_*))$ becomes singular as $t\rightarrow T_*$, then 
\[\sup_{t\in[0, T_*)} \|\omega\|_{L^{p_{*}}}=+\infty.\]

\noindent \textbf{Conjecture \ref{1dModelConjecture}.} 
Consider the 1d vorticity equation on $\mathbb{S}^1$:
\[\partial_t\omega+u\partial_x\omega=\omega\partial_x u,\qquad u=\mathcal{K}(\omega),\] where $\mathcal{K}$ is a non-trivial Fourier integral operator with bounded symbol $m$ satisfying $|m(k)|\leq \frac{C}{1+|k|}$. 
\begin{itemize}
\item In general, $H^s$ solutions should develop singularities when $s<3/2$. 
\item If $m$ is odd, there exist analytic solutions that become singular in finite time. 
\item If $m$ is even, $H^s$ solutions must be global whenever $s>\frac{3}{2}$. 
\end{itemize}
\vspace{2mm}

\noindent \textbf{Problem \ref{inftime3d}.} 
Establish generic infinite-time singularity formation and convergence to compact orbits for smooth solutions to \eqref{1dEuler} in the infinite-time limit. 
\vspace{2mm}

\noindent \textbf{Problem \ref{scaleinvprob}.} 
Does there exist a scale invariant solution with $\nabla u\in \ring{C}^\infty$ belonging to an admissible symmetry class that develops a singularity in finite time? 
\vspace{2mm}

\noindent \textbf{Problem \ref{compsup}.} 
Determine whether a smooth compactly supported solution to \eqref{2dFM} on $\mathbb{R}^2$ can have $\nabla\omega$ grow double exponentially. 
\vspace{4mm}

\noindent \textbf{Section IV.}\phantom{Adsf}
\vspace{2mm}

\noindent \textbf{Question \ref{Sullivan}.} (D. Sullivan)
Do there exist blow-up solutions with non-vanishing vorticity?
\vspace{2mm}

\noindent \textbf{Problem \ref{bousprob}.} 
Prove that there exist scale-invariant blow-up solutions to the Boussinesq system \eqref{B1}-\eqref{B3} on the corner domain $\Omega=\{(r,\theta): -\pi/4<\theta<\pi/4\}$ with vanishing vorticity and density gradient on $\partial\Omega$. 
\vspace{2mm}

\noindent \textbf{Conjecture \ref{HLConj}.} 
There exists a smooth solution to the axi-symmetric Euler equation on the cylindrical domain $\{(r,z): 0\leq r\leq 1\}$ that becomes singular in finite time. 
\vspace{2mm}

\noindent \textbf{Problem \ref{axiblowupfund}.} 
Establish blow-up for axisymmetric no-swirl solutions in the vicinity of a large class of blow-up profiles to the fundamental model. 
\vspace{2mm}

\noindent \textbf{Problem \ref{instabquest}.} 
Study the stability/instability of the self-similar profiles of Theorem \ref{EClassical} with respect to general 3d perturbations. 
\vspace{2mm}

\noindent \textbf{Problem \ref{selfsimcoer}.} 
Establish the existence of self-similar profiles relying only on the invertibility of the linear operator rather than its coercivity. 
\vspace{2mm}

\noindent \textbf{Problem \ref{curvequest}.} 
Study the behavior of the curve of self-similar solutions in $\alpha$ constructed in \cite{E_Classical}. Does the curve continue for all $\alpha\leq 2$ or does it stop at some $\alpha_*<2$? Establish continuation criteria for the curve of solutions.
\vspace{2mm}

\noindent \textbf{Conjecture \ref{alpha13}.} 
For any $\alpha<\frac{1}{3}$, there exists a $C^\alpha_c$ axi-symmetric no-swirl local Euler solution that becomes singular in finite time. 
\vspace{2mm}

\noindent \textbf{Conjecture \ref{radsmooth}.} 
In the preceding conjecture, $\omega_0$ can be $C^\infty$ in $\rho=\sqrt{x_1^2+x_2^2+x_3^2}$. 
\vspace{-1mm}

\noindent \textbf{Conjecture \ref{artbound}.} 
There exist smooth solutions to the axi-symmetric Euler equation non-vanishing on the axis of symmetry that become singular in finite time. 
\vspace{-1mm}

\noindent \textbf{Conjecture \ref{propring}.} 
The singular set of the solutions of \cite{E_Classical} propagates within the natural symmetry class as an expanding ring $\mathcal{S}(t)=\{(r,x_3)=(\delta(t),0)\}$, where $\delta(T_*)=0$ while $\delta$ is increasing in $t$ for all $t\geq T_*$. 
\vspace{4mm}

\noindent \textbf{Section V.}\phantom{Adsf}
\vspace{2mm}

\noindent \textbf{Question \ref{questinf}.} 
Is it possible to construct a blowup for smooth solutions in  large but finite dimension $d$ perturbatively from a suitable infinite dimensional blowup?
\vspace{2mm}

\noindent \textbf{Question \ref{questinf2}.} 
Can singularities form from smooth data for the axisymmetric no swirl Euler equations on $\mathbb{R}^d$ when $d\geq 4$?


\end{document}